\documentclass[10pt]{amsart}

\usepackage{anysize}
\usepackage[utf8x]{inputenc}
\usepackage[english]{babel}
\usepackage{amssymb, amsmath, amsthm}
\usepackage{mathabx}
\usepackage{enumitem}
\usepackage{mathtools}
\usepackage{tikz}
\usepackage{tikz-cd}
\usetikzlibrary{arrows}
\usepackage{graphicx}
\usepackage{url}
\usepackage{color}
\usepackage[all,cmtip]{xy}
\usepackage[hidelinks,bookmarksnumbered]{hyperref}

\let\cal\mathcal
\def\AA{{\cal A}}
\def\BB{{\cal B}}
\def\CC{{\cal C}}
\def\DD{{\cal D}}
\def\EE{{\cal E}}
\def\FF{{\cal F}}

\def\II{{\cal I}}

\def\NN{{\cal N}}

\def\SS{{\cal S}}
\def\TT{{\cal T}}
\def\UU{{\cal U}}

\def\WW{{\cal W}}

\let\blb\mathbb

\def\bK{{\blb K}}

\def\bZ{{\blb Z}}

\def\bN{{\blb N}}

\def\bZ{{\blb Z}}

\newcommand{\Set}{\mathrm{Set}}
\newcommand{\boldZ}{\mathbf{Z}}

\newcommand{\inflation}{\rightarrowtail}
\newcommand{\deflation}{\twoheadrightarrow}

\author{Ruben Henrard}
\address{Ruben Henrard \\ Universiteit Hasselt \\ Campus Diepenbeek \\ Departement WNI \\ 3590 Diepenbeek \\ Belgium}
\email{ruben.henrard@uhasselt.be}
\author{Adam-Christiaan van Roosmalen}
\address{Adam-Christiaan van Roosmalen \\ Universiteit Hasselt \\ Campus Diepenbeek \\ Departement WNI \\ 3590 Diepenbeek \\ Belgium}
\email{adamchristiaan.vanroosmalen@uhasselt.be}
\title{Derived categories of (one-sided) exact categories and their localizations}

\newtheorem{theorem}{Theorem}[section]
\newtheorem{proposition}[theorem]{Proposition}
\newtheorem{lemma}[theorem]{Lemma}
\newtheorem{corollary}[theorem]{Corollary}

\theoremstyle{definition}
\newtheorem{definition}[theorem]{Definition}
\newtheorem{remark}[theorem]{Remark}
\newtheorem{example}[theorem]{Example}
\newtheorem{construction}[theorem]{Construction}

\DeclareMathOperator{\coker}{coker}
\DeclareMathOperator{\Ext}{Ext}

\DeclareMathOperator{\Hom}{Hom}
\DeclareMathOperator{\Adm}{Adm}
\DeclareMathOperator{\Fun}{Fun}
\DeclareMathOperator{\Mor}{Mor}

\DeclareMathOperator{\Mod}{Mod}
\DeclareMathOperator{\modd}{mod}

\DeclareMathOperator{\Ab}{\mathsf{Ab}}
\DeclareMathOperator{\LCA}{\mathsf{LCA}}
\DeclareMathOperator{\LC}{\mathsf{LC}}
\DeclareMathOperator{\Ob}{Ob}

\DeclareMathOperator{\cone}{cone}
\DeclareMathOperator{\Tot}{Tot}
\DeclareMathOperator{\Proj}{Proj}
\DeclareMathOperator{\Inj}{Inj}
\DeclareMathOperator{\can}{can}
\DeclareMathOperator{\real}{real}

\DeclareMathOperator{\Homology}{\mathbf{H}}
\DeclareMathOperator{\C}{\mathbf{C}}
\DeclareMathOperator{\D}{\mathbf{D}}
\DeclareMathOperator{\DAb}{\mathbf{D}_{\mathcal{A}}^b}
\DeclareMathOperator{\CAb}{\mathbf{C}_{\mathcal{A}}^b}
\DeclareMathOperator{\Cnm}{\mathbf{C}^{[n,m]}}
\DeclareMathOperator{\Conen}{\mathbf{C}^{[1,n]}}
\DeclareMathOperator{\Conenmone}{\mathbf{C}^{[1,n-1]}}
\DeclareMathOperator{\Seq}{Seq}
\DeclareMathOperator{\Seqb}{Seq^b}
\DeclareMathOperator{\K}{\mathbf{K}}
\DeclareMathOperator{\Ac}{\mathbf{Ac}}
\DeclareMathOperator{\Acb}{\mathbf{Ac}^b}
\DeclareMathOperator{\AcC}{\mathbf{Ac}_{\mathbf{C}}}
\DeclareMathOperator{\AcbC}{\mathbf{Ac}^b_{\mathbf{C}}}

\DeclareMathOperator{\AcK}{\mathbf{Ac}_{\mathbf{K}}}

\DeclareMathOperator{\Cm}{\mathbf{C}^-}
\DeclareMathOperator{\Km}{\mathbf{K}^-}
\DeclareMathOperator{\Dm}{\mathbf{D}^-}
\DeclareMathOperator{\Cb}{\mathbf{C}^b}
\DeclareMathOperator{\Db}{\mathbf{D}^b}
\DeclareMathOperator{\Kb}{\mathbf{K}^b}
\DeclareMathOperator{\Kab}{\mathbf{K}^{-,b}}

\DeclareMathOperator{\Cdg}{\mathbf{C}_{dg}}
\DeclareMathOperator{\Cbdg}{\mathbf{C}^b_{dg}}
\DeclareMathOperator{\Dbdg}{\mathbf{D}^b_{dg}}

\DeclareMathOperator{\Acbdg}{\mathbf{Ac}^b_{dg}}

\DeclareMathOperator{\Cbinf}{\mathbf{C}^b_{\infty}}
\DeclareMathOperator{\Dbinf}{\mathbf{D}^b_{\infty}}

\DeclareMathOperator{\Acbinf}{\mathbf{Ac}^b_\infty}

\DeclareMathOperator{\Ndg}{N_{dg}}
\DeclareMathOperator{\Ho}{Ho}
\DeclareMathOperator{\Idem}{Idem}

\newcommand{\StabInfCat}{\mathrm{Cat}_\infty^{\mathrm{Ex}}}

\newcommand{\dq}{{/\mkern-6mu/}}

\DeclareMathOperator{\add}{add}

\makeatletter
\newcommand{\myitem}[1]{%
\item[#1]\protected@edef\@currentlabel{#1}%
}
\makeatother


\begin{document}

\begin{abstract}
We consider the quotient of an exact or one-sided exact category $\EE$ by a so-called percolating subcategory $\AA$.  For exact categories, such a quotient is constructed in two steps.  Firstly, one localizes $\EE$ at a suitable class $S_\AA \subseteq \Mor(\EE)$ of morphisms.  The localization $\EE[S_\AA^{-1}]$ need not be an exact category, but will be a one-sided exact category.  Secondly, one can construct the exact hull $\EE\dq \AA$ of $\EE[S_\AA^{-1}]$ and show that this satisfies the 2-universal property of a quotient amongst exact categories.

In this paper, we show that this quotient $\EE \to \EE \dq \AA$ induces a Verdier localization $\Db(\EE) \to \Db(\EE \dq \AA)$ of bounded derived categories.  Specifically, (i) we study the derived category of a one-sided exact category, (ii) we show that the localization $\EE \to \EE[S_\AA^{-1}]$ induces a Verdier quotient $\Db(\EE) \to \Db(\EE[S^{-1}_\AA])$, and (iii) we show that the natural embedding of a one-sided exact category $\FF$ into its exact hull $\overline{\FF}$ lifts to a derived equivalence $\Db(\FF) \to \Db(\overline{\FF})$.

We furthermore show that the Verdier localization is compatible with several enhancements of the bounded derived category, so that the above Verdier localization can be used in the study of localizing invariants, such as non-connective $K$-theory.
\end{abstract}

\maketitle
\tableofcontents

\section{Introduction}

Abelian categories and their derived categories provide a framework for homological algebra. There are various meaningful, more general settings in which one can still use some homological machinery. For example, the category of vector bundles on a variety and the category of filtered objects of an abelian category are not abelian but they are exact categories in the sense of Quillen (see \cite{Quillen73}).

Quotients of abelian categories are well understood (see \cite{Gabriel62}).  Let $\AA$ be a Serre subcategory of an abelian category $\EE$.  There is an exact quotient functor $Q\colon \EE \to \EE/\AA$ which is universal among all exact functors $F\colon \EE \to \FF$ with $F(\AA) = 0$ (here, $\FF$ is an abelian category).  Such quotients occur often as localizations in algebra and geometry.  Similar statements have been investigated for exact categories; see for example \cite{Cardenas98,Schlichting04}.

In \cite{HenrardvanRoosmalen19}, the authors introduced a (left or right) percolating subcategory $\AA$ of an exact category $\EE$ and showed that there is an exact quotient functor $Q\colon \EE \to \EE \dq \AA$ in the 2-category of exact categories.  In contrast to the aforementioned quotients, the functor $Q$ here is not given by a localization at a class of morphisms.  Instead, the quotient in constructed in two steps.  In the first step, one introduces a class of morphisms $S_\AA$ of $\EE$ and considers the localization $\EE \to \EE[S_\AA^{-1}]$.  The localization $\EE[S_\AA^{-1}]$ need not be a Quillen exact category, but one can show that it is a one-sided exact category (in the sense of \cite{BazzoniCrivei13, Rump11}, see theorem \ref{theorem:MainTheoremPartI}).  The next step consists of taking the exact hull of the category $\EE[S_\AA^{-1}],$ as described in \cite{Rosenberg11} or section \ref{section:ExactHull} in this paper; this exact hull is the quotient category $\EE \dq \AA.$

In this paper, we show that the localization sequence $\AA\to \EE \to \EE \dq \AA$ induces a Verdier localization sequence on the level of the bounded derived categories.  The following statement combines theorems \ref{Theorem:ExactHullIntroduction} and \ref{Theorem:MainTheoremIntroduction} below (see corollary \ref{Corollary:MainTheoremExact} in the text).

\begin{theorem}\label{theorem:IntroductionQuotientsForExact}
Let $\AA$ be a (left or right) percolating subcategory of an exact category $\EE.$  There is an induced Verdier localization sequence 
\[\DAb(\EE) \to \Db(\EE) \stackrel{Q}{\rightarrow} \Db(\EE \dq \AA).\]
\end{theorem}

Here, $\DAb(\EE)$ is the thick subcategory of $\Db(\EE)$ generated by all complexes with entries in $\AA$.  To establish this result, we consider three main steps: we study the derived category of the localization $\EE[S^{-1}_\AA]$ (or more generally, of a one-sided exact category, see section \ref{subsection:DerivedCategory}), we study the derived localization functor $\Db(\EE)\to \Db(\EE[S^{-1}_\AA])$ (see section \ref{subsection:IntroductionLocalizing}), and we study the exact hull of a one-sided exact category (see section \ref{subsection:IntroductionExactHull}). 

We now briefly discuss each of these main parts.

\subsection{Deriving one-sided exact categories}\label{subsection:DerivedCategory}

Quillen exact categories are additive categories together with a distinguished class of kernel-cokernel pairs called \emph{conflations}. The kernel morphism of a conflation is called an \emph{inflation} and the cokernel morphism is called a \emph{deflation}. The axioms of exact categories can be partitioned into two dual sets of axioms: one set only referring to the inflation-side and one set only referring to the deflation-side. Requiring only one of these sets leads to the notion of \emph{one-sided exact categories} (see \cite{BazzoniCrivei13,Rump11}).  Similar one-sided exact structures occur in various places in the literature.  Indeed, these are additive versions of left/right exact categories \cite{Rosenberg11} and homological categories \cite{BorceuxBourn04}, and are closely related to Waldhausen categories and categories with fibrations \cite{Weibel13}).

Recall that for an abelian category $\AA$, the derived category can be constructed from the category of complexes $\C(\AA)$ by inverting all \emph{quasi-isomorphisms} (i.e., all morphisms between complexes which induce isomorphisms on homology), thus $\D(\AA) = \C(\AA)[S^{-1}]$ where $S$ is the set of all quasi-isomorphisms.  One can simplify the construction by factoring the localization as $\C(\AA) \to \K(\AA) \to \C(\AA)[S^{-1}]$, where $\K(\AA)$ is the homotopy category, and the functor $\K(\AA) \to \C(\AA)[S^{-1}]$ is a Verdier localization (see \cite{Verdier96} or \cite[\S1.2]{Keller96}).  Specifically, $\K(\AA) / \Ac(\AA) \simeq \C(\AA)[S^{-1}] = \D(\AA)$ where $\Ac(\AA)$ is the category of all acyclic complexes (i.e.~complexes with zero homology).

In the setting of exact categories, even though the notion of homologies is not readily available, the notion of an acyclic complex is easily expressible within the framework: a complex is \emph{acyclic} if it can be obtained by splicing conflations (or, equivalently, all differentials are admissible maps and the image of one differential is the kernel of the ensuing one).  In \cite[lemma~1.1]{Neeman90}, Neeman shows that the subcategory $\Ac(\EE)\subseteq \K(\EE)$ of acyclic complexes (up to homotopy) is a triangulated subcategory of $\K(\EE)$.  The derived category $\D(\EE)$ is then defined as the Verdier localization $\K(\EE)/\Ac(\EE)$, or equivalently, $\K(\EE)/\langle\Ac(\EE)\rangle_{\textrm{thick}}$ where $\langle\Ac(\EE)\rangle_{\textrm{thick}}$ is the thick closure of $\Ac(\EE)$ in $\K(\EE)$ (see \cite{Keller96, Neeman90}).  When $\EE$ is an abelian category, this construction recovers the usual derived category.

As a one-sided exact category comes equipped with a set of conflations (and hence with acyclic complexes), one can construct the derived category in an analogous way.  This was done in \cite{BazzoniCrivei13}.  Explicitly, Bazzoni and Crivei show that the category of acyclic complexes $\Ac(\EE)$ is a triangulated subcategory of the homotopy category $\K(\EE)$ and define the derived category as the Verdier localization $\K(\EE)/\langle\Ac(\EE)\rangle_{\textrm{thick}}$.

In section \ref{section:DerivedCategoriesOfRightExactCategories}, we delve deeper into the construction of the derived category and establish some basic properties.  We first remark that, if $\EE$ is a deflation-exact category not satisfying axiom \ref{R0*}, then the derived category misses some key desirable properties.  As an example, even for a conflation $X \inflation Y \deflation Z$ in $\EE$, the corresponding chain map
\[\xymatrix{
\cdots \ar[r] & 0 \ar[d] \ar[r] & X \ar@{>->}[r]\ar[d] & Y \ar[r]\ar@{->>}[d] & 0 \ar[d] \ar[r] &\cdots\\ 
\cdots \ar[r] & 0 \ar[r] & 0 \ar[r] & Z \ar[r] & 0 \ar[r] &\cdots
}\]
need not be a quasi-isomorphism (for this map to be a quasi-isomorphism, we additionally need that $Z \to 0$ is a deflation; see remark \ref{remark:StandardContractible} for more information).  If $\EE$ is a deflation-exact category with \ref{R0*}, then this morphism is indeed a quasi-isomorphism (see proposition \ref{proposition:ConflationsYieldTriangles}).  Note that an exact category automatically satisfies \ref{R0*} and hence this situation does not occur there.

We give some basic properties of the derived category of a one-sided exact category.  The following theorem summarizes some of the main results of section \ref{section:DerivedCategoriesOfRightExactCategories} (see \cite{Positselski11} for the case of an exact category).

\begin{theorem}\label{theorem:SummarazationTheorem}
Let $\EE$ be a deflation-exact category.
\begin{enumerate}
	\item The natural embedding $i\colon \EE \to \D(\EE)$ is fully faithful.
	\item For all $X,Y \in \EE$ and $n > 0$, we have $\Hom_{\D(\EE)}(\Sigma^n i(X), i(Y)) = 0$.
\end{enumerate}
If $\EE$ satisfies {\ref{R0*}}, then
\begin{enumerate}[resume]
	\item a conflation $X \inflation Y \deflation Z$ in $\EE$ lifts to a triangle $i(X) \to i(Y) \to i(Z) \to \Sigma i(X)$ in $\D(\EE)$.
	\item if $\EE$ has enough projectives, the natural functor $\Kab(\Proj(\EE)) \to \Db(\EE)$ is a triangle equivalence.
\end{enumerate}
\end{theorem}

We remark that, even when the natural embedding $i\colon \EE \to \D(\EE)$ maps conflations to triangles, the essential image of $i$ need not be extension-closed in $\D(\EE).$  Moreover, when $\EE$ is weakly idempotent complete and satisfies \ref{R3}, the essential image is extension-closed if and only if $\EE$ is two-sided exact (see proposition \ref{proposition:ExactMeansExtensionClosed}).

\subsection{The exact hull}\label{subsection:IntroductionExactHull}

In section \ref{section:ExactHull}, we use the derived category to construct the \emph{exact hull} $\overline{\EE}$ of a deflation-exact category $\EE$ with \ref{R0*} as the extension closure of $i(\EE)$ in $\D(\EE)$.  Based on \cite{Dyer05}, we endow $\overline{\EE}$ with the following exact structure: a sequence $X \to Y \to Z$ is a conflation in $\overline{\EE} \subset \Db(\EE)$ if there is a triangle $i(X) \to i(Y) \to i(Z) \to \Sigma i(X)$ in $\Db(\EE)$.  The following theorem is the main result of section \ref{section:ExactHull}.

\begin{theorem}\label{Theorem:ExactHullIntroduction}
Let $\EE$ be a deflation-exact category satisfying axiom \ref{R0*}. The embedding $j\colon \EE\to \overline{\EE}$ is an exact embedding which is $2$-universal among exact functors to exact categories.  Moreover,
\begin{enumerate}
  \item the embedding $\overline{\EE} \to \Db(\EE)$ lifts to a triangle equivalence $\Db(\overline{\EE}) \simeq \Db(\EE)$, and
	\item the embedding $\EE \to \overline{\EE}$ lifts to a triangle equivalence $\Db({\EE}) \simeq \Db(\overline{\EE})$
\end{enumerate}
\end{theorem}

Note that the universal property in theorem \ref{Theorem:ExactHullIntroduction} recovers a result from \cite{Rosenberg11}. The two triangle equivalences indicate that taking the exact hull is homologically and $K$-theoretically a gentle operation.  However, as example \ref{Example:IsbellCategory} illustrates, in going to the exact hull, one might fail to preserve some attractive properties, such as being pre-abelian. 

\subsection{Localizing with respect to percolating subcategories}\label{subsection:IntroductionLocalizing}

In \cite{HenrardvanRoosmalen19}, we introduce a percolating subcategory $\AA$ of a one-sided exact category $\EE$ and construct the quotient $\EE / \AA$ in the 2-category of one-sided exact categories.  One of the main results of this paper (theorem \ref{theorem:MainTheorem} in the text) states that the quotient functor $Q\colon \EE \to \EE / \AA$ induces a Verdier localization functor $\Db(\EE) \to \Db(\EE/\AA)$.  This generalizes \cite[theorem 3.2]{Miyachi91} and \cite[proposition 2.6]{Schlichting04} (our proof follows these references closely).

\begin{theorem}\label{Theorem:MainTheoremIntroduction}
	Let $\EE$ be a deflation-exact category and let $\AA$ be a deflation-percolating subcategory.
	\begin{enumerate}
		\item	The derived quotient functor $\Db(\EE)\to \Db(\EE/\AA)$ is a Verdier localization.
		\item If $\EE$ satisfies axiom \ref{R0*}, then the following is a Verdier localization sequence:
		\[\DAb(\EE)\to \Db(\EE)\to \Db(\EE/\AA).\]
	\end{enumerate}
\end{theorem}
Here, $\DAb(\EE)$ is the thick subcategory of $\Db(\EE)$ generated by $\AA$ (when $\EE$ is an abelian category, $\DAb(\EE)$ is usually defined as the full subcategory of $\Db(\EE)$ consisting of those objects with homologies in $\AA$).  In general, there is a natural triangle functor $\Db(\AA) \to \DAb(\EE)$, mapping a complex to itself, but it need not be an equivalence.

\subsection{Application to locally compact abelian groups}

As an application, we consider localizations of the exact category $\LCA$ of locally compact (Hausdorff) abelian groups. The subcategory $\LCA_D$ of discrete groups is a deflation-percolating subcategory and the quotient $\LCA / \LCA_D$ satisfies the conditions of theorem \ref{Theorem:MainTheoremIntroduction}. Moreover, we show in corollary \ref{corollary:VerdierSequencesLocallyCompactModules} that the Verdier localization sequence from theorem \ref{Theorem:MainTheoremIntroduction} is equivalent to the sequence
\[\Db(\LCA_D)\rightarrow \Db(\LCA) \rightarrow \Db(\LCA / \LCA_D).\]
There is a similar sequence concerning the category $\LCA_C$ of compact abelian groups.

It follows from \cite{HenrardvanRoosmalen19} that the quotient $\LCA/\LCA_D$ is a deflation-exact category, but it is not immediately clear whether it is two-sided exact.  To this end, we establish the following criterion (theorem \ref{theorem:ExactLocalization} in the text).

\begin{theorem}\label{Theorem:ExactLocalizationTheoremIntroduction}
	Let $\EE$ be a weakly idempotent complete deflation-exact category and let $\AA\subseteq \EE$ be a deflation-percolating subcategory. If 
	$\Hom_{\Db(\EE)}(i(E),\Sigma^2 i(A))=0$
	for all $E\in \EE$ and $A\in \AA$, then the quotient $\EE/\AA$ is an exact category.
\end{theorem}

In \cite{HoffmannSpitzweck07}, Hoffmann and Spitzweck investigated the homological properties of $\LCA$ and showed that $\Hom_{\Db(\LCA)}(\LCA,\Sigma^2\LCA)=0$.  This allow us to conclude the following statement.

\begin{corollary}
For any inflation- or deflation-percolating subcategory $\AA$ of $\LCA$, the quotient category $\LCA / \AA$ is an exact category.
\end{corollary}

\subsection{Enhancements of the derived category}

Verdier localization sequences such as the one in theorem \ref{Theorem:MainTheoremIntroduction} are abundant in $K$-theory.  In section \ref{section:enhancements}, we consider three enhancements of the bounded derived category of a one-sided exact category $\EE$ (using Frobenius pairs, dg categories, and $\infty$-categories) and formulate theorems \ref{Theorem:ExactHullIntroduction} and \ref{Theorem:MainTheoremIntroduction} in the corresponding enhanced settings.

These enhancements allow us to define (non-connective) $K$-theory, as well as other localizing invariants (in the sense of \cite{BlumbergGepnerTabuada13}), for one-sided exact categories.  Specifically, using the results of this paper, we find the following theorem.

\begin{theorem}\label{theorem:KTheoryIntroduction}
Let $\EE$ be a deflation-exact category.
\begin{enumerate}
	\item The embedding $\EE \to \overline{\EE}$ of $\EE$ into its exact hull, yields an equivalence $\bK(\EE) \to \bK(\overline{\EE}).$
	\item For a deflation-percolating subcategory $\AA \subseteq \EE$ such that the natural map $\Db(\AA) \to \DAb(\EE)$ is an equivalence, there is a homotopy fibration sequence
	\[\mathbb{K}(\AA)\to \mathbb{K}(\EE)\to \mathbb{K}(\EE/\AA).\]
\end{enumerate}
Here, $\bK$ stands for the non-connective $K$-theory spectrum.
\end{theorem}

We will obtain this theorem using the $\infty$-categorical model, following \cite{BlumbergGepnerTabuada13}.

\subsection{The weak idempotent completion}  As is remarked in \cite{Neeman90}, the embedding of an exact category $\EE \to \widehat{\EE}$ into its weak idempotent completion (there called the semi-saturation) lifts to a derived equivalence $\Db(\EE) \simeq \Db(\widehat{\EE})$.  A similar observation holds for one-sided exact categories; this will be examined in appendix \ref{appendix:WeakIdempotentCompletion}.  The situation is slightly more subtle in the one-sided exact setting, as is indicated by the presence of the axiom \ref{R3} in the following statement (combining proposition \ref{proposition:WeakIdempotentCompletionRightExact} and theorem \ref{theorem:WeakidempotentCompleteTriangleEquivalence}).

\begin{theorem}\label{theorem:WeakidempotentCompleteTriangleEquivalenceIntroduction}
Let $\EE$ be a deflation-exact category satisfying axiom {\ref{R0*}}.
\begin{enumerate}
	\item The weak idempotent completion $\widehat{\EE}$ has a canonical deflation-exact structure satisfying axiom {\ref{R3}} such that the embedding $j_{\EE}\colon \EE \to \widehat{\EE}$ is $2$-universal among exact functors to weakly idempotent complete deflation-exact categories satisfying axiom {\ref{R3}}.
	\item The natural embedding $j_\EE\colon \EE \to \widehat{\EE}$ lifts to a triangle equivalence $\Db(\EE) \simeq \Db(\widehat{\EE})$.
\end{enumerate}
\end{theorem}

We note that, even though the bounded derived categories of $\EE$ and its weak idempotent completion $\widehat{\EE}$ are equivalent, a similar statement does not hold for its idempotent completion (see \cite{BalmerSchlichting01}).

\subsection*{Acknowledgments}  We are grateful to Frederik Caenepeel and Freddy Van Oystaeyen for motivating discussions, and to Francesco Genovese and Ivo Dell'Ambrogio for helpful comments on an earlier draft.  We thank Oliver Braunling for suggesting section \ref{section:enhancements} as well as for many helpful comments.  The second author is currently a postdoctoral researcher at FWO (12.M33.16N).
\section{Preliminaries}

Throughout the paper all categories are assumed to be small.

\subsection{One-sided exact categories}

We recall the definition of an inflation-exact and a deflation-exact category from \cite{BazzoniCrivei13}.

\begin{definition}\label{definition:ConflationCategory}
	Let $\CC$ be an additive category.  A \emph{kernel-cokernel pair} in $\CC$ is a pair of composable morphisms $A\xrightarrow{f} B\xrightarrow{g} C$ with $f = \ker g$ and $g = \ker f$.	A \emph{conflation category} $\CC$ is an additive category with a chosen class of kernel-cokernel pairs, called \emph{conflations}, closed under isomorphisms.  Given a conflation $A\xrightarrow{f} B\xrightarrow{g} C$, we refer to the map $f$ as an \emph{inflation} and to the map $g$ as a \emph{deflation}.  Inflations will often be denoted by $\rightarrowtail$ and deflations by $\twoheadrightarrow$.  A map $f\colon X\rightarrow Y$ is called \emph{admissible} or \emph{strict} if it admits a deflation-inflation factorization, i.e.~$f$ factors as $X\twoheadrightarrow Z\rightarrowtail Y$.
		
	Let $\CC$ and $\DD$ be conflation categories.  An additive functor $F\colon \CC\rightarrow \DD$ is called \emph{conflation-exact} (or just \emph{exact}) if it maps conflations in $\CC$ to conflations in $\DD.$
\end{definition}

\begin{remark}
	The dual of a conflation category is a conflation category in a natural way: the duality exchanges inflations and deflations.\\
\end{remark}

\begin{definition}\label{definition:RightExact}
	A conflation category $\CC$ is called \emph{right exact} or \emph{deflation-exact} if it satisfies the following axioms:
	\begin{enumerate}[label=\textbf{R\arabic*},start=0]
		\item\label{R0} The identity morphism $1_0\colon 0\rightarrow 0$ is a deflation.
		\item\label{R1} The composition of two deflations is again a deflation.
		\item\label{R2} The pullback of a deflation along any morphism exists and is again a deflation, i.e. 
		\[\xymatrix{
		X\ar@{.>>}[d]\ar@{.>}[r] & Y \ar@{->>}[d]\\
		Z\ar[r] & W
		}\]
	\end{enumerate}
	
	 Dually, a conflation category $\CC$ is called \emph{left exact} or \emph{inflation-exact} if the opposite category $\CC^{op}$ is deflation-exact.  For completeness, an inflation-exact category is a conflation category such that the class of conflations satisfies the following axioms:
		\begin{enumerate}[label=\textbf{L\arabic*},start=0]
		\item\label{L0} The identity morphism $1_0\colon 0\rightarrow 0$ is an inflation.
		\item\label{L1} The composition of two inflations is again an inflation.
		\item\label{L2} The pushout of an inflation along any morphism exists and is again an inflation, i.e. 
		\[\xymatrix{
		X\ar@{>->}[d]\ar@{->}[r] & Y \ar@{>.>}[d]\\
		Z\ar@{.>}[r] & W
		}\]
	\end{enumerate}	
\end{definition}
	
	\begin{definition}\label{definition:StrongRightExact}
Let $\CC$ be a conflation category. In addition to the properties listed in definition \ref{definition:RightExact}, we will also refer to the following properties:
\begin{enumerate}[align=left]
\myitem{\textbf{R0}$^\ast$}\label{R0*} For any $A\in \Ob(\CC)$, $A\rightarrow 0$ is a deflation. 
\myitem{\textbf{R3}}\label{R3} \hspace{0.175cm}If $i\colon A\rightarrow B$ and $p\colon B\rightarrow C$ are morphisms in $\CC$ such that $p$ has a kernel and $pi$ is a deflation, then $p$ is a deflation.
\myitem{\textbf{L0}$^\ast$}\label{L0*} For any $A\in \Ob(\CC)$, $0\rightarrow A$ is an inflation. 
\myitem{\textbf{L3}}\label{L3} \hspace{0.175cm}If $i\colon A\rightarrow B$ and $p\colon B\rightarrow C$ are morphisms in $\CC$ such that $i$ has a cokernel and $pi$ is an inflation, then $i$ is an inflation.
\end{enumerate}
A deflation-exact category satisfying \ref{R3} is called \emph{strongly deflation-exact}.  Dually, an inflation-exact category satisfying \ref{L3} is called \emph{strongly inflation-exact}.
\end{definition}

\begin{remark}\makeatletter
\hyper@anchor{\@currentHref}%
\makeatother
\label{remark:R0*SplitKernelCokernel}
\begin{enumerate}
  \item The left/right terminology we use, follows \cite{Rosenberg11}, but is opposite to \cite{BazzoniCrivei13, Rump11}.  In order to avoid confusion, we use the terminology of deflation/inflation-exact categories.
  \item Axioms \ref{R3} and \ref{L3} are sometimes referred to as Quillen's \emph{obscure axioms} (see \cite{Buhler10,ThomasonTrobaugh90}). Note that axioms \ref{R0} and \ref{R3} imply axiom \ref{R0*}, and that axioms \ref{L0} and \ref{L3} imply axiom \ref{L0*} (see \cite[lemma~3.4]{BazzoniCrivei13} or \cite[lemma~3.6]{HenrardvanRoosmalen19}).  Indeed, by \ref{R0}, the composition $0 \xrightarrow{i} A \xrightarrow{p} 0$ is a deflation.  By \ref{R3}, we then see that $i\colon A \to 0$ is a deflation.
	\item An \emph{exact category} in the sense of Quillen (see \cite{Quillen73}) is a conflation category $\EE$ which is both left and right exact. It is shown in \cite[appendix~A]{Keller90} that an exact category automatically satisfies the obscure axioms \ref{R3} and \ref{L3}. 
	\item Let $\EE$ be a deflation-exact category.  Axiom \ref{R0*} has the following characterization: split kernel-cokernel pairs are conflations if and only if $\EE$ satisfies axiom \ref{R0*}.  Indeed, if for every $C \in \CC$, the split kernel-cokernel pair $\xymatrix@1{C \ar@{=}[r] & C \ar[r] & 0}$ is a conflation, then $\CC$ satisfies \ref{R0*}.  The converse follows from \cite[proposition~5.6]{BazzoniCrivei13}.
	\item \label{enumerate:R3Plus}	If $\EE$ is a weakly idempotent complete deflation-exact category, one can drop the requirement that $p$ admits a kernel in axiom \ref{R3}.  Thus, $\EE$ satisfies \ref{R3} if and only if for all maps $i\colon A\rightarrow B$ and $p\colon B\rightarrow C$ in $\EE$ such that $pi$ is a deflation, $p$ is a deflation. A dual statement holds for \ref{L3} (see \cite[proposition~7.6]{Buhler10}).
\end{enumerate}
\end{remark}

{We will need the following generalization of \cite[proposition~5.9]{BazzoniCrivei13} (see also \cite[corollary~2.18]{Buhler10}).  It states that, under the assumption of \ref{R3}, retracts of conflations are themselves conflations.}

\begin{proposition}\label{proposition:RetractConflation}
{Let $\EE$ be a deflation-exact category satisfying axiom \ref{R3}.  Consider the commutative diagram}
\[
\xymatrix{A \ar[r] \ar[d] & B \ar[r] \ar[d]&C \ar[d]^{s}\\
X \ar@{>->}[r] \ar[d] & Y \ar@{->>}[r] \ar[d] & Z \ar[d]^{r} \\
A \ar[r] & B \ar[r] &C}
\]
{where the vertical arrows compose to the identity.  If the sequence $X \to Y \to Z$ is a conflation, then so is the sequence $A \to B \to C$.}
\end{proposition}

\begin{proof}
As $\EE$ satisfies axiom \ref{R3}, we know $\EE$ satisfies axiom \ref{R0*} as well.  As $A \to B \to C$ is a retract of the kernel-cokernel pair $X \inflation Y \deflation Z,$ we know it is a kernel-cokernel pair itself.  By taking the pullback of $s\colon C \to Z$ along the deflation $Y \deflation Z$, we obtain a commutative diagram
\[
\xymatrix{X \ar@{>->}[r] \ar@{=}[d] & P \ar@{->>}[r] \ar[d] & C \ar[d]^{s}\\
X \ar@{>->}[r] \ar[d] & Y \ar@{->>}[r] \ar[d]&Z \ar[d]^{r}\\
A \ar[r] & B \ar[r]& C}\]
where the top rows is a conflation by axiom \ref{R2}.  As $r \circ s = 1_C$, we see that the composition $P \to B \to C$ is a deflation, so that axiom \ref{R3} implies that $B \to C$ is a deflation.  This finishes the proof.
\end{proof}

\subsection{Localizations and right multiplicative systems}
We recall some material about localization of categories at left or right multiplicative systems.  The material of this section is based on \cite{GabrielZisman67, KashiwaraSchapira06}.

\begin{definition}\label{definition:LocalizationWithRespectToMorphisms}
Let $\CC$ be any category and let $S \subseteq \Mor \CC$ be any subset of morphisms of $\CC$.  The \emph{localization} of $\CC$ with respect to $S$ is a universal functor $Q\colon \CC \to S^{-1} \CC$ such that $Q(s)$ is invertible, for all $s \in S$.  Here, universality means that any functor $F \colon \CC \to \DD$ for which $F(s)$ is invertible for all $s \in S$, factors uniquely through $Q\colon \CC \to S^{-1} \CC$.
\end{definition}

In this paper, we often consider localizations with respect to right multiplicative systems.

\begin{definition}\label{definition:RMS}
Let $\mathcal{C}$ be a category. A set $S\subseteq \Mor(\CC)$ is called a \emph{right multiplicative system} if it has the following properties:
\begin{enumerate}[label=\textbf{RMS\arabic*},start=1]
	\item\label{RMS1} For every object $A\in \Ob(\CC)$, the identity $1_A\colon A\to A$ is contained in $S$. The set $S$ is closed under composition.
	\item\label{RMS2} Every solid diagram
	\[\xymatrix{
		X \ar@{.>}[r]^{g} \ar@{.>}[d]_{t}^{\rotatebox{90}{$\sim$}}& Y\ar[d]_{s}^{\rotatebox{90}{$\sim$}}\\
		Z\ar[r]_{f} & W
	}\] with $s\in S$ can be completed to a commutative square with $t\in S$. 
	\item\label{RMS3} For every $s\in S$ with source $Y$ and every pair of morphisms $f,g\colon X\rightarrow Y$ in $\Mor(\CC)$ such that $s\circ f= s\circ g$ there exists a $t\in S$ with target $X$ such that $f\circ t =g\circ t$.
\end{enumerate}
Often arrows in $S$ will be endowed with $\sim$.

A \emph{left multiplicative system} is defined dually.  We say that $S$ is a \emph{multiplicative system} if it is both a left and a right multiplicative system.
\end{definition}

For localizations with respect to a right multiplicative system, we have the following description of the localization.

\begin{construction}
	Let $\mathcal{C}$ be a category and $S$ a right multiplicative system in $\mathcal{C}$. We define a category $S^{-1}\mathcal{C}$ as follows:
	\begin{enumerate}
		\item We set $\Ob(S^{-1}\mathcal{C})=\Ob(\mathcal{C})$.
		\item Let $f_1\colon X_1\rightarrow Y, s_1\colon X_1\rightarrow X, f_2\colon X_2\rightarrow Y, s_2\colon X_2\rightarrow X$ be morphisms in $\mathcal{C}$ with $s_1,s_2\in S$. We call the pairs $(f_1,s_1), (f_2,s_2) \in (\Mor \CC) \times S$ equivalent (denoted by $(f_1,s_1) \sim (f_2,s_2)$) if there is a commutative diagram
		\[\xymatrix@!{
			& X_1\ar[ld]_{s_1}^{}\ar[rd]^{f_1} & \\
			X &X_3\ar[d]^{v}\ar[u]_{u}\ar[l]_{s_3}^{}\ar[r]_{f_3} & Y\\
			& X_2 \ar[ul]^{s_2}_{}\ar[ur]_{f_2}&
		}\] with $s_3\in S$.
		\item $\Hom_{S^{-1}\mathcal{C}}(X,Y)=\left\{(f,s)\mid f\in \Hom_{\mathcal{C}}(X',Y), s\colon X'\rightarrow X \mbox{ with } s\in S \right\} / \sim$
		\item The composition of $(f\colon X'\rightarrow Y, s\colon X'\rightarrow X)$ and $(g\colon Y'\rightarrow Z, t\colon Y'\rightarrow Y)$ is given by $(g\circ h\colon X''\rightarrow Z,s\circ u\colon X''\rightarrow X)$ where $h$ and $u$ are given by a commutative diagram 
		\[\xymatrix{
		X''\ar[r]^{h}\ar[d]_{u}^{\rotatebox{90}{$\sim$}} & Y' \ar[d]_{t}^{\rotatebox{90}{$\sim$}}\\
		X'\ar[r]^{f} & Y
		}\] as in \ref{RMS2}.
	\end{enumerate}                                                                                                       
\end{construction}

\begin{proposition}\label{Proposition:BasicPropertiesOfLocalization}
Let $\mathcal{C}$ be a category and $S$ a right multiplicative system in $\mathcal{C}$.
	\begin{enumerate}
			\item The assignment $X\mapsto X$ and $(f\colon X\rightarrow Y)\mapsto (f\colon X\rightarrow Y, 1_X\colon X\rightarrow X)$ defines a localization functor $Q\colon \mathcal{C}\rightarrow S^{-1}\mathcal{C}$.  In particular, for any $s\in S$, the map $Q(s)$ is an isomorphism.
			\item The localization functor commutes with finite limits.
			\item If $\mathcal{C}$ is an additive category, then $S^{-1}\mathcal{C}$ is an additive category as well and the localization functor $Q\colon \CC \to S^{-1}\CC$ is additive.
	\end{enumerate}
\end{proposition}

\begin{definition}\label{definition:PropertiesOfMS}
Let $\CC$ be any category and let $S \subseteq \Mor \CC$ be any subset.
\begin{enumerate}
\item We say that $S$ satisfies the \emph{2-out-of-3 property} if, for any two composable morphisms $f,g \in \Mor(\CC)$, the following property holds: if two of $f,g,fg$ are in $S$, then so is the third.
\item Let $Q \colon \CC \to S^{-1}\CC$ be the localization of $\CC$ with respect to $S$.  We say that $S$ is \emph{saturated} if ${S} = \{f \in \Mor (\CC) \mid \mbox{$Q(f)$ is invertible}\}.$ 
\end{enumerate}
\end{definition}

\subsection{Verdier localizations}\label{subsection:VerdierLocalization}

We recall some notions and properties of Verdier localizations.  Our main references are \cite{Krause10, Neeman01, Verdier96}.  Let $\TT$ be a triangulated category and $\SS \subseteq \TT$ a full triangulated subcategory.  We write $\NN(\SS)$ for the set of morphisms $f\colon X \to Y$ in $\TT$ which fit in a triangle $X \to Y \to S \to \Sigma(X)$ with $S \in \SS$.

\begin{proposition}
The set $\NN(\SS)$ is a (left and right) multiplicative system satisfying the 2-out-of-3 property.  Furthermore, $\NN(\SS)$ is saturated if and only if $\SS$ is a thick subcategory of $\TT$.
\end{proposition}

We now define the \emph{Verdier localization} $\TT/\SS$ as the localization $\NN(\SS)^{-1} \TT$ together with the localization functor $\TT\to \TT/\SS$. 

\begin{proposition}
Let $\SS$ be a triangulated subcategory of a triangulated category $\TT$.
\begin{enumerate}
  \item The category $\TT / \SS$ has a unique structure of a triangulated category such that the localization functor $Q\colon \TT \to \TT / \SS$ is a triangle functor.
  \item The kernel of $Q\colon \TT \to \TT / \SS$ is the thick closure of $\SS$.
  \item A morphism in $\TT$ becomes zero in $\TT / \SS$ if and only if it factors through an object in $\SS.$
  \item Any triangle functor $F\colon \TT \to \UU$ satisfying $F(\SS) = 0$ factors uniquely through $Q\colon \TT \to \TT / \SS$.
\end{enumerate}
\end{proposition}

A sequence $\SS\stackrel{i}{\hookrightarrow} \TT\xrightarrow{f} \CC$ is called a \emph{Verdier localization sequence} if $f\circ i=0$, and the unique functor $\TT/\SS \to \CC$ induced by the previous proposition is an equivalence.
\section{Derived categories of deflation-exact categories}\label{section:DerivedCategoriesOfRightExactCategories}

We now come to the derived category of a deflation-exact category $\EE$.  We recall the definition of an acyclic complex and construct the derived category $\D(\EE)$ of $\EE$ as the Verdier quotient $\K(\EE) / \Ac(\EE)$ of the homotopy category by the category of acyclic complexes (see \cite{BazzoniCrivei13}). We discuss the embedding $i\colon \EE\to \D(\EE)$ and show that, if $\EE$ has enough projective objects, then the natural embedding $\Kab(\Proj(\EE))\to \Db(\EE)$ is a triangle equivalence. The main results are summarized in theorem \ref{theorem:SummarazationTheorem}.

\subsection{Definitions}

Let $\AA$ be an additive category.  We denote by $\C(\AA)$ the additive category of cochain complexes in $\AA$ and by $\K(\AA)$ the homotopy category of $\AA$.  It is well known that $\K(\AA)$ has the structure of a triangulated category induced by the strict triangles in $\C(\AA)$, see e.g.~\cite{BeilinsonBernsteinDeligne82,Happel88,KashiwaraSchapira06,Verdier96,Weibel94}. We denote by $\C^+(\AA), \C^-(\AA)$ and $\Cb(\AA)$ the full subcategories of $\C(\AA)$ generated by the left bounded, right bounded and bounded complexes, respectively. Similarly, we write $\K^+(\AA), \K^-(\AA)$ and $\Kb(\AA)$ for the full triangulated subcategories of $\K(\AA)$ which are homotopic to the left bounded, right bounded and bounded complexes, respectively. We write $\Cnm(\AA)$ for the bounded complexes in $\AA$ supported only in degrees $n$ to $m$.

Whenever we write $\C^*(\AA)$ or $\K^*(\AA)$ without explicitly specifying $*$, the statement holds for both the unbounded and bounded setting.

\begin{definition}\makeatletter
\hyper@anchor{\@currentHref}%
\makeatother\label{definition:ConesDoubleComplexesContractibles}
	\begin{enumerate}
		\item Given a chain map $f^{\bullet}\colon X^{\bullet}\rightarrow Y^{\bullet}$ in $\C^*(\AA)$ or $\K^*(\AA)$, the \emph{mapping cone} $\cone(f^{\bullet})$ is the complex given by $\cone(f^{\bullet})^n=X^{n+1}\oplus Y^n$ and the differentials $d^{n}_{f}$ are given by 
		\[\begin{pmatrix}
-d_X^{n+1} & 0\\
f^{n+1} & d_Y^n
\end{pmatrix}.\]
		\item Let $C^{\bullet,\bullet}$ be a double complex. We denote the vertical differentials by $d_v^{n,m}$ and the horizontal differentials by $d_h^{n,m}$. By convention, all squares of the bicomplex $C^{\bullet,\bullet}$ commute. The associated \emph{total complex} $\Tot^{\oplus}C^{\bullet,\bullet}$ is defined by 
\[\left(\Tot^{\oplus}C^{\bullet,\bullet}\right)^n=\bigoplus_{i+j=n}C^{i,j}\]
and the differentials $d_{\Tot}$ are given by the rule 
\[d_{\Tot}=d_h+(-1)^{\text{horizontal degree}}d_{v}.\]
		\item A complex $X^{\bullet}$ is called a \emph{standard contractible} complex if it is of the form:
		$$\cdots \rightarrow 0\rightarrow X\xrightarrow{1_X} X \rightarrow 0 \rightarrow \cdots$$
	\end{enumerate}
\end{definition}

\begin{remark}\label{remark:ConnectionConeAndTotalization}
	There is a strong connection between cones and totalizations. Given a chain map $f^{\bullet}\colon X^{\bullet}\rightarrow Y^{\bullet}$, one finds the mapping cone as the total complex associated to the double complex defined by placing $X^{\bullet}$ in the column of degree $-1$ and $Y^{\bullet}$ in the column of degree $0$ (see for example \cite[exercise~1.2.8]{Weibel94}).
\end{remark}

From now on, $\EE$ denotes a deflation-exact category.

\begin{definition}\label{definition:AcylicComplex}
	Let $\EE$ be a deflation-exact category. A complex $X^{\bullet}\in \C(\EE)$ is called \emph{acylic in degree $n$} if $d_X^{n-1}\colon X^{n-1} \to X^n$ factors as 
	\[\xymatrix{
	X^{n-1}\ar[rr]^{d_X^{n-1}}\ar@{->>}[rd]^{p^{n-1}} && X^n\\
	 & \ker(d_X^{n})\ar@{>->}[ru]^{i^{n-1}} &
	}\] where the deflation $p^{n-1}$ is the cokernel of $d_X^{n-2}$ and the inflation $i^{n-1}$ is the kernel of $d_X^{n}$.
	
	A complex $X^{\bullet}$ is called \emph{acyclic} if it is acylic in each degree.  The full subcategory of $\C(\EE)$ of acyclic complexes is denoted by $\AcC(\EE)$.  We write $\AcK(\EE)$ for the full subcategory of $\K(\EE)$ given by those complexes which are homotopy equivalent to an acyclic complex (thus, $\AcK(\EE)$ is the closure of $\AcC (\EE)$ under isomorphisms in $\K(\EE)$).  We simply write $\Ac(\EE)$ for either $\AcC(\EE)$ or $\AcK(\EE)$ if there is no confusion.  The bounded versions are defined by $\AcC^*(\EE)=\AcC (\EE)\cap \C^*(\EE)$ and $\AcK^*(\EE)=\AcK (\EE)\cap \K^*(\EE)$.
\end{definition}

\begin{remark}
Our terminology follows \cite{Keller90}.  What we call an acyclic complex has been called an exact complex in \cite{Positselski11}; the objects in $\AcK(\EE)$ are called acyclic complexes in \cite{Positselski11}.
\end{remark}

\begin{remark}\makeatletter
\hyper@anchor{\@currentHref}%
\makeatother\label{remark:StandardContractible}
	\begin{enumerate}
		\item If $\EE$ is abelian, then $X^{\bullet}$ is acyclic in degree $n$ if and only if the cohomology group $\Homology^n(X^{\bullet})$ is zero.
		\item If $\EE$ is deflation-exact, an acyclic complex $X^{\bullet}$ with three consecutive non-zero terms and all other terms being zero gives a conflation in $\EE$.  However, a conflation need not necessarily be an acyclic complex.  Indeed, the complex $\cdots \to 0 \to X\to Y\to Z\to 0\to \cdots$ is acyclic if and only if $X\to Y \to Z$ is a conflation and $Z\to 0$ is a deflation.  If $\EE$ satisfies axiom \ref{R0*}, any conflation yields a three-term acyclic complex.
		\item While standard contractible complexes are null-homotopic, they are acyclic if and only if the category $\EE$ satisfies axiom \ref{R0*}.
	\end{enumerate}
\end{remark}

The next proposition (see \cite[proposition~7.2]{BazzoniCrivei13}) will be used in the definition of the derived category of $\EE$.

\begin{proposition}\label{proposition:ConeOfAcyclic}
Let $\EE$ be a deflation-exact category. The mapping cone of a chain map $f^{\bullet}\colon X^{\bullet}\rightarrow Y^{\bullet}$ between acyclic complexes is acyclic. In particular, the category $\Ac_{\K}(\EE)$ is a triangulated subcategory of $\K(\EE)$. Similar statements hold for the bounded versions.
\end{proposition}

\begin{definition}\label{definition:Quasi-isomorphism}
	A chain map $f^{\bullet}\colon A^{\bullet}\rightarrow B^{\bullet}$ is a called a \emph{quasi-isomorphism} if its mapping cone is homotopy equivalent to an acyclic complex.
	
	The chain map $f^{\bullet}\colon A^{\bullet}\rightarrow B^{\bullet}$ is called a \emph{weak quasi-isomorphism} if its mapping cone is a direct summand of an acyclic complex up to homotopy equivalence.
\end{definition}

\begin{definition}\label{definition:DerivedCategory}
	Let $\EE$ be a deflation-exact category. 
	\begin{enumerate}
		\item The \emph{(unbounded) derived category} $\D(\EE)$ is defined as the Verdier localization $\K(\EE)/\Ac_{\K}(\EE)$. 
		\item The \emph{bounded derived category} $\Db(\EE)$ is defined as the full triangulated subcategory of $\D(\EE)$ spanned by bounded complexes up to weak quasi-isomorphism. The left and right bounded derived categories $\D^+(\EE)$ and $\D^-(\EE)$ are defined similarly.
	\end{enumerate}
\end{definition}

\begin{remark}\makeatletter
\hyper@anchor{\@currentHref}%
\makeatother\label{remark:ConditionsDefinitionDerivedCategory}
	\begin{enumerate}
		\item Let $\EE$ be a deflation-exact category.  In the notation of section \ref{subsection:VerdierLocalization}, the set of quasi-isomorphisms is denoted by $\NN(\Ac_{\K}(\EE))$ and the set of weak quasi-isomorphisms is given by $\NN(\langle \Ac_{\K}(\EE) \rangle_{\textrm{thick}})$ where $\langle \Ac_{\K}(\EE) \rangle_{\textrm{thick}}$ is the thick closure of $\Ac_{\K}(\EE)$ in $\K(\EE)$.  A chain map in $\K^*(\EE)$ becomes an isomorphism in $\D^*(\EE)$ if and only if it is a weak quasi-isomorphism.
		\item Alternatively, one could define the bounded derived category $\Db(\EE)$ as the Verdier localization $\Kb(\EE)/\left(\Ac_{\K}(\EE)\cap \Kb(\EE)\right)$. To show that this definition coincides with the previous definition, it suffices to show that if a roof
	\[\xymatrix@R0.5em{
		& Z^{\bullet}\ar[ld]_{\rotatebox{30}{$\simeq$}}\ar[rd] & \\
		X^{\bullet} & & Y^{\bullet}
	}\] represents a map $X^{\bullet}\to Y^{\bullet}$ in $\D(\EE)$ between bounded complexes $X^{\bullet}$ and $Y^{\bullet}$, one can find an equivalent roof consisting of only bounded complexes. We leave it to the reader to verify that this follows from propositions \ref{proposition:AcylicComplexAsExtensionOfTruncations} and \ref{proposition:TruncatedReplacements} below (see also \cite[lemma 11.7]{Keller96}).
	\end{enumerate}
\end{remark}

In proposition \ref{proposition:NecessityOfWeakIdempotentCompleteness} below, we give conditions under which $\Ac^*(\EE)\subseteq \K^*(\EE)$ is closed under isomorphisms.  We start by stating an adaptation of a well-known lemma.

\begin{lemma}\label{lemma:CoretractionInHomotopyCategory}
	Let $\EE$ be a deflation-exact category.
	\begin{enumerate}
		\item Let $K^{\bullet}\to L^{\bullet}$ be a coretraction in $\K^*(\EE)$. The map $K^{\bullet}\to L^{\bullet}\oplus IK^{\bullet}$ where $IK=\cone(1_{K^{\bullet}})$ is a coretraction in $\C^*(\EE)$.
		\item If $\EE$ satisfies axiom \ref{R0*}, the complex $IK^{\bullet}$ is acyclic.
	\end{enumerate}
\end{lemma}

\begin{proof}
	The proof is a straightforward adaptation of \cite[section~2.3.a]{Keller90}.
\end{proof}

\begin{proposition}\label{proposition:NecessityOfWeakIdempotentCompleteness}
	Let $\EE$ be a deflation-exact category and let $*\in \left\{+,-,b\right\}$.
	\begin{enumerate}
		\item If $\Ac_{\K}^*(\EE)$ is a thick subcategory of $\K^*(\EE)$, then $\EE$ is weakly idempotent complete.
		\item If $\EE$ satisfies \ref{R3} and is weakly idempotent complete, then $\Ac_{\K}^*(\EE)$ is a thick subcategory of $\K^*(\EE)$.
	\end{enumerate}
	Replacing weakly idempotent complete by idempotent complete, the lemma remains valid for unbounded derived categories.
\end{proposition}

\begin{proof}
The dual of the proof given in \cite[proposition 10.14]{Buhler10} works in this setting.  A one-sided version of \cite[proposition 7.6]{Buhler10} has been shown in \cite[proposition 6.4]{BazzoniCrivei13} and requires axiom \ref{R3} (see remark \ref{remark:R0*SplitKernelCokernel}.\ref{enumerate:R3Plus}).  For the unbounded derived category, the proof of \cite[corollary~10.11]{Buhler10} can be taken.
\end{proof}

\subsection{Truncations of complexes}

Truncations are an important tool for studying complexes in an abelian (\cite{BeilinsonBernsteinDeligne82}) or quasi-abelian (\cite{Schneiders99}) setting. The next definition defines such truncations in an additive setting provided that some differentials have kernels (or cokernels).

\begin{definition}\label{definition:Truncations}
	Let $C^{\bullet}$ be a complex in an additive category $\AA$ such that $d^{n-1}_C\colon C^{n-1} \to C^n$ factors as 
	\[C^{n-1}\xrightarrow{p^{n-1}} \ker(d^n)\xrightarrow{i^{n-1}} C^n\]
	where $i^{n-1}$ is the kernel of $d^n_C$. The \emph{canonical truncation} $\tau^{\leq n}C^{\bullet}$ is a complex together with a morphism $\tau^{\leq n}C^{\bullet} \to C^\bullet$ given by:
	\[\xymatrix{
		\tau^{\leq n} C^{\bullet}\ar[d] &&\cdots\ar[r] & C^{n-3}\ar[r]\ar@{=}[d] & C^{n-2}\ar[r]\ar@{=}[d] & C^{n-1}\ar[r]^{p^{n-1}}\ar@{=}[d] & \ker(d_C^{n})\ar[r]\ar[d]^{i^{n-1}} & 0\ar[r]\ar[d] & \cdots\\
		C^{\bullet} 								&&\cdots\ar[r] & C^{n-3}\ar[r] & C^{n-2}\ar[r] & C^{n-1}\ar[r] & C^{n}\ar[r] & C^{n+1}\ar[r] & \cdots
	}\] and the \emph{canonical truncation} $C^\bullet \to \tau^{\geq n+1}C^{\bullet}$ is similarly defined by:
	\[\xymatrix{
	C^{\bullet}\ar[d] 		&& \cdots\ar[r] & C^{n-3}\ar[r]\ar[d] & C^{n-2}\ar[r]\ar[d] & C^{n-1}\ar[r]\ar[d]^{p^{n-1}} & C^{n}\ar[r]\ar@{=}[d]  & C^{n1}\ar[r]\ar@{=}[d] & \cdots\\
	\tau^{\geq n+1} C^{\bullet} && \cdots\ar[r] & 0\ar[r] 			 & 0\ar[r]& \ker(d_C^{n})\ar[r]^{i^{n-1}}& C^{n}\ar[r] & C^{n+1}\ar[r] &\cdots
	}\]
\end{definition}

\begin{proposition}\label{proposition:AcylicComplexAsExtensionOfTruncations}
Let $\AA$ be an additive category and let $C^{\bullet}\in \C^*(\AA)$. If $d^n_C$ has a kernel, then the following triangle is a distinguished triangle in $\K^*(\AA)$
\[\tau^{\leq n}C^{\bullet}\rightarrow C^{\bullet} \rightarrow \tau^{\geq n+1}C^{\bullet}\rightarrow \Sigma \left( \tau^{\leq n}C^{\bullet}\right)\]
In other words, $C^{\bullet}$ is an extension of the canonical truncation $\tau^{\geq n+1}C^{\bullet}$ by $\tau^{\leq n}C^{\bullet}$ in $\K^*(\AA)$.
\end{proposition}

\begin{proof}
It suffices to show that $\cone(f^\bullet\colon \tau^{\geq n+1}C^{\bullet}\rightarrow \Sigma ( \tau^{\leq n}C^{\bullet}))\cong \Sigma (C^\bullet)$ in $\K^*(\AA).$  The cone is given by the following diagram:
\[\xymatrix{
		\tau^{\geq n+1}C^{\bullet}\ar[d]^{f^{\bullet}}&\cdots\ar[r]&0\ar[r]\ar[d]&0\ar[rr]\ar[d]&&K\ar[r]^{i^{n-1}}\ar@{=}[d]&C^{n}\ar[r]\ar[d]&C^{n+1}\ar[r]\ar[d]&\cdots\\
		\Sigma \tau^{\leq n}C^{\bullet}\ar[d]^{-1}&\cdots\ar[r]&C^{n-2}\ar[d]^{\begin{psmallmatrix}0\\-1\end{psmallmatrix}}\ar[r]&C^{n-1}\ar[d]^{\begin{psmallmatrix}0\\-1\end{psmallmatrix}}\ar[rr]^{-p^{n-1}}&&K\ar[d]\ar[r]&0\ar[d]\ar[r]&0\ar[d]\ar[r]&\cdots\\
		\cone(f^{\bullet}) &\cdots\ar[r]& C^{n-2}\ar[r]_-{\begin{psmallmatrix}0\\-d^{n-2}_C\end{psmallmatrix}} & K\oplus C^{n-1}\ar[rr]_{\begin{psmallmatrix} -i^{n-1} & 0\\ 1 & -p^{n-1}\end{psmallmatrix}} && C^n\oplus K\ar[r]_{\begin{psmallmatrix}-d^{n}_C & 0\end{psmallmatrix}} & C^{n+1}\ar[r]_{-d^{n+1}_C} & C^{n+2}\ar[r] & \cdots}\]
It is now straightforward to verify that $\cone(f^\bullet) \cong \Sigma (C^\bullet).$
\end{proof}

\begin{remark}
Note that the truncations of proposition \ref{definition:Truncations} are not necessarily defined on the whole of $\K^*(\EE)$.  Moreover, the domain of the truncations is not necessarily closed under homotopy equivalence.
\end{remark}

\subsection{Congenial quasi-isomorphisms and roofs}

Let $Z^\bullet \to X^\bullet$ be a quasi-isomorphism.  In general, even if $X^\bullet$ were acyclic in $n\in \bZ$, then $Z^\bullet$ need not be.  We now consider a class of quasi-isomorphisms that better reflects properties of $X^\bullet.$

\begin{definition}\label{definition:Congenial}
A quasi-isomorphism $\alpha^{\bullet}\colon \zeta^{\bullet} \stackrel{\simeq}{\rightarrow} X^{\bullet}$ is called \emph{congenial} if the following conditions are satisfied:
\begin{enumerate}
	\item\label{enumerate:CongenialAcyclic} if $X^\bullet$ is acyclic in degree $n$, then $\zeta^\bullet$ is acyclic in degree $n$,
	\item\label{enumerate:CongenialDeflations} for each $n \in \bZ,$ the morphism $\alpha^n\colon \zeta^n \deflation X^n$ is a deflation,
	\item\label{enumerate:CongenialSameBound} for each $n \in \bZ,$ if $\forall m \geq n: X^m = 0$, then $\zeta^n = 0$. 
\end{enumerate}
\end{definition}

\begin{remark}
The last condition states that if $X^\bullet$ is supported in degrees at most $n$, then so is a congenial replacement of $X^\bullet.$  In particular, if $\zeta^\bullet \to 0$ is a congenial quasi-isomorphism, then $\zeta^\bullet = 0.$
\end{remark}

\begin{definition}
Let $X^\bullet \stackrel{\simeq}{\leftarrow} Z^\bullet \to Y^\bullet$ be a roof in $\C(\EE)$ with $\alpha\colon Z^\bullet\to X^\bullet$ a quasi-isomorphism.  We say that the roof is a \emph{congenial roof} if $\alpha$ is a congenial morphism, and for each $n \in \bZ, \forall m \leq n: X^m = Y^m=  0$, then $\zeta^n = 0$. 
\end{definition}

\begin{proposition}\label{proposition:HomotopyReplacement}\label{proposition:TruncatedReplacements}\label{proposition:CongenialReplacement}
	Let $\EE$ be a deflation-exact category.
	\begin{enumerate}
		\item For every quasi-isomorphism $\beta\colon Z^{\bullet}\xrightarrow{\simeq} X^{\bullet}$, there exists a quasi-isomorphism $\zeta^{\bullet} \xrightarrow{\simeq} Z^{\bullet}$ such that the composition $\zeta^\bullet \to Z^\bullet \to X^\bullet$ is a congenial quasi-isomorphism.
		\item Any morphism $X^\bullet \to Y^\bullet$ can be represented by a congenial roof $X^\bullet \stackrel{\simeq}{\leftarrow} \zeta^\bullet \to Y^\bullet$.
	\end{enumerate}
\end{proposition}

\begin{proof}
	Without loss of generality, we may assume that $X^\bullet \neq 0$.  Since $\beta$ is a quasi-isomorphism, there exists a homotopy equivalence $\cone(\beta)\xrightarrow{\delta} E^{\bullet}$ where $E^{\bullet}$ is an acyclic complex.  We find a commutative diagram in $\K(\EE)$ whose rows are triangles:
	\[\xymatrix{
		Z^{\bullet}\ar[r]^{\beta} & X^{\bullet}\ar[r]\ar@{=}[d] & \cone(\beta)\ar[r]\ar[d]^{\delta} & \Sigma Z^{\bullet}\\
		\Sigma^{-1}\cone(\gamma)\ar[r] & X^{\bullet}\ar[r]^{\gamma} & E^{\bullet}\ar[r] & \cone(\gamma)
	}\]
	Since $\delta$ is a homotopy equivalence, these triangles are isomorphic in $\K(\EE)$. Hence, there exists a homotopy equivalence $\Sigma^{-1}\cone(\gamma)\to Z^{\bullet}$.  The morphism $\Sigma^{-1}\cone(\gamma)\to X^{\bullet}$ is a quasi-isomorphism and a pointwise (split) deflation.  As in the proof of \cite[lemma 7.2]{BazzoniCrivei13}, if $X^\bullet$ is acyclic in degree $n$, then so is $\zeta^\bullet = \Sigma^{-1}\cone(\gamma).$  Hence, conditions \eqref{enumerate:CongenialAcyclic} and \eqref{enumerate:CongenialDeflations} of definition \ref{definition:Congenial} are satisfied.
	
If $X^\bullet$ is not bounded on the right, then condition \eqref{enumerate:CongenialSameBound} is trivially satisfied.  So, assume that $X^\bullet$ is bounded on the right.  Without loss of generality, assume that $X^0$ is the highest degree where $X^\bullet$ has nonzero entries.  Then $\zeta^0\cong X^0\oplus E^{-1},\zeta^1=E^0$ and $d_{\zeta}^0\colon X^0\oplus E^{-1}\to E^0$ is given by $-\begin{psmallmatrix} \gamma^0 & d^{-1}_E \end{psmallmatrix}$.  We claim that $d_{\zeta}^0$ has a kernel.
	
The map $\gamma^\bullet\colon X^\bullet \to E^\bullet$ induces the commutative diagram
\[\xymatrix@C+6pt{
		\ker(d_E^{-1})\ar@{>->}[r]^-{\iota^{-2}}\ar@{=}[d]&P\ar@{->>}[r]^{\rho^{-1}}\ar[d] & X^0\ar@{.>}[d]_{\phi}\ar[rd]^{\gamma^0} & & \\
		\ker(d_E^{-1})\ar@{>->}[r]^-{i_E^{-2}}&E^{-1}\ar@{->>}[r]^{p_E^{-1}} \ar@/_2pc/[rr]_{d^{-1}_E}& \ker(d_E^0)\ar@{>->}[r]^{i_E^{-1}} & E^0\ar[r]^{d_E^0} & E^1
	}\]
where the map $\phi\colon X^0 \to \ker(d_E^0)$ is induced by $d_E^0\circ \gamma^0=\gamma^1\circ d_X^0=0$ and the right square is a pullback square.  This establishes that $d_{\zeta}^0$ admits a kernel (the kernel of $-d_{\zeta}^0 = \begin{psmallmatrix} \gamma^0 & d^{-1}_E \end{psmallmatrix}\colon X^0\oplus E^{-1}\to E^0$ is given by $P \to X \oplus E^{-1}$).
	
	Hence, we may apply proposition \ref{proposition:AcylicComplexAsExtensionOfTruncations} to find a triangle $\tau^{\leq 0}\zeta^{\bullet}\rightarrow \zeta^{\bullet} \rightarrow \tau^{\geq 1}\zeta^{\bullet}\rightarrow \Sigma \left( \tau^{\leq 0}\zeta^{\bullet}\right)$ in $\K(\EE)$.  It is easily seen that $ \tau^{\geq 1}\zeta^{\bullet} \in \Ac_{\K}(\EE)$ so that $\tau^{\leq 0} \zeta^{\bullet} \to \zeta^{\bullet}$ is a quasi-isomorphism.  It follows that the composition $\tau^{\leq 0} \zeta^{\bullet} \xrightarrow{\simeq} \zeta^{\bullet}\xrightarrow{\simeq} Z^{\bullet}$ is a congenial quasi-isomorphism as required.
	
	For the second statement, let $\zeta^\bullet\to X^\bullet$ be the congenial quasi-isomorphism as constructed above.  In particular, all morphisms $\zeta^l \to X^l$ are split (safe for $l=0$).  Let $l \in \bZ$ be the smallest integer such that not both $X^l$ and $Y^l$ are zero.  As we have assumed that $X^0 \neq 0$, we have that $l \leq 0.$  Assume first that $l<0$.  Setting
		\[C^i = \begin{cases} 
	0 & i < l, \\
\coker d_\zeta^{l-1} (= \coker d_E^{l-1}\oplus X^l) & i=l, \\
	\zeta^i & i > l.
	\end{cases}\]
	we find a commutative diagram
		\[\xymatrix{
	X^\bullet & \zeta^\bullet \ar[r]\ar[l]^{\simeq} \ar[d]^{\rotatebox{90}{$\simeq$}}& Y^\bullet \\
	& C^\bullet \ar[ur] \ar[ul]^{\simeq}}\]
	as required.  When $l = 0$, we obtain a similar diagram where $C^\bullet$ is given by the stalk complex containing $\coker (\zeta^{-1} \to \zeta^0)$ in degree zero.
\end{proof}                    

\begin{remark}\label{remark:CongenialRoofBetweenStalkComplexes}
When $X^\bullet, Y^\bullet$ are stalk complexes, supported in degree 0, then in any congenial roof $X^\bullet \stackrel{\simeq}{\leftarrow} Z^\bullet \to Y^\bullet$, we know that $Z^\bullet$ is a stalk complex as well.  The map $Z^0\to X^0$ is then an isomorphism.
\end{remark}

\subsection{The canonical embedding \texorpdfstring{$i\colon \EE\hookrightarrow \Db(\EE)$}{}}\label{subsection:CanonicalEmbedding}

Let $\EE$ be a deflation-exact category. Having defined the derived category and truncations, we now examine the canonical embedding $i\colon \EE\hookrightarrow \D(\EE)$. The next proposition says that, as in the abelian or exact case, the category $\EE$ is equivalent to the category of stalk complexes in $\Db(\EE)$ in degree zero.

\begin{proposition}\label{proposition:StalkEmbeddingFunctors}
	Let $\EE$ be a deflation-exact category. The canonical functor $i\colon \EE\rightarrow \D^*(\EE)$ mapping objects to stalk complexes in degree zero is fully faithful.
\end{proposition}

\begin{proof}
It follows from remark \ref{remark:CongenialRoofBetweenStalkComplexes} that $i\colon \EE\rightarrow \D^*(\EE)$ is full.  To see that $i$ is faithful, consider a map $f\colon X \to Y$ in $\EE$ such that $i(f)\colon i(X) \to i(Y)$ is zero.  This implies that there is a quasi-isomorphism $\alpha\colon Z^\bullet \to i(X)$ such that $\alpha \circ i(f)$ is null-homotopic.  By proposition \ref{proposition:CongenialReplacement}, we may assume that $\alpha$ is a congenial quasi-isomorphism.  As $Z^1 = 0$, we find that $\alpha \circ i(f)$ is null-homotopic if and only if $\alpha \circ i(f) = 0.$  Since $\alpha^0$ is an epimorphism, this implies that $f$ is zero.  Hence, $i\colon \EE\rightarrow \D^*(\EE)$ is faithful.
\end{proof}

Under the above embedding, the category $\EE$ has no negative extensions in $\D^*(\EE)$.

\begin{proposition}\label{proposition:NoNegativeExtensions}
	Let $\EE$ be a deflation-exact category. Let $X,Y\in \Ob(\EE)$, then 
	\begin{align*}
	\Hom_{\D^*(\EE)}(i(X), \Sigma^{-n}i(Y))=0, && \mbox{for all $n>0$.}
	\end{align*}
\end{proposition}

\begin{proof}
	Let $n>0$ and let $g\in \Hom_{\D^*(\EE)}(i(X), \Sigma^{-n}i(Y))$. The morphism $g$ can be represented by a congenial roof $i(X) \xleftarrow{\alpha^{\bullet}} Z^{\bullet} \xrightarrow{f^{\bullet}} \Sigma^{-n}i(Y)$.  As $Z^\bullet$ and $\Sigma^{-n} i(Y)$ are stalk complexes concentrated in degrees 0 and $n$, respectively, the map $f\colon Z^\bullet \to \Sigma^{-n} i(Y)$ is zero. Hence, $g$ is zero, as required.
\end{proof}

The next lemma will be useful.

\begin{lemma}\label{lemma:TotalizationOfDoubleComplex}
Let $C^{\bullet,\bullet}$ be a double complex such that for every $n\in \mathbb{Z}$ the set $\left\{(i,j)\mid i+j=n, C^{i,j}\neq 0\right\}$ is finite. 
\begin{enumerate}
	\item If all rows or all columns of $C^{\bullet,\bullet}$ are acyclic, then $\Tot^{\oplus} C^{\bullet,\bullet}$ is acyclic.
	\item Let \[D^{n,m}=\begin{cases}C^{n,m} &\mbox{ if } n\leq k,\\
	0 & \mbox{ if } n>k\end{cases}$$ and let $$E^{n,m}=\begin{cases} 0 & \mbox{ if } n\leq k,\\
	C^{n,m} & \mbox{ if } i>k
	\end{cases}\] be na\"{i}ve truncations of $C^{\bullet,\bullet}$. Then
	\[\Tot^{\oplus} C^{\bullet,\bullet} = \cone(\Sigma^{-1}\Tot^{\oplus}D^{\bullet,\bullet}\rightarrow  \Tot^{\oplus}E^{\bullet,\bullet})\]
where $\Tot^{\oplus}D^{\bullet,\bullet}\rightarrow \Sigma \Tot^{\oplus}E^{\bullet,\bullet}$ is the natural map induces by the differential in $C^{\bullet,\bullet}$ (note that there appears an alternating sign in this map).
	
	A similar statement holds for vertical cuts.
\end{enumerate}
\end{lemma}

\begin{proof}
The second statement is an exercise in sign bookkeeping.  For the first statement, assume that the rows are acyclic.  To show that the complex $\Tot^{\oplus} C^{\bullet,\bullet}$ is acyclic, it suffices to show that it is acyclic in each $n \in \bZ$.  As in definition \ref{definition:AcylicComplex}, this is only dependent on a part of the complex:
\[\left(\Tot^{\oplus} C^{\bullet,\bullet}\right)^{n-2} \to \left(\Tot^{\oplus} C^{\bullet,\bullet}\right)^{n-1} \to \left(\Tot^{\oplus} C^{\bullet,\bullet}\right)^{n} \to \left(\Tot^{\oplus} C^{\bullet,\bullet}\right)^{n+1}.\]
As we have assumed that $\left\{(i,j)\mid i+j=n, C^{i,j}\neq 0\right\}$ is finite, the statement that the complex $\Tot^{\oplus} C^{\bullet,\bullet}$ is acyclic in degree $n$ depends only on finitely many rows of the bicomplex $C^{\bullet, \bullet}$.  So, consider the bicomplex $F^{\bullet, \bullet}$ obtained from $C^{\bullet, \bullet}$ by replacing all other rows by zero.  By construction, the sequence $\Tot^{\oplus} C^{\bullet,\bullet}$ is acyclic in degree $n$ if and only if $\Tot^{\oplus} E^{\bullet,\bullet}$ is acyclic in degree $n$.  As in the second statement, we find that $\Tot^{\oplus} E^{\bullet,\bullet}$ is obtained via consecutive cones.  As the rows are acyclic, this shows that $\Tot^{\oplus} E^{\bullet,\bullet}$ is acyclic (see proposition \ref{proposition:ConeOfAcyclic}).  We conclude that $\Tot^{\oplus} C^{\bullet,\bullet}$ is acyclic as well.
\end{proof}

Given a deflation-exact category $\EE$ one can extend the deflation-exact structure to $\C(\EE)$ degree-wise. The canonical embeddings $i\colon \EE\hookrightarrow \D^*(\EE)$ and $\C^*(\EE)\hookrightarrow \D^*(\EE)$ are compatible with the deflation-exact structure in the following sense.

\begin{proposition}\label{proposition:ConflationsYieldTriangles}
	Let $\EE$ be a deflation-exact category satisfying \ref{R0*}.  A conflation $X^{\bullet}\stackrel{f^{\bullet}}{\inflation} Y^{\bullet}\stackrel{g^{\bullet}}{\deflation} Z^{\bullet}$ in $\C^*(\EE)$ induces a triangle $$\xymatrix{X^{\bullet}\ar[r]^{f^{\bullet}} & Y^{\bullet}\ar[r]^{g^{\bullet}} & Z^{\bullet}\ar[r] & \Sigma X^{\bullet}}$$
 in $\D^*(\EE)$.
\end{proposition}

\begin{proof}
	Consider the commutative diagram in $\D^*(\EE)$ 
	\[\xymatrix{
	X^{\bullet}\ar[r]^{f^{\bullet}}\ar@{=}[d] & Y^{\bullet}\ar[r]\ar@{=}[d]			&  \cone(f^{\bullet})\ar[d]^{h^{\bullet}}\ar[r]^{w^\bullet} & \Sigma X^{\bullet}\\
	X^{\bullet}\ar[r]^{f^{\bullet}} 					& Y^{\bullet}\ar[r]^{g^{\bullet}} & Z^{\bullet} 	}\]
	where $h^\bullet$ is given by $\begin{psmallmatrix} 0 & g^i\end{psmallmatrix}\colon X^{i+1} \oplus Y^i \to Z^i$.
	
	It suffices to show that $h$ is a quasi-isomorphism and hence invertible in $\D^*(\EE)$, so that the lower sequence fits into the triangle:
	\[\xymatrix@1{
	X^{\bullet}\ar[r]^{f^{\bullet}} & Y^{\bullet}\ar[r]^{g^{\bullet}} & Z^{\bullet} \ar@{.>}[rr]^{w^\bullet \circ (h^\bullet)^{-1}} && {\Sigma X^\bullet.}
	}\]
By definition, the cone of $h^{\bullet}$ is given by 
	\[\cdots\rightarrow X^{n+2}\oplus Y^{n+1}\oplus Z^n \xrightarrow{\begin{pmatrix} d_X^{n+1} & 0&0\\ -f^{n+1} & -d_Y^n & 0\\ 0&g^{n+1} & d_Z^n\end{pmatrix}} X^{n+3}\oplus Y^{n+2}\oplus Z^{n+1} \rightarrow \cdots \] We need to show that this cone is acyclic. Now consider the double complex 
	\[\xymatrix{
		& 0\ar[d] & 0\ar[d] & 0\ar[d] & \\
		\ar[r] & X^n\ar[r]\ar[d]^{f^n} & X^{n+1}\ar[r]\ar[d]^{f^{n+1}} & X^{n+2}\ar[r]\ar[d]^{f^{n+2}} & \\
		\ar[r] & Y^n\ar[r]\ar[d]^{g^n} & Y^{n+1}\ar[r]\ar[d]^{g^{n+1}} & Y^{n+2}\ar[r]\ar[d]^{g^{n+2}} & \\
		\ar[r] & Z^n\ar[r]\ar[d]^{} & Z^{n+1}\ar[r]\ar[d] & Z^{n+2}\ar[r]\ar[d] & \\
		& 0 & 0 & 0 & 
	}\] By definition, the associated total complex is equal to the cone of $h^{\bullet}$. As each column is a conflation and $\EE$ satisfies axiom \ref{R0*}, remark \ref{remark:StandardContractible} yields that the columns are acyclic complexes. By lemma \ref{lemma:TotalizationOfDoubleComplex} the associated total complex is acyclic. This concludes the proof.
\end{proof} 

\subsection{The homotopy category of projectives}

We discuss projective and injective objects in a one-sided exact category.  We show, if $\EE$ has enough projectives, that $\Db(\EE)$ is triangle equivalent to the homotopy category of projective objects (see proposition \ref{proposition:EnoughProjectives}). The following definition is standard.

\begin{definition}
	Let $\EE$ be a deflation-exact category (or conflation category).
	\begin{enumerate}
		\item An object $P\in \EE$ is called \emph{projective} if $\Hom(P,-)\colon \EE \to \Ab$ is an exact functor.  We say that $\EE$ has \emph{enough projectives} if for every object $M$ of $\EE$ there is a deflation $P\twoheadrightarrow M$ where $P$ is projective.
		\item Dually, an object $I$ is called \emph{injective} if $\Hom(-,I)\colon \EE^\circ \to \Ab$ is an exact functor.   We say that $\EE$ has \emph{enough injectives} if for every object $M$ of $\EE$ there is an inflation $M \inflation I$ where $I$ is projective.
	\end{enumerate}
	We write $\Proj \EE$ and $\Inj \EE$ for the full subcategories of projective objects and injective objects, respectively.
\end{definition}

\begin{proposition}\label{proposition:ProjectiveCharacterizations}
Let $\EE$ be a deflation-exact category. The following are equivalent:
\begin{enumerate}
	\item $P$ is projective.
	\item\label{Item:ProjLift} For all deflations $f\colon X\twoheadrightarrow Y$ and any map $g\colon P\rightarrow Y$ there exists a map $h\colon P\rightarrow X$ such that $g=f\circ h$.
	\item\label{Item:ProjSplit} Any deflation $f\colon X\twoheadrightarrow P$ is a retraction, i.e.~there exist a map $g\colon P\rightarrow X$ such that $f\circ g=1_P$.
\end{enumerate}
\end{proposition}

\begin{proof}
	The only non-standard implication is $\eqref{Item:ProjSplit}\implies \eqref{Item:ProjLift}$. Thus, assume that \eqref{Item:ProjSplit} holds.  Let $f\colon X\twoheadrightarrow Y$ be a deflation and $g\colon P\rightarrow Y$ a map. By axiom \ref{R2} there exists a pullback square
	\[\xymatrix{
	Q\ar@{->>}[r]^{a}\ar[d]^b & P\ar[d]^g\\
	X\ar@{->>}[r]^f & Y
	}\] Since $Q\twoheadrightarrow P$ is a deflation, there is a corresponding section $h\colon P\rightarrow Q$.  It is straightforward to show that $f\circ(b\circ h)=g$, establishing \eqref{Item:ProjLift}.
\end{proof}	

\begin{proposition}\label{proposition:InjectiveCharacterizations}
Let $\EE$ be a deflation-exact category. The following are equivalent:
\begin{enumerate}
	\item $P$ is injective.
	\item For all inflations $f\colon X\rightarrowtail Y$ and for any map $g\colon X\rightarrow I$ there exists a map $h\colon Y\rightarrow I$ such that $g=h\circ f$.
\end{enumerate}
\end{proposition}

\begin{remark}\label{remark:NonDualityProjInj}
	In a deflation-exact category, it is not necessarily true that an object $I$ is injective if all inflations $I\rightarrowtail X$ are coretractions. Indeed, the proof of proposition \ref{proposition:ProjectiveCharacterizations} cannot be dualized as axiom \ref{L2} is not guaranteed.
\end{remark}

\begin{proposition}\label{proposition:LiftingPropertyProjectives}
Let $\EE$ be a deflation-exact category.  Let $P^\bullet \in \Cm(\EE)$ be a complex of which every entry is projective.  For all $X^\bullet \in \C(\EE)$, we have $\Hom_{\K(\EE)}(P^\bullet, X^\bullet) = \Hom_{\D(\EE)}(P^\bullet, X^\bullet)$.
\end{proposition}

\begin{proof}
Without loss of generality, we may assume that $P^i = 0$ for $i > 0$. As in \cite[theorem~12.4]{Buhler10}, the lifting property of projective objects yields that $\Hom_{\K(\EE)}(P^\bullet, X^\bullet) = 0$ whenever $X^\bullet$ is acyclic (and hence also when $X^\bullet$ is a direct summand of an acyclic complex in $\K(\EE)$).

Let $f\colon Y^\bullet \stackrel{\simeq}{\rightarrow} X^\bullet$ be a quasi-isomorphism in $\K(\EE)$.  As then $\Hom_{\K(\EE)}(P^\bullet, \Sigma^n \cone(f)) = 0$ for all $n \in \bZ$, we find that $f$ induces a bijection $\Hom_{\K(\EE)}(P^\bullet, Y^\bullet) \cong \Hom_{\K(\EE)}(P^\bullet, X^\bullet)$. From the description of the maps in $\D(\EE)$, we find that the natural map $\Hom_{\K(\EE)}(P^\bullet, X^\bullet)\to \Hom_{\D(\EE)}(P^\bullet, X^\bullet)$ is a bijection as well.
\end{proof}

\begin{definition}
	We denote by $\Kab(\EE)$ the full subcategory of $\K(\EE)$ consisting of those complexes $C^\bullet$ bounded in the positive direction (i.e.~$C^i = 0$, for $i \gg 0$) and which are quasi-isomorphic to a bounded complex.
\end{definition}

\begin{proposition}\label{proposition:EnoughProjectives}
Let $\EE$ be a deflation-exact category satisfying axiom \ref{R0*}.  If $\EE$ has enough projectives, then the natural functor $\Kab(\Proj \EE) \to \Db(\EE)$ is a triangle equivalence.  A similar statement holds if $\EE$ has enough injectives.
\end{proposition}

\begin{proof}
	Since each object of $\Kab(\Proj \EE)$ is quasi-isomorphic to a bounded complex, there is an obvious triangle functor $F\colon \Kab(\Proj \EE)\rightarrow \Db(\EE)$. It is shown in proposition \ref{proposition:LiftingPropertyProjectives} that this functor $\Kab(\Proj \EE)\rightarrow \Db(\EE)$ is fully faithful.
	
	We only need to show that $F$ is essentially surjective.  To show this, let $C^{\bullet}\in \Ob(\Db(\EE))$ be any object. Since $\EE$ has enough projectives, each $C^i$ has a projective resolution $P^{i,\bullet}$ (see for example \cite[theorem~12.7]{Buhler10}). Using projectivity one easily obtains the following commutative diagram
	\[\xymatrix{
		  & \ar@{.>}[d] & \ar@{.>}[d]  &  & \ar@{.>}[d]  & \ar@{.>}[d]  & \ar@{.>}[d]  &\\
		\ar[r] & 0\ar[r]\ar[d]& P^{m,-1} \ar@{.>}[r]\ar[d] & \cdots\ar@{.>}[r] & P^{n-1,-1} \ar@{.>}[r]\ar[d] & P^{n,-1} \ar@{.>}[r]\ar[d] & 0\ar[r]\ar[d] &\\
		\ar[r] & 0\ar[r]\ar[d] & P^{m,0} \ar@{.>}[r]\ar@{->>}[d] & \cdots\ar@{.>}[r] & P^{n-1,0} \ar@{.>}[r]\ar@{->>}[d] & P^{n,0} \ar@{.>}[r]\ar@{->>}[d] & 0\ar[r]\ar[d] &\\
		\ar[r] & 0\ar[r] & C^m \ar[r] & \cdots\ar[r] & C^{n-1} \ar[r] & C^n \ar[r] & 0\ar[r] &\\
	}\] By lemma \ref{lemma:TotalizationOfDoubleComplex} and remark \ref{remark:StandardContractible}, the totalization of this double complex is acyclic. Denote by $P^{\bullet,\bullet}$ the projective part of this double complex. It follows that the cone of the natural map $\Tot^{\oplus}(P^{\bullet,\bullet})\rightarrow C^{\bullet}$ is acyclic. Hence, $C^{\bullet}$ is quasi-isomorphic to $\Tot^{\oplus}(P^{\bullet,\bullet})\in \Ob(\Kab(\Proj \EE))$. Hence, the functor $F\colon \Kab(\Proj \EE)\rightarrow \Db(\EE)$ is essentially surjective. This establishes that $F$ is a triangle equivalence, as required.
\end{proof}

\begin{proposition}
Let $\EE$ be a deflation-exact category satisfying axiom \ref{R0*}.  If $\EE$ has enough projectives, then the natural functor $\Km(\Proj \EE) \to \Dm(\EE)$ is a triangle equivalence.  A similar statement holds if $\EE$ has enough injectives.
\end{proposition}
\section{Percolating subcategories and complexes}\label{section:PercolatingSubcateoriesAndComplexes}

In \cite{HenrardvanRoosmalen19}, we considered the quotient $\EE / \AA$ of a deflation-exact category $\EE$ with respect to a deflation-percolating subcategory $\AA \subseteq \EE$. We showed that the quotient $\EE/\AA$ can be obtained by localizing $\EE$ with respect to a right multiplicative system $S_{\AA}$.  The quotient $\EE / \AA$ has a deflation-exact structure whose conflations are given (up to isomorphism) by the images of the conflations in $\EE$ under the quotient functor $Q\colon \EE \to \EE / \AA$.

The category $\Cb(\EE)$ of bounded cochain complexes in $\EE$ can be endowed with a deflation-exact structure by extending the structure of $\EE$ degree-wise. In this section we show that $\Cb(\AA)$ is a deflation-percolating subcategory of $\Cb(\EE)$, and that the natural functor $\Cb(\EE) /\Cb(\AA) \to \Cb(\EE/ \AA)$ is an exact equivalence.

\subsection{Localization at percolating subcategories}

We recall the basic notions and results on localizations of deflation-exact categories at deflation-percolating subcategories. We start with the definition of a percolating subcategory.

\begin{definition}\label{definition:GeneralPercolatingSubcategory}
	Let $\EE$ be a conflation category. A non-empty full subcategory $\AA$ of $\EE$ satisfying the following four axioms is called a \emph{deflation-percolating subcategory} of $\EE$.
	\begin{enumerate}[label=\textbf{P\arabic*},start=1]
		\item\label{P1} $\AA$ is a Serre subcategory, that is:
		\[\mbox{ If } A'\rightarrowtail A \twoheadrightarrow A'' \mbox{ is a conflation in $\EE$, then } A\in \Ob(\AA) \mbox{ if and only if } A',A''\in \Ob(\AA).\]
		\item\label{P2} For all morphisms $C\rightarrow A$ with $C \in \Ob(\EE)$ and $A\in \Ob(\AA)$, there exists a commutative diagram
		\[\xymatrix{
		A'\ar[rd] & \\
		C \ar@{->>}[u]\ar[r]& A\\
				}\] with $A'\in \Ob(\AA)$, and where $C \twoheadrightarrow A'$ is a deflation.
		\item\label{P3} For any composition $\xymatrix{X\ar@{>->}[r]^i & Y\ar[r]^t & T}$ which factors through $\AA$, there exists a commutative diagram 
		\[\xymatrix{
			X\ar@{>->}[r]^i\ar@{->>}[d]^f & Y\ar@{->>}[d]^{f'}\ar@/^/[rdd]^t &\\
			A\ar@{>->}[r]^{i'}\ar@/_/[rrd] & P\ar@{.>}[rd] &\\
			&& T
		}\] such that the square $XYAP$ is a pushout square.	
		\item\label{P4} For all maps $X\stackrel{f}{\rightarrow} Y$ that factor through $\AA$ and for all inflations $A\stackrel{i}{\inflation} X$ such that $f\circ i=0$, the induced map $\coker(i)\to Y$ factors through $\AA$.
	\end{enumerate}
	 An \emph{inflation-percolating subcategory} of a conflation category is defined dually.  Following the terminology from \cite{Schlichting04}, a non-empty full subcategory $\AA$ of a deflation-exact category $\EE$ satisfying axioms \ref{P1} and \ref{P2} (respectively axioms \ref{P1} and \ref{P2}$^{op}$ is called \emph{right filtering} (respectively \emph{left filtering}). If $\AA$ is a right filtering subcategory of $\CC$ such that the map $A'\rightarrow A$ in axiom \ref{P2} can be chosen as a monomorphism, we will call $\AA$ a \emph{strongly right filtering subcategory}. A deflation-percolating subcategory which is also strongly right filtering will be abbreviated to a \emph{strongly deflation-percolating subcategory}.
\end{definition}

The next definition constructs a right multiplicative system $S_{\AA}$. The terminology is based on \cite{Cardenas98,Schlichting04}.

\begin{definition}\label{definition:WeakIsomorphisms}
	Let $\EE$ be a conflation category and let $\AA$ be a non-empty full subcategory of $\EE$. 	
	\begin{enumerate}
		\item	An inflation $f\colon X\rightarrowtail Y$ in $\CC$ is called an \emph{$\AA^{-1}$-inflation} if its cokernel lies in $\AA$.
		\item A deflation $f\colon X\twoheadrightarrow Y$ in $\CC$ is called a \emph{$\AA^{-1}$-deflation} if its kernel lies in $\AA$.
		\item A morphism $f\colon X\rightarrow Y$ is called a \emph{weak $\AA^{-1}$-isomorphism}, or simply a \emph{weak isomorphism} whenever $\AA$ is implied, if it is a finite composition of $\AA^{-1}$-inflations and $\AA^{-1}$-deflations.
	\end{enumerate}
	The set of weak isomorphisms is denoted by $S_{\AA}$.
	
	Given a weak isomorphism $f$, the \emph{composition length} of $f$ is smallest number $n$ such that $f$ can be written as a composition of $n$ $\AA^{-1}$-inflations and $\AA^{-1}$-deflations.
\end{definition}

\begin{definition}\label{definition:Localization}
	Let $\EE$ be a deflation-exact category and let $\AA$ be a full deflation-exact subcategory. The quotient of $\EE$ by $\AA$ is a deflation-exact category $\EE/\AA$ together with an exact functor $Q\colon \EE\rightarrow\EE/\AA$, called the \emph{quotient functor}, such that, for any deflation-exact category $\DD$ and any exact functor $F\colon \EE\rightarrow \DD$ with $F(\AA)\cong 0$, there exists a unique exact functor $G\colon \EE/\AA\rightarrow \DD$ such that $F=G\circ Q$. 
\end{definition}

The next proposition states that $S_{\AA}$ is a right multiplicative system with a convenient strengthening axiom \ref{RMS2}.

\begin{proposition}\label{proposition:WeakIsomorphismsRMS}
	Let $\EE$ be a deflation-exact category and let $\AA$ be a deflation-percolating subcategory. The set $S_{\AA}$ of weak isomorphism is a right multiplicative system.
	
	If $\AA$ is a strongly deflation-percolating subcategory, the square in axiom \ref{RMS2} can be chosen as a pullback square.
\end{proposition}

The following theorem is the main theorem of \cite{HenrardvanRoosmalen19}.

\begin{theorem}\label{theorem:MainTheoremPartI}
	Let $\EE$ be a deflation-exact category and let $\AA$ be a deflation-percolating subcategory.  The category $\EE[S_{\AA}^{-1}]$ is a deflation-exact category where the conflation structure is induced by the localization functor $Q\colon \EE\rightarrow \EE[S_{\AA}^{-1}]$.  The category $\EE[S_{\AA}^{-1}]$ satisfies the universal property of the quotient $\EE/\AA$.
\end{theorem}

The next proposition gives a useful characterization of the kernel of the localization functor $Q$.

\begin{proposition}\label{proposition:ZeroMaps}
Let $\AA$ be a deflation-percolating subcategory of a deflation-exact category $\EE$. For any $f\colon X\rightarrow Y$ a map in $\EE$ with $Q(f)=0$, there exists an $\AA^{-1}$-inflation $s$ such that $f\circ s=0$.
\end{proposition}

\subsection{\texorpdfstring{Percolating subcategories in $\Cb(\EE)$}{Percolating subcategories in Cb(E)}}\label{sebsection:LiftingPercoaltingSubcategories}

In this subsection, we start with a deflation-percolating subcategory $\AA$ of a deflation-exact category $\EE$.  In proposition \ref{proposition:ComplexPercolating}, we will show that $\Cb(\AA)$ is a deflation-percolating subcategory of the deflation-exact category $\Cb(\EE)$.

For this, it will be convenient to introduce the category $\Seqb(\EE)$ of bounded sequences in $\EE$.  Let $\boldZ$ be the poset category of $\bZ, \leq$.  An object in $\Seq(\EE) \coloneqq \Fun(\boldZ, \EE)$ is determined by a sequence of objects $\{X^i\}_{i \in \bZ}$ together with a sequence of morphisms $\{\delta_X^i\colon X^i \to X^{i+1}\}.$  We consider the full subcategory $\Seqb(\EE)$ of $\Seq(\EE)$ given by those objects $X^\bullet$ for which $X^i = 0$ whenever $|i| \gg 0.$

\begin{remark}\label{remark:BoundedSequences}
If $\EE$ is deflation-exact, so is $\Seq(\EE) = \Fun(\boldZ, \EE)$; a sequence $X^\bullet \to Y^\bullet \to Z^\bullet$ is a conflation if and only if each sequence $X^i \inflation Y^i \deflation Z^i$ is a conflation.  As $\Seqb(\EE)$ and $\Cb(\EE)$ are extension-closed subcategories of $\Seq(\EE)$, they are deflation-exact as well.
\end{remark}

The following lemma is a strengthening of axiom \ref{P2}: we recover axiom \ref{P2} when $Y^\bullet \in \Seqb(\AA).$

\begin{lemma}\label{lemma:SumTrick}
	Let $\EE$ be a deflation-exact category and let $\AA$ be a full subcategory satisfying axiom \ref{P2}. Let $f^{\bullet}\colon X^{\bullet}\to Y^{\bullet}$ be a morphism in $\Seqb(\EE)$. If each $f^n$ factors through an object of $\AA$, then there exists a factorization $\xymatrix{X^{\bullet}\ar@{->>}[r]^{\gamma^{\bullet}} & B^{\bullet}\ar[r]^{\delta^{\bullet}} & Y^{\bullet}}$ of $f^{\bullet}$ with $B^{\bullet}\in \Seqb(\AA)$.
\end{lemma}

\begin{proof}
	Assume that for each $k\in \mathbb{Z}$, $f^k$ factors as $\xymatrix{X^k\ar[r]^{g^k} & A^k\ar[r]^{h^k} & Y^k}$ with $A^k\in \Ob(\AA)$. By axiom \ref{P2}, we may assume that each $g^k$ is a deflation.  As $X^{\bullet}$ and $Y^{\bullet}$ are (right) bounded, we can fix an $n\in \mathbb{Z}$ such that $X^{n+k}=Y^{n+k}=A^{n+k}=0$ for all $k>0$.
	
	We now inductively define a sequence $B^{\bullet}\in \Seqb(\AA)$ and a deflation $\gamma^{\bullet}\colon X^{\bullet}\to B^{\bullet}$ as follows.  Set $B^n \coloneqq A^n$ and $B^{n+k} \coloneqq 0$ for all $k>0$ and let $\gamma^n=g^n$. Assume that $B^{n-k+1}$ and $\gamma^{n-k+1}$ are already defined for some $k\geq 0$. By axiom \ref{P2}, the morphism $\begin{psmallmatrix}g^{n-k}\\ \gamma^{n-k+1} d_X^{n-k}\end{psmallmatrix}\colon X^{n-k}\to A^{n-k}\oplus B^{n-k+1}$ factors as \[\xymatrix{X^{n-k}\ar@{->>}[r]^{\gamma^{n-k}} & B^{n-k}\ar[r]^-{\begin{psmallmatrix}\alpha^{n-k}\\ \beta^{n-k}\end{psmallmatrix}} & A^{n-k}\oplus B^{n-k+1}}.\]
Set $\delta^{m}=h^m\alpha^m$, then $\delta^{\bullet}\colon B^{\bullet}\to Y^{\bullet}$ is a map in $\Seqb(\EE)$ and $f^{\bullet}=\delta^{\bullet}\circ \gamma^{\bullet}$, as required. \qedhere
\end{proof}

\begin{proposition}\label{proposition:ComplexPercolating}
	Let $\EE$ be a deflation-exact category and let $\AA \subseteq \EE$ be a deflation-percolating subcategory. 
	\begin{enumerate}
		\item The category $\Cb(\EE)$ has a natural deflation-exact structure defined degree-wise.
		\item The subcategory $\Cb(\AA)$ is a deflation-percolating subcategory of $\Cb(\EE)$.
	\end{enumerate}
\end{proposition}

\begin{proof}
	The first statement is straightforward to check (see also remark \ref{remark:BoundedSequences}). To show the second statement we need to verify axioms \ref{P1}, \ref{P2}, \ref{P3}, and \ref{P4}.
	Axiom \ref{P1} is automatic. Axioms \ref{P2} and \ref{P4} follow immediately from lemma \ref{lemma:SumTrick} (note that if $X^\bullet \in \Cb(\EE)$, then $B \in \Cb(\EE)$).  It remains to show axiom \ref{P3}. To that end, let $\xymatrix{X^{\bullet}\ar@{>->}[r]^{i^{\bullet}} & Y^{\bullet}\ar@{->>}[r]^{p^{\bullet}} & Z^{\bullet}}$ be a conflation in $\Cb(\EE)$ and $t^{\bullet}\colon Y^{\bullet}\to T^{\bullet}$ a map such that $t^{\bullet}\circ i^{\bullet}$ factors through $\Cb(\AA)$.  We will construct complexes $A^\bullet$ and $P^\bullet$ as in the statement of axiom \ref{P3} inductively.  As $Y^\bullet$ is bounded, we may, without loss of generality, assume that $Y^k = 0$ for $k \geq 1$.  We start by setting $A^{k} = P^{k} = 0$ for $k \geq 1$.  For $k \leq 0$, we can apply axiom \ref{P3} in $\EE$ inductively (for decreasing values of $k$) to obtain diagrams:
		\[\xymatrix{
		X^{k}\ar@{>->}[r]^{i^{k}}\ar@{->>}[d]^{\alpha^{k}} & Y^{k}\ar@{->>}[r]^{p^{k}}\ar@/^2pc/[rr]^{\begin{psmallmatrix} t^{k} \\ \beta^{k+1}\circ d_Y^k\end{psmallmatrix}}\ar@{->>}[d]^{\beta^{k}} & Z^n\ar@{=}[d] & {T^{k} \oplus P^{k+1}} \ar@{=}[d]\\
		A^{k}\ar@{>->}[r]^{\iota^k} & P^k\ar@{->>}[r]^{\rho^k}\ar@/_2pc/[rr]^{s^k} & Z^k & {T^k \oplus P^{k+1}}
	}\]
	Note that the composition $X^{k} \to Y^{k} \to Y^{k+1} \to P^{k+1}$ is equal to $X^{k} \to X^{k+1} \to Y^{k+1} \to P^{k+1}$ and hence factors through $A^{k+1} \in \AA$.  The differentials $d_P^i\colon P^i \to P^{i+1}$ are induced by the maps $s^i\colon P^i \to T^i \oplus P^{i+1}.$  Note that the above diagram shows that $\beta^{i+1}\circ d_Y^i = d_P^{i} \circ \beta^i$.  As $d^{i+1}_Y\circ d_Y^i = 0$ and every $\beta^i\colon Y^i \to P^i$ is a deflation, we see that the sequence $P^\bullet$ is a complex as well.  Similarly, one sees that $A^\bullet$ is a complex.  Finally, it follows from \cite[proposition 3.7]{HenrardvanRoosmalen19} that the left square in
		\[\xymatrix{
		X^\bullet\ar@{>->}[r]^{i^{\bullet}}\ar@{->>}[d]^{\alpha^{\bullet}} & Y^{\bullet}\ar@{->>}[r]^{p^{\bullet}}\ar@{->>}[d]^{\beta^{\bullet}} & Z^n\ar@{=}[d] \\
		A^{\bullet}\ar@{>->}[r]^{\iota^\bullet} & P^\bullet\ar@{->>}[r]^{\rho^\bullet} & Z^\bullet
	}\]
is both a pushout and a pullback square.
\end{proof}

\subsection{The equivalence \texorpdfstring{$\Cb(\EE/\AA)\simeq\Cb(\EE)/\Cb(\AA)$}{localizations of complexes}}

	We write $Q\colon \EE\rightarrow \EE/\AA$ for the quotient functor.  The functor $Q$ induces a natural exact functor $\widehat{Q}\colon \Cb(\EE)\rightarrow \Cb(\EE/\AA)$ such that $\widehat{Q}(\Cb(\AA))=0$.  By the universal property of the quotient $\Cb(\EE)/\Cb(\AA)$, there exists an induced functor $\Psi\colon \Cb(\EE)/\Cb(\AA)\rightarrow \Cb(\EE/\AA)$.  We will show that $\Psi$ is an equivalence of deflation-exact categories.

We start with some useful results when lifting complexes from $\Cb(\EE / \AA)$ to complexes in $\Cb(\EE).$  The first lemma uses the category of bounded sequences we considered in \S\ref{sebsection:LiftingPercoaltingSubcategories}.

\begin{lemma}\label{lemma:BoundedSequencesAndWeakIsomorphismsA}
Let $\AA \subseteq \EE$ be a deflation-percolating subcategory of a deflation-exact category $\EE$ and let $X\bullet \in \Seqb(\EE).$
\begin{enumerate}
	\item\label{enumerate:BoundedSequencesAndWeakIsomorphismsA1} Assume that there is, for every $i \in \bZ$, a given weak isomorphism $s^i\colon Y^i \stackrel{\sim}{\rightarrow} X^i.$  There is a $Z^\bullet \in \Seqb(\EE)$ and morphisms $t^i\colon Z^i \to Y^i$ such that $s^\bullet \circ t^\bullet\colon Z^\bullet \to X^\bullet$ is a morphism in $\Seqb(\EE).$
	\item\label{enumerate:BoundedSequencesAndWeakIsomorphismsA2} If $Q(X^\bullet) \in \Cb(\EE / \AA) \subseteq \Seqb(\EE / \AA),$ then there is a morphism $s^\bullet\colon Y^\bullet \stackrel{\sim}{\rightarrow} X^\bullet$ in $\Seqb(\EE)$ where each $s^i\colon Y^i \stackrel{\sim}{\inflation} X^i$ is a weak isomorphism and with $Y^\bullet \in \Cb(\EE) \subseteq \Seqb(\EE).$
\end{enumerate}
\end{lemma}

\begin{proof}
\begin{enumerate}
	\item As $Z^\bullet$ is bounded, we know that $Y^i \in \AA$ for $|i| \gg 0$.  Hence, $0 \stackrel{\sim}{\inflation} Y^i$ is a weak isomorphism, and we may assume that $Y^i = 0$ for $|i| \gg 0.$  The required property now follows easily from axiom \ref{RMS2}, applied consecutively from high indices to low indices (thus, from right to left).
	\item We apply lemma \ref{lemma:SumTrick} to the morphism $\epsilon^\bullet\colon X^\bullet \to \Sigma^2 X^\bullet$ given by $\epsilon^i = \delta^{i+1} \circ \delta^i\colon X^i \to X^{i+2}$ to obtain a factorization $\xymatrix{X^{\bullet}\ar@{->>}[r]^{\gamma^{\bullet}} & B^{\bullet}\ar[r] & \Sigma^2 X^{\bullet}}$ of $\epsilon^{\bullet}.$  We can take $Y^\bullet = \ker \gamma^\bullet.$  Indeed, as $s^{i+2} \circ \delta_Y^{i+1} \circ \delta_Y^i = \delta_X^{i+1} \circ \delta_X^i \circ s^i = 0$ and $s^{i+2}\colon Y^{i+2} \stackrel{\sim}{\inflation} X^{i+2}$ is a monomorphism, we have that $\delta_Y^{i+1} \circ \delta_Y^i = 0,$ as required. \qedhere
\end{enumerate}
\end{proof}

\begin{proposition}\label{proposition:BoundedSequencesAndWeakIsomorphismsB}
Let $\AA \subseteq \EE$ be a deflation-percolating subcategory of a deflation-exact category $\EE$.
\begin{enumerate}
	\item\label{enumerate:BoundedSequencesAndWeakIsomorphismsB1} Let $X^\bullet, Y^\bullet \in \Cb(\EE).$  Let $f^i\colon Y^i \to X^i$ be morphisms for all $i \in \bZ$ such that $Q(f^\bullet)\colon Q(Y^\bullet) \to Q(X^\bullet)$ is a morphism in $\Cb(\EE / \AA)$.  There is a morphism $s^\bullet\colon Z^\bullet \stackrel{\sim}{\rightarrow} Y^\bullet$ in $\Cb(\EE)$ such that every $s^i\colon Y^i \to X^i$ is a weak isomorphism and the composition $f^\bullet \circ s^\bullet\colon Z^\bullet \to X^\bullet$ is a morphism in $\Cb(\EE).$
	\item\label{enumerate:BoundedSequencesAndWeakIsomorphismsB2} Let $X^\bullet \in \Cb(\EE).$  Assume that there is, for every $i \in \bZ$, a given weak isomorphism $s^i\colon Y^i \stackrel{\sim}{\rightarrow} X^i.$  There is a $Z^\bullet \in \Cb(\EE)$ and morphisms $t^i\colon Z^i \to Y^i$ such that $s^\bullet \circ t^\bullet\colon Z^\bullet \to X^\bullet$ is a morphism in $\Cb(\EE).$
\end{enumerate}
\end{proposition}

\begin{proof}
\begin{enumerate}
	\item We consider the map $c^\bullet\colon Y^\bullet \to \Sigma X^\bullet$ in $\Cb(\EE)$ given by $c^i = f^i \circ \delta_X^i - \delta_Y^{i+1} \circ f^{i+1}\colon X^i \to Y^{i+1}.$  As each $c^i$ factors through $\AA$, lemma \ref{lemma:SumTrick} gives a factorization $\xymatrix{Y^{\bullet}\ar@{->>}[r]^{\gamma^{\bullet}} & B^{\bullet}\ar[r]^{\delta^{\bullet}} & \Sigma X^{\bullet}}$ of $c^{\bullet}.$  We can take $Z^\bullet = \ker \gamma^\bullet.$
	\item This follows from lemma \ref{lemma:BoundedSequencesAndWeakIsomorphismsA}. \qedhere
\end{enumerate}
\end{proof}

\begin{proposition}\label{proposition:MainResultComplexLevel}
	Let $\EE$ be a deflation-exact category and let $\AA$ be a deflation-percolating subcategory. The natural functor $\Cb(\EE)\to \Cb(\EE/\AA)$ factors through an equivalence $\Psi\colon \Cb(\EE)/\Cb(\AA)\xrightarrow{\sim} \Cb(\EE/\AA)$.
\end{proposition}

\begin{proof}	
	We first show that $\Psi$ is essentially surjective. Let $X^{\bullet}\in \Cb(\EE/\AA)$. The complex $X^{\bullet}$ has the following form:
	\[\xymatrix{
	  &     &\ar[dl]_{\rotatebox{45}{$\sim$}}\ar[dr] &         &\ar[dl]_{\rotatebox{45}{$\sim$}}\ar[dr] &         &\ar[dl]_{\rotatebox{45}{$\sim$}} &  \ar[dr]&         &\\
	0\ar[r] & X^n & & X^{n+1} & & X^{n+2} \ar@{.}[rrr]&  &  & X^{n+m}\ar[r] & 0
	}\]
	Repeatedly applying axiom \ref{RMS2}, we obtain a commutative diagram:
	\[\xymatrix@!0@C=4em@R=3em{
	&&\ar@{.}[dr]&&&&&&\\
	&\ar@{.}[dr]&& Y^{n+m-3}\ar[dl]_{\rotatebox{45}{$\sim$}}\ar[dr] && && &\\
	\ar@{.}[dr]&& \ar[dl]_{\rotatebox{45}{$\sim$}}\ar[dr] && Y^{n+m-2}\ar[dr]\ar[dl]_{\rotatebox{45}{$\sim$}}&& &\\
	&\ar[dl]_{\rotatebox{45}{$\sim$}}\ar[dr]&&\ar[dl]_{\rotatebox{45}{$\sim$}}\ar[dr] && Y^{n+m-1}\ar[dr]\ar[dl]_{\rotatebox{45}{$\sim$}}&\\
	X^{n+m-3} & & X^{n+m-2} && X^{n+m-1} && X^{n+m}\ar[r] & 0\\	
	}\]
	Note that $Y^\bullet$ need not be a cochain complex in $\Cb(\EE)$, but that $Q(Y^\bullet) \cong Q(X^\bullet) \in \Cb(\EE / \AA).$  It now follows from lemma \ref{lemma:BoundedSequencesAndWeakIsomorphismsA}.\ref{enumerate:BoundedSequencesAndWeakIsomorphismsA2} that $Q\colon \Cb(\EE) \to \Cb(\EE / \AA)$ is essentially surjective.
	
	To show that $\Psi$ is full, let $\phi^{\bullet}\colon Q(X^{\bullet})\rightarrow Q(Y^{\bullet})$ be an arbitrary map in $\Cb(\EE/\AA)$.  We may represent $\phi^{\bullet}=(\phi^i,s^i)_{i\in \mathbb{Z}}$ in $\EE$ by the following diagram:
	\[\xymatrix{
		\dots \ar[r]& X^{n-1}\ar[r] & X^n\ar[r] & X^{n+1}\ar[r] & \dots \\
		 & L^{n-1}\ar[u]^{\rotatebox{90}{$\sim$}}_{s^{n-1}}\ar[d]_{\phi^{n-1}} & L^n\ar[u]^{\rotatebox{90}{$\sim$}}_{s^{n}}\ar[d]_{\phi^{n}} & L^{n+1}\ar[u]^{\rotatebox{90}{$\sim$}}_{s^{n+1}}\ar[d]_{\phi^{n+1}} & \\
		\dots\ar[r] & Y^{n-1}\ar[r] & Y^n\ar[r] & Y^{n+1}\ar[r] &\dots 
	}\]
	Applying proposition \ref{proposition:BoundedSequencesAndWeakIsomorphismsB}.\ref{enumerate:BoundedSequencesAndWeakIsomorphismsB2} on the morphisms $s^i\colon L^i \to X^i$, we may assume that $s\colon L^\bullet \to X^\bullet$ is a morphism of complexes.  Using proposition \ref{proposition:BoundedSequencesAndWeakIsomorphismsB}.\ref{enumerate:BoundedSequencesAndWeakIsomorphismsB1}, we may furthermore assume that $\phi^\bullet\colon L^\bullet \to Y^\bullet$ is a morphism of cochain complexes.  This shows that $\Psi$ is full.
	
	It remains to show that $\Psi$ is faithful. To that end, let $\phi^{\bullet}\colon X^{\bullet}\rightarrow Y^{\bullet}$ be a map in $\Cb(\EE)$ such that $\Psi(\phi^{\bullet})=0$.  It follows that each $\phi^i$ factors through an object of $\AA$. By lemma \ref{lemma:SumTrick}, there exists a complex in $\Cb(\AA)$ such that $\phi^{\bullet}$ factors through it. It follows that $\phi^{\bullet}$ is zero in $\Cb(\EE)/\Cb(\AA)$. This concludes the proof.
\end{proof}

\begin{remark}
	Note that the results in this section can be extended to right bounded complexes, i.e.~$\Cm(\EE/\AA)\simeq \Cm(\EE)/\Cm(\AA)$. For inflation-exact categories, this dualizes to left bounded complexes.
\end{remark}
\section{The Verdier localization sequence}\label{section:MainResult}

Given a deflation-exact category $\EE$ and a deflation-percolating subcategory $\AA\subseteq \EE$, we study the derived category $\Db(\EE/\AA)$. We show that, under mild assumptions, the sequence 
\[\DAb(\EE)\to \Db(\EE)\to \Db(\EE/\AA)\]
is a Verdier localization sequence.

\subsection{Some results on localizations}

The following proposition is an adaptation of \cite[proposition~I.3.4]{Hartshorne66} (see also \cite[proposition~I.1.3.(iv)]{GabrielZisman67}).

\begin{proposition}\label{proposition:2Universal}
Let $L\colon \CC \to \CC[S^{-1}]$ be a localization.  For each category $\DD$, there is a fully faithful functor 
\[- \circ L\colon \Fun(\CC[S^{-1}], \DD) \to \Fun(\CC, \DD)\]
whose image consists of those functors $F\colon \CC \to \DD$ for which $F(s)$ is invertible (for all $s \in S$).
\end{proposition}

\begin{proof}
It is clear that the functor $- \circ L\colon \Fun(\CC[S^{-1}], \DD) \to \Fun(\CC, \DD)$ is faithful.  Moreover, it follows from the universal property that the image is as described.  We only need to show that it is full.

Let $F,G\colon \CC \to \DD$ be functors in the image of $-\circ L$ (thus, $F = F' \circ L$ and $G = G' \circ L$ for some $F,G \colon \CC[S^{-1}] \to \DD$) and $\eta\colon F \Rightarrow G$ a natural transformation.  This describes a functor $N\colon \CC \to \DD^\to$ (where $\DD^\to$ is the arrow category of $\DD$) by $C \mapsto \left[F(C) \stackrel{\eta_C}{\rightarrow} G(C)\right]$.  As $F$ and $G$ map every element of $S$ to an invertible element, so does $N$.  This implies that $N$ factors as $\CC \stackrel{L}{\rightarrow} \CC[S^{-1}] \stackrel{M}{\rightarrow} \DD^{\to}$.  The functor $M$ then gives the required natural transformation $F' \Rightarrow G'$.
\end{proof}

The following lemma is well known.  We provide a proof for the benefit of the reader.

\begin{lemma}\label{lemma:InducedLocalization}\label{lemma:CompositionOfLocalizations}
Let $\CC$ be any category.
\begin{enumerate}
  \item Let $S,T\subset \Mor(\mathcal{C})$. Let $N\colon \CC \to \CC[S^{-1}]$ and $M\colon \CC \to \CC[T^{-1}]$ be the corresponding localizations.  If $F\colon \mathcal{C}[S^{-1}] \rightarrow \mathcal{C}[T^{-1}]$ is a functor such that 
	\[\xymatrix{
		\mathcal{C}\ar[d]^N\ar[r]^{M} & \mathcal{C}[T^{-1}]\\
		\mathcal{C}[S^{-1}]\ar[ru]_F & 
	}\] commutes, then $F$ is a localization itself.
	\item Let $S\subset \Mor(\mathcal{C})$ be a right multiplicative system and $U \subseteq \Mor(\CC[S^{-1}])$ any set of morphisms.
	Let $F\colon \mathcal{C}\rightarrow \mathcal{C}[S^{-1}]$ and $G\colon \mathcal{C}[S^{-1}] \rightarrow (\mathcal{C}[S^{-1}])[U^{-1}]$ be the localizations functors, then the composition $G\circ F$ is a localization as well.
	\item Let $F\colon \CC \to \CC[S^{-1}]$ be a localization.  Let $\Phi\colon \DD \to \CC$ be any functor.  Let $T = \{ f \in \Mor(\DD) \mid \mbox{$\Phi \circ F(f)$ is invertible} \}.$  If $\Phi$ is an equivalence, so is the natural functor $G\colon \DD[T^{-1}] \to \CC[T^{-1}].$ 
\end{enumerate}
\end{lemma}

\begin{proof}
\begin{enumerate}
\item Define $U = \{u \in \Mor(S^{-1}\mathcal{C}) \mid \mbox{$F(u)$ invertible}\}.$  We will verify that $F\colon S^{-1}\mathcal{C}\rightarrow T^{-1}\mathcal{C}$ is a localization with respect to $U$ by showing it satisfies the corresponding universal property.
	
	Therefore, let $G\colon \mathcal{C}[S^{-1}]\rightarrow \mathcal{D}$ be a functor such that $G(u)$ is invertible for all $u\in U$. As $M$ is a localization with respect to $T$, we know that $M(t)=F (N(t))$ is invertible, for all $t \in T$. By definition, we have that $N(t)\in U$ and hence $G\circ N(t)$ is invertible as well.  By the universal property of $\mathcal{C}[T^{-1}]$, there exists a unique functor $L\colon \mathcal{C}[T^{-1}] \rightarrow D$ such that $L\circ M=G\circ N$.  As then $L\circ F\circ N=G\circ N$, the universality of $N$ implies that $L \circ F = G$.  This finishes the proof.
	
\item 	Define $W = \{w \in \Mor(\CC) \mid \mbox{$GF(u)$ invertible} \}$. We will show that $G\circ F$ satisfies the universal property of the localization with respect to $W$. Let $H\colon \mathcal{C}\rightarrow \mathcal{D}$ be a functor such that $H(w)$ is invertible for any $w\in W$. Obviously $S\subseteq W$, so that the universal property of $\mathcal{C}[S^{-1}]$ implies that there exists a unique functor $K\colon S^{-1}\mathcal{C}$ such that $H=K\circ F$.
	
	We now show that $K$ maps every morphism in $U$ to an invertible morphism in $\DD$. Let $u\in U$.  Since $S$ is a right multiplicative system, we can write $u=fs^{-1}$ with $s\in S$ and $f$ a map in $\mathcal{C}$. Since $G(u)$ and $G(s)$ are invertible, $G(f) = G(u)G(s)$ is invertible as well. Using that $F(f)=f$, we infer that $G\circ F(f)$ is invertible and hence $f\in W$.  Thus, $H(f)$ is invertible, and it follows that $K(u)$ is invertible.  By the universal property of the localization $G\colon \mathcal{C}[S^{-1}]\rightarrow U^{-1}(\mathcal{C}[S^{-1}])$, there exists a unique functor $L\colon (\mathcal{C}[S^{-1}])[U^{-1}] \to \DD$ such that $K=L\circ G$. Putting everything together we find that $H=L\circ (G\circ F)$. This establishes the required properties.
	\item This follows easily from proposition \ref{proposition:2Universal}. \qedhere
	\end{enumerate}
\end{proof}

\subsection{Proof of main theorem}

The following proposition was proven in \cite[proposition 1.3]{Rickard89} (see also \cite[criterion 1.3]{Neeman90}).

\begin{proposition}\label{proposition:CriterionRickard}
Let $\CC$ be a triangulated category and let $\TT \subseteq \CC$ be a full subcategory.  The following are equivalent:
\begin{enumerate}
	\item $\TT$ is a thick subcategory of $\CC$,
	\item for any morphism $f\colon X \to Y$ in $\CC$ with $\cone(f) \in \TT$, if $f$ factors through $\TT$ then $X,Y \in \TT.$
\end{enumerate}
\end{proposition}

In the statement of the main theorem we use the following definition.

\begin{definition}
	Let $\AA$ be a deflation-percolating subcategory of a deflation-exact category $\EE$ satisfying axiom \ref{R0*}.  The triangulated subcategory of $\Db(\EE)$ generated by the image of $\AA$ under the canonical embedding is denoted by $\DAb(\EE)$.
\end{definition}

\begin{proposition}\label{proposition:WeakIsoYieldsACones} 
	Let $\EE$ be a deflation-exact category satisfying axiom \ref{R0*} and let $\AA\subseteq \EE$ be a deflation-percolating subcategory.  Let $f^\bullet\colon X^\bullet \to Y^\bullet$ be a morphism in $\Cb(\EE).$
	\begin{enumerate}
		\item If $f^\bullet$ is a weak isomorphism (i.e.~$f^{\bullet}\in S_{\Cb(\AA)}$), then $\cone(f^{\bullet}) \in \DAb(\EE)$.
		\item If $f^\bullet$ is a weak isomorphism, then $X^{\bullet}\in \DAb(\EE)$ if and only if $Y^{\bullet}\in \DAb(\EE)$.
		\item If $Q(f^\bullet)$ is an isomorphism in $\Cb(\EE/ \AA)$, then $X^\bullet \in \langle \DAb(\EE) \rangle_{\textrm{thick}}$ if and only if $Y^\bullet \in \langle \DAb(\EE) \rangle_{\textrm{thick}}.$
	\end{enumerate}
\end{proposition}

\begin{proof}	
	\begin{enumerate}
		\item In any triangulated category, the cone of a composition $g \circ f$ is an extension of $\cone(g)$ by $\cone(f)$ (this is an immediate corollary of \cite[proposition~1.4.6]{Neeman01}). Since $f^{\bullet}$ factors as a composition of $\Cb(\AA)^{-1}$-deflations and $\Cb(\AA)^{-1}$-inflations, it suffices to show that the cones of $\Cb(\AA)^{-1}$-deflations (resp. $\Cb(\AA)^{-1}$-inflations) belong to $\Db_{\AA}(\EE)$. The latter follows immediately from proposition \ref{proposition:ConflationsYieldTriangles} and the definition of $\Cb(\AA)^{-1}$-deflations (resp. $\Cb(\AA)^{-1}$-inflations).
		\item This statement follows from the previous statement.
		\item Assume first that $X^\bullet \in \langle \DAb(\EE) \rangle_{\textrm{thick}}.$  Using that $\Cb(\EE) / \Cb(\AA) \simeq \Cb(\EE / \AA)$, we know that $f^\bullet$ descends to an isomorphism in $\Cb(\EE) / \Cb(\AA).$  As $S_{\Cb(\AA)}$ is a right multiplicative system, there is a morphism $g^\bullet\colon Z^\bullet \to X^\bullet$ such that $f^\bullet \circ g^\bullet \in S_{\Cb(\AA)}.$  Hence, $\cone(g^\bullet \circ f^\bullet) \in \DAb(\EE).$  It now follows from proposition \ref{proposition:CriterionRickard} that $Y^\bullet \in \langle \DAb(\EE) \rangle_{\textrm{thick}}.$
		
		Assume now that $Y^\bullet \in \langle \DAb(\EE) \rangle_{\textrm{thick}}.$  As before, there is a morphism $g^\bullet\colon Z^\bullet \to X^\bullet$ such that $f^\bullet \circ g^\bullet \in S_{\Cb(\AA)}.$  Note that $Q(g^\bullet)$ is an isomorphism. Moreover, as both $Y^\bullet, \cone(f^\bullet \circ g^\bullet) \in \langle \DAb(\EE) \rangle_{\textrm{thick}}$, we find that $Z^\bullet \in \langle \DAb(\EE) \rangle_{\textrm{thick}}.$  We have reduced this case to the previous case.  \qedhere
	\end{enumerate}
\end{proof}

We now come to the the main theorem; our proof closely resembles the proofs given in \cite[theorem 3.2]{Miyachi91} and \cite[proposition 2.6]{Schlichting04}.

\begin{theorem}\label{theorem:MainTheorem}
	Let $\EE$ be a deflation-exact category and let $\AA$ be a deflation-percolating subcategory.
	\begin{enumerate}
		\item The derived quotient functor $\Db(\EE)\to \Db(\EE/\AA)$ is a Verdier localization.
		\item If additionally $\EE$ satisfies axiom \ref{R0*}, then the sequence 
		\[\DAb(\EE)\to \Db(\EE)\to \Db(\EE/\AA)\]
		is a Verdier localization sequence.
	\end{enumerate}
\end{theorem}

\begin{proof}
	\begin{enumerate}
		\item Consider the natural diagram
	\[\xymatrix{
		\Cb(\EE)\ar[r]\ar[d] & \Cb(\EE)/\Cb(\AA) \simeq \Cb(\EE/\AA) \ar[d]\\
		\Kb(\EE)\ar@{.>}[r]^G & \Kb(\EE/\AA) 
	}\]  where $G$ is a triangle functor.  Recall from proposition \ref{proposition:ComplexPercolating} that the functor $\Cb(\EE) \to \Cb(\EE)/\Cb(\AA)$ is a localization with respect to a right multiplicative set and that, for an additive category $\CC$, the quotient $\Cb(\CC) \to \Kb(\CC)$ is a localization functor (with respect to the Hurewicz model structure, see \cite{GolasinskiGromadzki82,BarthelMayRiehl14}).  It follows from lemma \ref{lemma:InducedLocalization} that the triangle functor $G$ is a Verdier localization.

Similarly, it follows from the diagram
	\[\xymatrix{
		\Kb(\EE)\ar[r]^G\ar[d] & \Kb(\EE/\AA)\ar[d]\\
		\Db(\EE)\ar@{.>}[r]^F & \Db(\EE/\AA) 
	}\] that $F$ is a Verdier localization.
	
	\item We need to show that $\ker(F) = \langle \DAb(\EE) \rangle_{\textrm{thick}}$.	It is clear that $\DAb(\EE) \subseteq \ker(F)$.  As $\ker(F)$ is a thick subcategory, it follows that $\langle \DAb(\EE) \rangle_{\textrm{thick}} \subseteq \ker(F)$. It remains to show the other inclusion.
	
	Let $X^{\bullet}\in \ker(F)$. By the above diagram we find that $G(X^{\bullet})\in \langle\Acb(\EE/\AA)\rangle_{\textrm{thick}}$. It follows from lemma \ref{lemma:CoretractionInHomotopyCategory} that $X^{\bullet}$ is a direct summand in $\Kb(\EE/\AA)$ of an acyclic complex in $\Cb(\EE/\AA)$.  We split the proof into two parts.  Assume first that $X^{\bullet} \in \AcbC(\EE/\AA)$.  We proceed by induction on the width $n$ of the support of $X^{\bullet}\in \Cb(\EE)$.
	
	If $X^{\bullet}$ is a stalk complex with $X$ in degree $k$, we find that $X\cong 0$ in $\EE/\AA$ as $X^{\bullet}\in \AcbC(\EE/\AA)$. It follows from proposition \ref{proposition:ZeroMaps} that $X\in \AA$. Thus $X^{\bullet}\in \DAb(\EE)$.
	
	Assume that $X^{\bullet}$ has support of width $n\geq 2$.  Without loss of generality assume that $X^{\bullet}\in \Conen(\EE)$.  We claim that there is a complex $Z^\bullet\in \langle \DAb(\EE) \rangle_{\textrm{thick}}$ together with a zig-zag of morphisms $X^\bullet \stackrel{\alpha^\bullet}{\leftarrow} W^\bullet \stackrel{\beta^\bullet}{\rightarrow} Y^\bullet \stackrel{\gamma^\bullet}{\leftarrow} Z^\bullet$ where $\alpha^\bullet$ is a weak isomorphism, $\beta^\bullet$ becomes an isomorphism under $Q$, and $\gamma^\bullet$ is a quasi-isomorphism.  It then follows from proposition \ref{proposition:WeakIsoYieldsACones} that $X^\bullet \in \langle \DAb(\EE) \rangle_{\textrm{thick}}$ as well.
	
	We now prove the above claim.  Since $X^{\bullet}\in \AcbC(\EE/\AA)$, the differential $d_X^{n-1}$ is a deflation in $\EE/\AA$. By construction, there is a diagram in $\EE$:
	\[\xymatrix{
	X^{n-1}\ar[r] & X^n\\
	\ar[u]^{\rotatebox{90}{$\sim$}}_{s^{n-1}}U^{n-1}\ar[d]^{f^{n-1}} &\ar[u]^{\rotatebox{90}{$\sim$}}_{s^n} U^n\ar[d]^{f^n} \\
	Y^{n-1}\ar@{->>}[r] & Y^n
	}\]
	corresponding to a commutative diagram in $\EE / \AA$.  Applying proposition \ref{proposition:BoundedSequencesAndWeakIsomorphismsB}.\ref{enumerate:BoundedSequencesAndWeakIsomorphismsB2} to the complex $X^\bullet$ and the morphisms $s^{n-1}\colon U^{n-1} \stackrel{\sim}{\rightarrow} X^{n-1}$, $s^n\colon U^{n} \stackrel{\sim}{\rightarrow} X^{n}$, and identities in the other positions, we obtain a morphism $t^\bullet\colon V^\bullet \stackrel{\sim}{\rightarrow} X^\bullet$ in $\Conen(\EE)$ such that $t^{n}$ and $t^{n-1}$ factor through $s^{n}$ and $s^{n-1}$, respectively.  The induced square
	\[\xymatrix{
	V^{n-1} \ar[r] \ar[d] & V^n \ar[d] \\
	Y^{n-1} \ar@{->>}[r] & Y^n}\]
need not commute, but it becomes a commutative square under application of $Q$.  Hence, there is a morphism $K \stackrel{\sim}{\rightarrow} V^{n-1}$ such that the composition $K \to V^{n-1} \to V^{n} \to Y^{n}$ is equal to $K \to V^{n-1} \to Y^{n-1} \deflation Y^{n}.$  Applying proposition \ref{proposition:BoundedSequencesAndWeakIsomorphismsB}.\ref{enumerate:BoundedSequencesAndWeakIsomorphismsB2} to the complex $V^\bullet$ and the morphism $K \stackrel{\sim}{\rightarrow} V^{n-1}$ (and the identity morphisms in the other positions), we obtain a morphism $W^\bullet \stackrel{\sim}{\rightarrow} V^\bullet$ such that the diagram
	\[\xymatrix{
	W^{n-1} \ar[r] \ar[d] & W^n \ar[d] \\
	Y^{n-1} \ar@{->>}[r] & Y^n}\]
commutes.  We extend the morphism $Y^{n-1} \deflation Y^n$ to a complex $Y^\bullet$ by setting $Y^i = W^i$ for $i \neq n, n-1$ with the obvious differentials.  There is then a morphism $W^\bullet \to Y^\bullet$ (this map is the identity outside of the positions $n$ and $n-1$) which becomes an isomorphism under $Q$.  We now set $Z^\bullet \coloneqq \tau^{\leq n-1} Y^\bullet.$

We obtain the aforementioned zig-zag $X^\bullet {\leftarrow} W^\bullet {\rightarrow} Y^\bullet {\leftarrow} Z^\bullet.$  Note that $Z^\bullet \in \Conenmone(\EE)$ and that $Z^\bullet$ becomes an acyclic complex in $\EE / \AA$ (it is the truncation of $Y^\bullet$, which is isomorphic to $W^\bullet$ and to $X^\bullet$ in $\EE / \AA$, which is acyclic).  The induction hypothesis shows that $Z^\bullet\in \langle \DAb(\EE) \rangle_{\textrm{thick}}$.  We conclude that $X^\bullet\in \langle \DAb(\EE) \rangle_{\textrm{thick}}$, as required.
	
	We now consider the more general case where $X^\bullet \in \Cb(\EE)$ is a direct summand of an acyclic complex $V \in \Cb(\EE / \AA).$  By proposition \ref{proposition:MainResultComplexLevel}, $\Cb(\EE)/\Cb(\AA)\simeq \Cb(\EE/\AA)$ and thus $X^{\bullet}$ is a direct summand in $\Cb(\EE)/\Cb(\AA)$ of $V^{\bullet}\in \Cb(\EE)$.  Let $V^\bullet \stackrel{\sim}{\leftarrow} W^\bullet \rightarrow X^\bullet$ and $X^\bullet \stackrel{\sim}{\leftarrow} Y^\bullet \rightarrow V^\bullet$ be roofs in $\Cb(\EE)$ representing the retraction and a corresponding section.  The following roof represents the identity on $X^{\bullet}$ in $\Cb(\EE)/\Cb(\AA)$:
	\[\xymatrix@C=0.5em@R=0.3em{
	&&Z^{\bullet}\ar[ld]_{\rotatebox{45}{$\sim$}}\ar[rd] & &\\
	& Y^{\bullet}\ar[ld]_{\rotatebox{45}{$\sim$}}\ar[rd] &&W^{\bullet}\ar[ld]_{\rotatebox{45}{$\sim$}}\ar[rd]&\\
	X^{\bullet} &&V^{\bullet} &&X^{\bullet}
	}\]
	and hence we may assume that the left leg and the right leg compose to the same morphism $f\colon Z \stackrel{\sim}{\rightarrow} X$.  We have shown that $V^\bullet \in \langle \DAb(\EE) \rangle_{\textrm{thick}}$ and, hence, $f$ factors through $W^\bullet \in \langle \DAb(\EE) \rangle_{\textrm{thick}}$.  By proposition \ref{proposition:WeakIsoYieldsACones}, we know that $\cone(f) \in \DAb(\EE).$  It now follows from proposition \ref{proposition:CriterionRickard} that $Z^\bullet, X^\bullet \in \langle\DAb(\EE)\rangle_{\textrm{thick}},$ as required. \qedhere
	\end{enumerate}
\end{proof}

\subsection{Application to locally compact modules}\label{subsection:ApplicationLocallyCompact}

Let $\LCA$ be the category of locally compact (Hausdorff) abelian groups. Let $R$ be a unital ring endowed with the discrete topology. Denote by $R-\LC$ the category of locally compact left $R$-modules.  We furthermore write $R-\LC_D$ and $R-\LC_C$ for the full subcategories of discrete $R$-modules and compact $R$-modules, respectively. Note that $\mathbb{Z}-\LC$ is simply the category $\LCA$. We use similar notations for locally compact right $R$-modules. We recall the following proposition.

\begin{proposition}\makeatletter
\hyper@anchor{\@currentHref}%
\makeatother\label{proposition:LocallyCompactPercolating}
	\begin{enumerate}
		\item The categories $R-\LC$ and $\LC-R$ are quasi-abelian categories.
		\item The standard Pontryagin duality can be extended to a duality
		\[\mathbb{D}\colon R-\LC\rightarrow \LC-R \mbox{ and } \mathbb{D}'\colon \LC-R\rightarrow R-\LC\]
		which interchanges the discrete and compact modules.
		\item The category $R-\LC_D$ is a deflation-percolating subcategory of $R-\LC$.
		\item The category $R-\LC_C$ is an inflation-percolating subcategory of $R-\LC$.
	\end{enumerate}
\end{proposition}

\begin{proof}
The first two statements are well-known (see for example \cite{Levin73}).  For the last two statements, we refer to \cite[proposition~8.28]{HenrardvanRoosmalen19}.
\end{proof}

\begin{remark}
	\begin{enumerate}
		\item The exact structure on $R-\LC$ is described as follows: closed injections are inflations and open surjections are deflations.
		\item As noted in \cite[example~4]{Braunling20}, the category $R-\LC_C$ is in general not left s-filtering, nor is $R-\LC_C$ a right filtering subcategory of $R-\LC$.  It follows that one cannot use the localization theories of \cite{Schlichting04,Cardenas98} to describe the quotient $R-\LC/R-\LC_C$. The above proposition shows that this quotient satisfies the conditions of theorems \ref{theorem:MainTheoremPartI} and \ref{theorem:MainTheorem}.
	\end{enumerate}
\end{remark}

The following proposition is \cite[theorem~12.1.b]{Keller96}.

\begin{proposition}\label{proposition:DesiredVerdierLocalizationSequence}
	Let $\EE$ be an exact category and let $\AA\subseteq \EE$ be a fully exact subcategory. Suppose that either of the following conditions hold:
	\begin{enumerate}
		\myitem{\textbf{C2}}\label{C2} For each conflation $A\rightarrowtail E \twoheadrightarrow E'$ of $\EE$ with $A\in \AA$, there is a commutative diagram
		\[\xymatrix{
			A\ar@{>->}[r]\ar@{=}[d] & E\ar@{->>}[r]\ar[d] & E'\ar[d]\\
			A\ar@{>->}[r] & A'\ar@{->>}[r] & A''
		}\] whose second row is a conflation in $\AA$.
		\myitem{$\textbf{C2}^{\text{op}}$}\label{C2op} For each conflation $E'\rightarrowtail E \twoheadrightarrow A$ of $\EE$ with $A\in \AA$, there is a commutative diagram 
		\[\xymatrix{
			E'\ar@{>->}[r] & E\ar@{->>}[r] & A\ar@{=}[d]\\
			A''\ar@{>->}[r] \ar[u]& A'\ar@{->>}[r]\ar[u] & A
		}\] whose second row is a conflation in $\AA$.
	\end{enumerate}
	Then the canonical functor $\Db(\AA)\to \Db(\EE)$ is fully faithful. In particular, $\DAb(\EE)= \Db(\AA)$.
\end{proposition}

\begin{remark}
	Conditions \ref{C2}/\ref{C2op} are closely related to right/left s-filtering subcategories (see \cite[proposition~A.2]{BraunlingGroechenigWolfson16}).
\end{remark}

\begin{corollary}\label{corollary:VerdierSequencesLocallyCompactModules}
	The following are Verdier localization sequences:
	\[\Db(R-\LC_D)\rightarrow \Db(R-\LC)\rightarrow \Db(R-\LC/R-\LC_D)\]
	and 
	\[\Db(R-\LC_C)\rightarrow \Db(R-\LC)\rightarrow \Db(R-\LC/R-\LC_C).\]
\end{corollary}

\begin{proof}
By proposition \ref{proposition:LocallyCompactPercolating}, $R-\LC_D \subseteq R-\LC$ satisfies the conditions of theorem \ref{theorem:MainTheorem}. By proposition \ref{proposition:DesiredVerdierLocalizationSequence}, it suffices to check whether condition \ref{C2op} holds. Note that for each $D\in R-\LC_D$, there is a deflation $R^{\oplus I} \stackrel{p}{\deflation} D$ for some index set $I$ (here, $R^{\oplus I}$ has the discrete topology). As $R$ is projective in $R-\LC$, we can use the lifting property to obtain condition \ref{C2op}.

By Pontryagin duality, we conclude that $\Db(R-\LC_C)\rightarrow \Db(R-\LC)\rightarrow \Db(R-\LC/R-\LC_C)$ is a Verdier localization as well.
\end{proof}

Let $R-\LC_{D, \textrm{noeth}}$ be the full subcategory of $R-\LC_{D}$ consisting of those objects whose underlying modules are noetherian.  It is shown in \cite{HenrardvanRoosmalen19} that $R-\LC_{D, \textrm{noeth}}$ is a deflation-percolating subcategory of $R-\LC$.  Similar to the above results, we find the following proposition.

\begin{proposition}\label{proposition:FGD}
There is a Verdier localization sequence:
	\[\Db(R-\LC_{D, \textrm{noeth}})\rightarrow \Db(R-\LC)\rightarrow \Db(R-\LC/R-\LC_{D, \textrm{noeth}})\]
\end{proposition}
\section{\texorpdfstring{Localizations of $\LCA$ are exact}{Localizations of LCA are exact}}

In general, the localization $\EE[S_\AA^{-1}]$ where $\EE$ is an exact category and $\AA$ is a deflation-percolating subcategory need not be exact.  In this section we provide a sufficient condition (theorem \ref{theorem:ExactLocalization}) such that the localization is (two-sided) exact.  As an application, we find that the quotients of $\LCA$ by left or right percolating subcategories (such as $\LCA/\LCA_D$ and $\LCA/\LCA_C$) are exact.

\subsection{Extension-closed subcategories of derived categories}  Our aim is to prove proposition \ref{proposition:ExactMeansExtensionClosed} below, which says that a deflation-exact category satisfying axiom \ref{R3} is exact if and only if it corresponds to an extension-closed subcategory of its derived category.

\begin{lemma}\label{lemma:RetractOfAcyclic}
Let $\EE$ be a deflation-exact category satisfying axiom \ref{R3} and let $C^\bullet \in \C^{[0,2]}(\EE)$ be a complex.  If $C^\bullet \in \langle \Ac(\EE) \rangle_{\textrm{thick}}$, then $C^\bullet$ is a conflation.
\end{lemma}

\begin{proof}
It follows from lemma \ref{lemma:CoretractionInHomotopyCategory} that $C^\bullet$ is a direct summand of an acyclic complex $A^\bullet$ in $\C(\EE)$.  It is straightforward to verify that the coretraction $C^\bullet \to A^\bullet$ induces a coretraction $C^\bullet \to \tau^{\geq 2} (\tau^{\leq 2} A^\bullet).$  (Note that the complex $\tau^{\geq 2} (\tau^{\leq 2} A^\bullet)$ is given by the conflation $\ker d^1_A \inflation A^1 \deflation \ker d^2_A$).  It now follows from proposition \ref{proposition:RetractConflation} that $C^\bullet$ is a conflation.
\end{proof}

\begin{proposition}\label{proposition:MeaningR3}
Let $\EE$ be a deflation-exact category and let $i \colon \EE \to \Db (\EE)$ the canonical embedding. The following are equivalent:
\begin{enumerate}
	\item\label{item:R3} The category $\EE$ satisfies axiom \ref{R3}.
	\item\label{item:ConflationsAndTriangles} A sequence $X \stackrel{f}{\to} Y \stackrel{g}{\to} Z$ in $\EE$ is a conflation if and only if there is a triangle $i(X) \stackrel{i(f)}{\to} i(Y) \stackrel{i(g)}{\to} i(Z) \to \Sigma i(X)$ in $\Db(\EE)$.
\end{enumerate}
\end{proposition}

\begin{proof}
We first show that (\ref{item:R3}) implies (\ref{item:ConflationsAndTriangles}), so assume that $\EE$ satisfies axiom \ref{R3}.  Proposition \ref{proposition:ConflationsYieldTriangles} shows that conflations in $\EE$ yield triangles in $\Db(\EE)$.  Consider now a sequence $X \stackrel{f}{\to} Y \stackrel{g}{\to} Z$ in $\EE$ such that $i(X) \stackrel{i(f)}{\to} i(Y) \stackrel{i(g)}{\to} i(Z) \to \Sigma i(X)$ is a triangle in $\Db(\EE)$. By proposition \ref{proposition:StalkEmbeddingFunctors}, the functor $i\colon \EE\to \Db(\EE)$ is fully faithful and hence $i(g\circ f)=0$ yields $g\circ f=0$. Writing $j\colon \EE \to \Kb(\EE)$ for the canonical embedding, there is a morphism $h\colon \cone(j(f)) \to j(Z)$ as in the following diagram:
\[\xymatrix{j(X) \ar[r]^{j(f)} & j(Y) \ar[r] \ar[d]_{j(g)}& {\cone(j(f))} \ar[r] \ar@{..>}[dl]^{h} & \Sigma j(X) \\ & j(Z) }\]
which becomes an isomorphism in $\Db(\EE)$.  {This implies that $\cone(h) \in \langle \Acb(\EE) \rangle_{\textrm{thick}}$.}
As $\cone(h)$ is given by the complex
\[\xymatrix@1{ \cdots \ar[r] &0 \ar[r] & \ar[r] X \ar[r]^{f} & Y \ar[r]^{g} & Z \ar[r] & 0 \ar[r] & \cdots}\]
with $Z$ in degree 0, it follows from lemma \ref{lemma:RetractOfAcyclic} that $X \stackrel{f}{\to} Y \stackrel{g}{\to} Z$ is a conflation in $\EE$. This shows (\ref{item:R3})$\Rightarrow$(\ref{item:ConflationsAndTriangles}).

Conversely, assume (\ref{item:ConflationsAndTriangles}). To establish that \ref{R3} holds, we start with two morphisms $X \stackrel{\alpha}{\rightarrow} Y \stackrel{\beta}{\rightarrow} Z$ such that $\beta \alpha$ is a deflation.  We show that if $\beta$ has a kernel, then $\beta$ is a deflation.

We consider the following diagram where the right square is a pullback (see \cite[proposition 5.4]{BazzoniCrivei13} or \cite[proposition~3.7]{HenrardvanRoosmalen19}):
\[
\xymatrix{ 
\ker \gamma \ar@{>->}[r] \ar@{=}[d] & P \ar@{->>}[r]^{\gamma} \ar[d]_{\delta} & Y \ar[d]^{\beta} \\
\ker \beta \alpha \ar@{>->}[r] & X \ar@{->>}[r]^{\beta \alpha} & Z}
\]
It follows from the pullback property that $\delta\colon P \to X$ is a retraction {and that $\ker \delta = \ker \beta$.}  Suppressing the functor $i$, we may use \cite[lemma 1.4.4]{Neeman01} to obtain a diagram in $\Db(\EE)$
\[
\xymatrix{
0\ar[d] \ar[r] & \ker \delta \ar@{=}[r] \ar[d] & \ker \beta \ar[d] \ar[r] & 0 \ar[d]\\ 
\ker \gamma \ar[r] \ar@{=}[d] & P \ar[r]^{\gamma} \ar[d]_{\delta} & Y \ar[d]^{\beta} \ar[r] & \Sigma \ker \gamma \ar@{=}[d]\\
\ker \beta \alpha \ar[r] \ar[d] & X \ar[r]^{\beta \alpha} \ar[d] & Z \ar[r] \ar[d] & \Sigma \ker \beta \alpha \ar[d] \\
0 \ar[r] & \Sigma \ker \delta \ar@{=}[r] & \Sigma \ker \beta \ar[r] & 0}
\]
where the rows and columns are triangles.  From this, we infer that there is a conflation $\ker \beta \rightarrowtail Y \stackrel{\beta}{\twoheadrightarrow} Z$, as required. \qedhere
\end{proof}

The first part of the following proposition was also shown in \cite[\S A.8]{Positselski11}.

\begin{proposition}\label{proposition:ExactMeansExtensionClosed}
Let $\EE$ be a deflation-exact category satisfying axiom \ref{R0*} and let $i\colon \EE\to \Db(\EE)$ be the canonical embedding.
\begin{enumerate}
	\item If $\EE$ is exact, then $i(\EE)$ is extension closed in $\Db(\EE)$.
	\item If $i(\EE)$ is extension closed in $\Db(\EE)$, then $i(\EE)$ can be endowed with a natural exact structure such that the inclusion $\EE\hookrightarrow i(\EE)$ is exact.
	
	{If, in addition, $\EE$ satisfies axiom \ref{R3}, then the deflation-exact category $\EE$ is exact.}
\end{enumerate}
\end{proposition}

\begin{proof}
\begin{enumerate}
	\item Assume that $\EE$ is exact, we will show that $i(\EE)$ is extension closed in $\Db(\EE)$.  Consider a triangle $i(A)\rightarrow B^{\bullet}\rightarrow i(C)\stackrel{f}{\rightarrow} \Sigma i(A)$ in $\Db(\EE)$ with $A,C\in \EE$.  Following proposition \ref{proposition:CongenialReplacement}, we may choose a congenial roof $i(C) \stackrel{\simeq}{\leftarrow}Z^{\bullet}\stackrel{g}{\rightarrow} \Sigma i(A)$ representing $f \in \Ext^1(i(C), i(A)).$  Here, $Z^\bullet$ is a two-term complex $Z^{-1} \inflation Z^0$.
	
	The cone of $g\colon Z^\bullet \to A^\bullet$ is given by the two-term complex $Z^{-1} \stackrel{h}{\inflation} Z^0 \oplus A$ with $h=\begin{psmallmatrix} {g^0} \\ {d^{-1}_Z} \end{psmallmatrix}$.  Hence, in $\Db(\EE)$, we find that $\cone(g) \cong \cone(f)$ is a stalk complex containing $\coker (h)$ in degree $-1$.  As $B^\bullet \cong \Sigma^{-1}\cone(f)$, we see that $B^\bullet$ is indeed a stalk complex concentrated in degree 0.

	\item Assume that $i(\EE)$ is extension closed in $\Db(\EE)$. By proposition \ref{proposition:NoNegativeExtensions}, $\Ext_{\Db(\EE)}^{n}(i(\EE), i(\EE)) = 0$, for all $n < 0$.  Following \cite{Dyer05}, the triangles in $\Db(\EE)$ induce an exact structure on $i(\EE)$. By axiom \ref{R0*} and proposition \ref{proposition:ConflationsYieldTriangles}, each conflation in $\EE$ gives rise to a triangle in $i(\EE)$. It follows that the inclusion $\EE\hookrightarrow i(\EE)$ is exact.

If $\EE$ satisfies axiom \ref{R3}, proposition \ref{proposition:MeaningR3} yields that $\EE$ has the same conflation structure as $i(\EE)$. In particular, $\EE$ is an exact category.\qedhere
\end{enumerate}
\end{proof}

\subsection{Obtaining exact localizations} Having established a criterion for a strongly deflation-exact category to be two-sided exact (proposition \ref{proposition:ExactMeansExtensionClosed}), we can now proceed to proving theorem \ref{theorem:ExactLocalization}. 

\begin{lemma}\label{lemma:HereditaryFactorisation}
Let $\EE$ be a exact category.  Consider the morphisms $X \stackrel{f}{\twoheadrightarrow} Y \stackrel{g}{\rightarrowtail} Z$.  If \[\Hom_{\D^b(\EE)}(\coker g,\Sigma^2 \ker f) = 0,\] then $gf$ admits an inflation-deflation factorization $X  \stackrel{g'}{\rightarrowtail} W \stackrel{f'}{\twoheadrightarrow} Z$.  Moreover, $\ker f \cong \ker f'$ and $\coker g \cong \coker g'$.
\end{lemma}

\begin{proof}
We consider the triangles $\ker f \to X \to Y \to \Sigma \ker f$ and $\Sigma^{-1} \coker g \to Y \to Z \to \coker g$ in $\Db(\EE)$. As $\Hom_{\D^b(\EE)}(\coker g,\Sigma^2 \ker f) = 0$, we may use \cite[proposition 1.4.6]{Neeman01} to obtain the commutative diagram:
\[\xymatrix{
0 \ar[r] & {\coker g} \ar@{=}[r] & {\coker g} \ar[r] & 0\\
{\ker f} \ar[r] \ar[u] & W \ar[r]^{f'} \ar[u] & Z \ar[r] \ar[u] & {\Sigma \ker f} \ar[u]  \\
{\ker f} \ar[r] \ar@{=}[u]& X \ar[r] \ar[u]^{g'} & Y \ar[r] \ar[u] & \Sigma \ker f \ar@{=}[u]\\
0 \ar[r] \ar[u] & {\Sigma^{-1} \coker g} \ar[u] \ar@{=}[r] & {\Sigma^{-1} \coker g} \ar[u] \ar[r] & 0 \ar[u]
}\]
Proposition \ref{proposition:ExactMeansExtensionClosed} shows that $W \in i(\EE) \subset \D^b (\EE)$ and proposition \ref{proposition:MeaningR3} (recall that an exact category satisfies axiom \ref{R3}) yields that $g'\colon X \rightarrowtail W$ is an inflation and $f'\colon W \twoheadrightarrow Z$ is a deflation. It is obvious that $\ker f \cong \ker f'$ and $\coker g \cong \coker g'$.
\end{proof}

\begin{theorem}\label{theorem:ExactLocalization}
Let $\EE$ be a exact category and let $\AA \subseteq \EE$ be a deflation-percolating subcategory.  If $$\Hom_{\D^b(\EE)}(E,\Sigma^2 A) = 0,$$ for all $E \in \EE$ and $A \in \AA$, then the deflation-exact structure on $\EE / \AA$ is an exact structure.
\end{theorem}

\begin{proof}
By \cite{Keller90}, a conflation category satisfying axioms \ref{R0}, \ref{R1}, \ref{R2} and \ref{L2} is an exact category. As the first three conditions are guaranteed by Theorem \ref{theorem:MainTheoremPartI}, we only need to show that axiom \ref{L2} holds.

Thus, consider a span $Y \stackrel{f}{\leftarrow} X \stackrel{g}{\rightarrowtail} Z$ in $\EE / \AA$ where $g$ is an inflation, thus, $g$ is isomorphic to a morphism $Q(i)$ where $i$ is an inflation in $\EE$.  This gives the following diagram in $\EE/\AA$
\[\xymatrix{
{X'} \ar[r]^{\alpha} \ar@{>->}[d]_{Q(i)} & {X} \ar[r]^{f} \ar[d]^{g} & Y \\
{Z'} \ar[r]^{\beta} & {Z} & }\]
where $\alpha, \beta$ are isomorphisms. We now only consider the upper and the left part of this diagram.  Writing $f\alpha$ as a roof $hs^{-1}$, we obtain a diagram
\[\xymatrix{
{X'} \ar@{>->}[d]_{i} & \ar[l]_{s}^{\sim} {\overline{X}} \ar[r]^{h} & Y \\
{Z'} & }\]
in $\EE$.  Since $s$ is a weak isomorphism, it is a finite composition of $\AA^{-1}$-inflations and $\AA^{-1}$-deflations. If $s$ ends on an $\AA^{-1}$-inflation, axiom \ref{L1} allows us to absorb this $\AA^{-1}$-inflation into the inflation $i$. Hence we may assume that $s$ ends on a deflation, i.e.~$s$ factors as $\xymatrix{\overline{X} \ar[r]^{\sim}_{s_2} & X''\ar@{->>}[r]^{\sim}_{s_1} & X'}$. By lemma \ref{lemma:HereditaryFactorisation}, the composition $i\circ s_1$ factors as $\xymatrix{X''\ar@{>->}[r]_j & W \ar@{->>}[r]^{\sim}_{s_1'} & Z' }$. Iterating this argument on the length of $s$, we obtain a span $\xymatrix{\overline{Z}&\overline{X}\ar[r]^h\ar@{>->}[l]_{k} &Y }$ in $\EE$ which is isomorphic to the original span in $\EE/\AA$.

Since $\EE$ is exact, the pushout $P$ of this span exists in $\EE$ and yields an inflation $Y\rightarrowtail P$. Since the localization functor $Q$ is conflation-exact, it commutes with cokernels of inflations. Since the pushout $P$ can be obtained as a cokernel of an inflation in $\EE$, the pushout square $\overline{X}YP\overline{Z}$ descends to a pushout square in $\EE/\AA$. Clearly the map $Y\rightarrowtail P$ descends to an inflation as well. This completes the proof. 
\end{proof}

We now apply the previous theorem to the category of locally compact abelian groups $\LCA$.

\begin{corollary}\label{corollary:LocalizationAtCompactGroups}
For any left or right percolating subcategory $\AA$ of $\LCA$, the quotient $\LCA / \AA$ is exact.
\end{corollary}

\begin{proof}
It is shown in \cite[proposition~4.13, proposition~4.15]{HoffmannSpitzweck07} that $\Hom_{\D^b(\EE)}(E,\Sigma^2 A) = 0,$ for all $E,A \in \EE$.  Theorem \ref{theorem:ExactLocalization} yields that $\LCA / \AA$ is an exact category.
\end{proof}
\section{The exact hull}\label{section:ExactHull}

In this section we construct the {exact hull} $\overline{\EE}$ of $\EE$. Explicitly, $\overline{\EE}$ is the extension closure of $i(\EE)\subseteq \Db(\EE)$. If $\EE$ satisfies axiom \ref{R0*}, the inclusion map $j\colon \EE\hookrightarrow \overline{\EE}$ is exact and $2$-universal among exact functors to exact categories (see proposition \ref{proposition:ExactHullUniversalAppendix}). In theorem \ref{theorem:DerivedEquivalenceExactHull} we show that the embedding $\overline{\EE} \subseteq \Db(\EE)$ lifts to a triangle equivalence $\Db(\overline{\EE})\stackrel{\sim}{\rightarrow}\Db(\EE)$.

\subsection{Construction of the exact hull}

\begin{definition}\label{definition:ExactHull}
	Let $\EE$ be a deflation-exact category and let $i\colon \EE\hookrightarrow \Db(\EE)$ be the canonical embedding. We denote the extension closure of $i(\EE)$ in $\Db(\EE)$ by $\overline{\EE}$. The category $\overline{\EE}$ is called the \emph{exact hull} of $\EE$. The composition $\EE\hookrightarrow i(\EE)\hookrightarrow \overline{\EE}$ is denoted by $j$.
\end{definition}

We extend proposition \ref{proposition:NoNegativeExtensions} to the extension closure $\overline{\EE}$.

\begin{proposition}\label{proposition:NoNegativeExtensionsOfClosure}
	For all $X,Y\in \Ob(\overline{\EE})$ we have that $\Hom_{\Db(\EE)}(X,\Sigma^{-n}Y)=0$ for all $n>0$.
\end{proposition}

\begin{proof}
	Let $\EE_0=i(\EE)$. We inductively define $\EE_i$ (for $i\geq 0$) as the set of all objects $E\in \Ob(\Db(\EE))$ occurring in a triangle
	\[X \rightarrow E \rightarrow Y \rightarrow \Sigma X\]
	where $X,Y\in \EE_{i-1}$. We now claim that $\Hom_{\Db(\EE)}(\EE_i, \Sigma^{-1}\EE_i)=0$. We proceed by induction on $i$. The case where $i=0$ is shown in proposition \ref{proposition:NoNegativeExtensions}.
	
	Consider now $i>0$. Let $E_{i-1}\in \EE_{i-1}$. Applying the functor $\Hom(E_{i-1},-)$ to the triangle $X \rightarrow E \rightarrow Y \rightarrow \Sigma X$ as above (thus, with $X,Y\in \EE_{i-1}$), yields the long exact sequence:
	\[\cdots\rightarrow\Hom_{\Db(\EE)}(E_{i-1},\Sigma^{-n}X)\rightarrow \Hom_{\Db(\EE)}(E_{i-1},\Sigma^{-n}E)\rightarrow \Hom_{\Db(\EE)}(E_{i-1},\Sigma^{-n}Y)\rightarrow\cdots\]
	By the induction hypothesis, we have that $\Hom_{\Db(\EE)}(\EE_{i-1},\Sigma^{-n}\EE_{i-1})=0$ and hence $\Hom_{\Db(\EE)}(\EE_{i-1},\Sigma^{-n}E)=0$. This shows that $\Hom_{\Db(\EE)}(\EE_{i-1},\Sigma^{-n}\EE_i)=0$.\\
	On the other hand, applying the functor $\Hom(-,\Sigma^{-n}E_i)$ with $E_i\in \EE_i$ to the same triangle yields the long exact sequence
	\[\cdots\rightarrow \Hom_{\Db(\EE)}(Y, \Sigma^{-n}E_i)\rightarrow  \Hom_{\Db(\EE)}(E, \Sigma^{-n}E_i)\rightarrow  \Hom_{\Db(\EE)}(X, \Sigma^{-n}E_i)\rightarrow\cdots\]
	Using that $\Hom_{\Db(\EE)}(\EE_{i-1},\Sigma^{-n}\EE_i)=0$ and $X,Y\in \EE_{i-1}$, we find that $\Hom_{\Db(\EE)}(E,\Sigma^{-n}E_i)=0$. Hence $\Hom_{\Db(\EE)}(\EE_i,\Sigma^{-n}\EE_i)=0$. This shows the claim.
	
	The required statement follow from $\overline{\EE}=\bigcup_{i\geq 0}\EE_i$.
\end{proof}

\begin{corollary}\label{corollary:ExactHullDefinition}
	Let $\EE$ be a deflation-exact category. The category $\overline{\EE}\subseteq \Db(\EE)$ has the structure of an exact category, the conflations are given by triangles in $\Db(\EE)$ with three consecutive terms in $\overline{\EE}$.
	
	If $\EE$ satisfies axiom \ref{R0*}, the embedding $\EE \to \overline{\EE}$ is exact.
\end{corollary}

\begin{proof}
	Proposition \ref{proposition:NoNegativeExtensionsOfClosure} shows that $\overline{\EE}$ is an extension closed subcategory of $\Db(\EE)$ having no negative self-extensions. Following \cite{Dyer05}, $\overline{\EE}$ is an exact category.
	
	If $\EE$ satisfies axiom \ref{R0*}, proposition \ref{proposition:ConflationsYieldTriangles} implies that the embedding $\EE\hookrightarrow \overline{\EE}$ is exact.
\end{proof}

\subsection{\texorpdfstring{Universality of the embedding $\EE \to \overline{\EE}$}{Universality of the embedding in the exact hull}} We now show that the embedding of a deflation-exact category into its exact hull $\EE \to \overline{\EE}$ is 2-universal (see proposition \ref{proposition:ExactHullUniversalAppendix}).  We start by showing that every object in $\overline{\EE}$ can be obtained as the cokernel of a monomorphism in $\EE \subseteq \overline{\EE}$ (see corollary \ref{corollary:CokernelExactHull}).

\begin{proposition}\label{proposition:TwoTermReplacement}
	Let $\EE$ be a deflation-exact category and let $\overline{\EE}$ be its exact hull. Any complex $W^{\bullet}\in \overline{\EE}$ can be represented by a two-term complex $Y^{-1}\hookrightarrow Y^{0}$ in $\Db(\EE)$.
\end{proposition}

\begin{proof}
	Since $\overline{\EE}$ is the extension closure of $\EE$ in $\Db(\EE)$, we have $\overline{\EE}=\bigcup_{i\geq 0}\EE_i$ where $\EE_0=\EE$ and where each $\EE_i$ ($i \geq 1$) is given by objects $W^{\bullet}\in \Db(\EE)$ occurring in a triangle $A\to W^{\bullet}\to X^{\bullet}\to \Sigma A$ with $A\in \EE$ and $X\in \EE_{i-1}$. Let $W^{\bullet}\in \EE=\bigcup_{i\geq 0} \EE_i$, then $W^{\bullet}\in \EE_i$ for some $i\geq 0$. We proceed by induction on $i\geq 0$. If $i=0$, there is nothing to prove, so assume that $i\geq 1$.

By the induction hypothesis, $X$ can be chosen to be a two-term complex $X^{-1} \to X^0.$  The map $X^{\bullet}\to \Sigma i(A)$ in $\Db(\EE)$ can be represented by a roof $X^{\bullet} \xleftarrow{\simeq} Z^{\bullet}\to \Sigma i(A)$ in $\Kb(\EE)$ with $Z^\bullet \stackrel{\simeq}{\rightarrow} X^\bullet$ a congenial morphism and where $Z^i = 0$ whenever $i\not=0, -1$ (see proposition \ref{proposition:CongenialReplacement}).

We have that $W^\bullet \cong \cone(Z^\bullet \to i(A))$ so that $W^\bullet$ is represented a two-term complex $Z^{-1} \stackrel{d_W^{-1}}{\rightarrow} Z^0 \oplus A$.  We need to show that $d_W^{-1}$ is a monomorphism.  For this, consider any morphism $f\colon T \to Z^{-1}$ in $\EE$ such that $d^{-1}_W \circ f = 0$.  This gives a morphism $f^\bullet\colon \Sigma i(T) \to W^\bullet$.  As $W^\bullet, \Sigma i(T) \in \overline{\EE}$, we know that $f^\bullet = 0$ and hence $f = 0$.  This shows that $d_W^{-1}$ is a monomorphism.
\end{proof}

\begin{corollary}\label{corollary:CokernelExactHull}
For every $Z\in \overline{\EE}$, there is a conflation $X \inflation Y \deflation Z$ in $\overline{\EE}$ with $X,Y \in \EE$.
\end{corollary}

\begin{proof}
By proposition \ref{proposition:TwoTermReplacement}, $Z \in \overline{\EE} \subseteq \Db(\EE)$ can be represented by a two term complex $X\xrightarrow{f} Y$.  Note that $Z \cong \cone(i(f))$.  Hence, there is a triangle $i(X) \xrightarrow{f} i(Y) \to Z \to \Sigma i(X)$ in $\Db(\EE)$.  This shows that $i(X) \inflation i(Y) \deflation Z$ is a conflation in $\overline{\EE}$.
\end{proof}

In the following, we write $\Adm(\FF)$ for the category of admissible arrows in $\FF$ and $\Adm_{\overline{\EE}}(\EE)$ for the category of arrows in $\EE$ which become admissible under $j\colon \EE \to \overline{\EE}$.  Recall from \cite{AdamekElBashirSobralVelebil01} that a functor $F\colon \CC \to \DD$ is called a \emph{lax epimorphism} if the induced functor $-\circ F\colon \Fun(\DD, \Set) \to \Fun(\CC, \Set)$ is fully faithful.

\begin{corollary}\label{corollary:AdmissibleUnderJ}
The functor $\coker\colon \Adm_{\overline{\EE}}(\EE) \to \overline{\EE}$ is a lax epimorphism.
\end{corollary}

\begin{proof}
Let $f\colon X\to Y$ be a map in $\overline{\EE}$ such that $X,Y$ are in the essential image of $\coker\colon \Adm_{\overline{\EE}}(\EE) \to \overline{\EE}$. By axiom \ref{R2} we obtain the following solid commutative diagram in $\overline{\EE}$:
	\[\xymatrix{
		P\ar@{.>>}[r]&P'\ar[d]\ar@{->>}[r]&X\ar[d]_f\\
		A\ar@{>->}[r]^{g}&B\ar@{->>}[r]&Y	
	}\] Here, we used corollary \ref{corollary:CokernelExactHull} to obtain the lower conflation with $A,B\in \EE$. Using corollary \ref{corollary:CokernelExactHull} once more, we find an object $P\in \EE$ and a deflation $P\deflation P'$. Applying corollary \ref{corollary:CokernelExactHull} to the kernel of $P\deflation X$, we obtain a map $Q\xrightarrow{h} P$ such that $h\in \Adm_{\overline{\EE}}(\EE)$ and $\coker(h)\cong X$. The induced map $h\to g$ in $\Adm_{\EE}(\EE)$ maps to $f$ under the functor $\coker$. This shows the third characterization of a lax epimorphism in \cite[theorem~1.1]{AdamekElBashirSobralVelebil01}.
\end{proof}

\begin{proposition}\label{proposition:ExactHullUniversalAppendix}
	Let $\EE$ be a deflation-exact category satisfying axiom \ref{R0*}.  The embedding $j\colon \EE\hookrightarrow \overline{\EE}$ is $2$-universal among exact functors to exact categories.
\end{proposition}

\begin{proof}
	Let $F\colon \EE\rightarrow \FF$ be an exact functor to an exact category $\FF$.  By the universal property of Verdier localizations, there exists a triangle functor $\widetilde{F}\colon \Db(\EE)\to \Db(\FF)$ such that the following diagram commutes:
	\[\xymatrix{
		\EE\ar[d]^{F}\ar@{^{(}->}[r] & \Db(\EE)\ar[d]^{\widetilde{F}}\\
		\FF\ar@{^{(}->}[r] & \Db(\FF)
	}\]
	By proposition \ref{proposition:ExactMeansExtensionClosed}, $\FF$ is an extension-closed subcategory of $\D^b(\FF)$. It follows that the restriction of $\widetilde{F}$ to $\overline{\EE}$ maps $\overline{\EE}$ into the essential image of $\FF \subseteq \Db(\FF)$. Let $\overline{F}$ be the restriction of $\widetilde{F}$ to $\overline{\EE}$. Clearly $\overline{F}\colon \overline{\EE}\to \FF\subseteq \Db(\FF)$ is exact and $\overline{F}\circ j=F$.
	
	Now assume that $G,H\colon \overline{\EE}\to \FF$ are two exact functors such that $G\circ j = H\circ j=F$.  Using the notation of corollary \ref{corollary:AdmissibleUnderJ}, one finds that both $G$ and $H$ give an essentially commutative diagram:
	\[\xymatrix{
	{ \Adm_{\overline{\EE}}(\EE)} \ar[rr]^{\coker} \ar[d]_{F} && {\overline{\EE}} \ar@<-1ex>[d]_{G} \ar@<1ex>[d]^{H} \\
	{ \Adm(\FF)} \ar[rr]^{\coker} && {\FF}}	\]
	Hence, $G \circ \coker \cong H \circ \coker$.  As $\coker\colon \Adm_{\overline{\EE}}(\EE) \to \overline{\EE}$ is a {lax epimorphism}, we find that $G$ and $H$ are naturally equivalent.
\end{proof}

\begin{remark}
	In \cite[proposition~I.7.5]{Rosenberg11}, a construction for the exact hull of a deflation-exact category is given using the embedding of $\EE$ into its category of sheaves.  As both hulls satisfy the same $2$-universal property, they are equivalent.
\end{remark}

\subsection{Derived equivalence of the exact hull}

{We now complete the proof of theorem \ref{Theorem:ExactHullIntroduction} from the introduction.  We will use the following definition from \cite{Keller06}.}

\begin{definition}\label{definition:AlgebraicTriangulatedCategory}
	Let $\mathcal{T}$ be a $k$-linear triangulated category (with $k$ a commutative ring). We say that $\mathcal{T}$ is \emph{algebraic} if it is triangle-equivalent to $\underline{\CC}$ for some $k$-linear Frobenius category $\CC$ (here, $\underline{\CC}$ denotes the stable category of $\CC$).
\end{definition}

The following proposition is well known (see \cite[section~3.6]{Keller06}).

\begin{proposition}\label{proposition:VerdierLocalizationPreservesAlgebraicTriangulation}
Full triangulated subcategories and Verdier localizations of algebraic categories are algebraic.
\end{proposition}

\begin{corollary}\label{corollary:HomotopyAndTriangulatedCatAreAlgebraic}
Let $\EE$ be a deflation-exact category. The homotopy category $\K(\EE)$ and the derived category $\D(\EE)$ are algebraic triangulated categories.
\end{corollary}

\begin{proof}
	We endow the category of complexes $\C(\EE)$ with the structure of a Frobenius exact category by taking as set of conflations these kernel-cokernel pairs which are degreewise split.  The projective-injective objects are then the null-homotopic complexes and the corresponding stable category is the homotopy category $\K(\EE)$.  It follows from proposition \ref{proposition:VerdierLocalizationPreservesAlgebraicTriangulation} that the derived category $\D(\EE)$ is algebraic as well.
\end{proof}

The following theorem is the last theorem of \cite{KellerVossieck87}.

\begin{theorem}\label{theorem:KellerVossieck}
Let $\AA$ be a Frobenius category, $\BB$ an additive category and $F\colon \BB\rightarrow \underline{\AA}$ an additive functor such that 
\begin{align*}
\Hom_{\underline{\AA}}(\Sigma^nFA,FB)&=0, & \mbox{for all $A,B\in \BB$, and for all $n>0$.}
\end{align*}
\begin{enumerate}
	\item The functor $F$ extends to a triangulated functor $\tilde{F}\colon\Kb(\BB)\rightarrow \underline{\AA}$. Moreover, $\tilde{F}$ is fully faithful if and only if $\Hom_{\underline{\AA}}(FA, \Sigma^nFB)=0$ for all $A,B\in \BB$ and for all $n>0$.
	\item If the category $\BB$ is exact and the image of every short exact sequence extends to a triangle, then $\tilde{F}$ decomposes as $\Kb(\BB)\xrightarrow{\can}\Db(\BB)\xrightarrow{\widehat{F}}\underline{\AA}$. Moreover, $\widehat{F}$ is fully faithful if and only if $F$ is fully faithful and for each $A,A'\in \BB$, $n>0$  and $f\in \Hom_{\underline{\AA}}(FA',\Sigma^nFA)$, every short exact sequence 
	\[0\rightarrow A\xrightarrow{j} B \rightarrow C\rightarrow 0\] in $\BB$ satisfies $(\Sigma^nFj)\circ f=0$.
\end{enumerate}
\end{theorem}

The next proposition is similar to \cite[proposition~4.5]{StanleyvanRoosmalen16}. The adaptation to this setting is left to the reader.

\begin{lemma}\label{lemma:HigherExtsFactorThroughLowerExts}
Let $\EE$ be a deflation-exact category and let $\overline{\EE}$ be the extension closure of $i(\EE)$ in $\Db(\EE)$. The following are equivalent:
\begin{enumerate}
	\item For all $f\in \Hom_{\D^b(\EE)}(A,\Sigma^n B)$ with $A,B\in \overline{\EE}$ and $n\geq 2$, there exists an inflation $g\colon B\rightarrowtail C$ in $\overline{\EE}$ such that the composition $\xymatrix{A\ar[r]^f & \Sigma^nB \ar@{>->}[r]^{\Sigma^n g} &\Sigma^nC}$ in $\Db(\EE)$ is zero.
	\item For all $f\in \Hom_{\D^b(\EE)}(A,\Sigma^n B)$  with $A,B\in \overline{\EE}$ and $n\geq 2$, there exists a deflation $h\colon C\twoheadrightarrow A$ such that $\xymatrix{C\ar@{->>}[r]^h & A\ar[r]^f & \Sigma^n B}$ in $\Db(\EE)$ is zero.
	\item For all $f\in \Hom_{\D^b(\EE)}(A,\Sigma^n B)$ with $A,B\in \overline{\EE}$ and $n\geq 2$, $f$ factors in $\Db(\EE)$ as follows
	\[\xymatrix{
	A\ar@/^2.0pc/[rrrrr]^f \ar[r] & \Sigma A_1\ar[r] & \Sigma^2 A_2\ar[r] & \dots \ar[r]& \Sigma^{n-1} A_{n-1}\ar[r] & \Sigma^n B
	}\] for some $A_1, \dots ,A_{n-1}\in \overline{\EE}$.
\end{enumerate}
\end{lemma}

\begin{corollary}\label{corollary:FullyFatihfulnessRealisationFunctor}
Let $\EE$ be a deflation-exact category and $\overline{\EE}$ its exact closure. The inclusion functor $\overline{\EE}\hookrightarrow \Db(\EE)$ factors through a functor
\[\real\colon \Db(\overline{\EE})\rightarrow \Db(\EE)\]
which is an equivalence if and only if one of the equivalent properties of lemma \ref{lemma:HigherExtsFactorThroughLowerExts} holds.
\end{corollary}

\begin{proof}
	By proposition \ref{proposition:VerdierLocalizationPreservesAlgebraicTriangulation} and corollary \ref{corollary:HomotopyAndTriangulatedCatAreAlgebraic}, the category $\Db(\EE)$ is algebraic, i.e.~$\Db(\EE)$ is triangle equivalent to $\underline{\CC}$ for some Frobenius category $\CC$. Hence we may apply theorem \ref{theorem:KellerVossieck} to the canonical embedding $F\colon\overline{\EE}\rightarrow \Db(\EE)$. The functor $\real\coloneqq \widehat{F}$ is fully faithful if and only if the conditions of lemma \ref{lemma:HigherExtsFactorThroughLowerExts} hold. As $\EE\subseteq \overline{\EE}$, it follows that $\real$ is essentially surjective.
\end{proof}

\begin{theorem}\label{theorem:DerivedEquivalenceExactHull}
Let $\EE$ be a deflation-exact category satisfying axiom {\ref{R0*}} and let $j\colon \EE \to \overline{\EE}$ be its embedding into its exact hull.
\begin{enumerate}
  \item\label{enumerate:Equivalence1} The embedding $\overline{\EE} \to \Db(\EE)$ lifts to a triangle equivalence $\Db(\overline{\EE}) \simeq \Db(\EE)$, and
	\item\label{enumerate:Equivalence2} The embedding $j\colon \EE \to \overline{\EE}$ lifts to a triangle equivalence $\Db({\EE}) \simeq \Db(\overline{\EE})$
\end{enumerate}
\end{theorem}

\begin{proof}
For the first statement, let $f\colon X^\bullet \to \Sigma^n Y^\bullet$ be a morphism in $\Db(\EE)$ with $X^\bullet,Y^\bullet \in \overline{\EE}$ and $n \geq 2.$  By corollary \ref{corollary:CokernelExactHull}, we may assume that $X^\bullet$ and $Y^\bullet$ are 2-term complexes (there are conflations $X^{-1} \inflation X^0 \to X^\bullet$ and $Y^{-1} \inflation Y^0 \deflation Y^\bullet$ in $\overline{\EE}$ with $X^{-1}, Y^{-1}, X^0, Y^0 \in \EE$) and that $f$ is represented by a roof $X^\bullet \stackrel{\simeq}{\leftarrow}Z^\bullet \to \Sigma^n Y^\bullet$ where $Z^\bullet \stackrel{\simeq}{\rightarrow} X^\bullet$ is a congenial quasi-isomorphism and $Z^\bullet\in \C^{[-n-1,0]}.$  The map $Z^\bullet \to \Sigma^n Y$ is represented by a diagram
\[\xymatrix{
\cdots \ar[r] &0 \ar[r]\ar[d] & Z^{-n-1}\ar[d] \ar@{>->}[r]^{d_Z^{-n-1}} & Z^{-n} \ar[d]\ar[r] & Z^{-n+1} \ar[r]\ar[d] & \cdots \ar[r] & Z^{-1} \ar[r]\ar[d]& Z^0 \ar[r]\ar[d] & 0 \ar[r]\ar[d] & \cdots \\
\cdots \ar[r] &0 \ar[r] & Y^{-1} \ar@{>->}[r] & Y^0 \ar[r] & 0 \ar[r] & \cdots \ar[r] & 0 \ar[r]& 0 \ar[r] & 0 \ar[r] & \cdots }\]
As $X^\bullet$ is acyclic in degree $-n+1$ (the map $X^{-1} \to X^0$ is a monomorphism), so is $Z^\bullet$.  Hence, $d_Z^{-n}\colon Z^{-n} \to Z^{-n+1}$ is admissible.  Writing $C$ for $\coker d_Z^{-n-1}$, the left part of the above diagram induces solid span in $\overline{\EE}:$
\[\xymatrix{ C \ar@{>->}[r]\ar[d] & {Z^{-n+1}} \ar@{..>}[d]\\
{Y^\bullet} \ar@{>..>}[r] & P} \]
which can be completed to a pushout square as above (as $\overline{\EE}$ is exact, see corollary \ref{corollary:ExactHullDefinition}).  Here, $P \cong \coker (C \to Z^n \oplus Y^\bullet).$  Put differently, $P \cong \cone([Z^{-n-1}\to Z^{-n} \to Z^{n-1}] \to [Y^{-1}\to Y^0 \to 0]).$  Hence, the composition $Z^\bullet \to \Sigma^n Y \to \Sigma^n P = 0$.  As $Y\inflation P$ is an inflation in $\overline{\EE}$, the conditions of theorem \ref{theorem:KellerVossieck} are satisfied.

For the second statement, it follows from proposition \ref{proposition:ConflationsYieldTriangles} that there is a commutative diagram of functors:
\[ \xymatrix{
\Acb(\EE) \ar[r] \ar[d] & \Kb(\EE) \ar[r] \ar[d]^{j} & \Db(\EE) \ar[d]^{j} \\
\Acb(\overline{\EE}) \ar[r] & \Kb(\overline{\EE}) \ar[r] & \Db(\overline{\EE})
}\]
and, hence, a triangle functor $\Db(\EE) \to \Db(\overline{\EE})$.  It suffices to show that $\langle \Acb(\overline{\EE}) \cap \Kb(\EE) \rangle_{\textrm{thick}} = \langle \Acb(\EE) \rangle_{\textrm{thick}}.$  Indeed, it then follows from \cite[proposition 7.2.1]{KashiwaraSchapira06} that the functor $\Db(\EE) \to \Db(\overline{\EE})$ is fully faithful.  That it is essentially surjective follows from proposition \ref{corollary:CokernelExactHull}.

The inclusion $\langle \Acb(\overline{\EE}) \cap \Kb(\EE) \rangle_{\textrm{thick}} \supseteq \langle \Acb(\EE) \rangle_{\textrm{thick}}$ is clear.  To show the other inclusion, consider an acyclic complex $X^\bullet \in \Acb(\overline{\EE})$ with entries in $\EE$.  Without loss of generality, assume that $X^\bullet \in \AcC^{[0,n]}(\EE)$.  We write $K_i = \ker(d^i\colon X^i \to X^{i+1})\in \overline{\EE}$; these kernels exist since $X^\bullet$ is acyclic in $\overline{\EE}$.  Note that $K^l = 0$ for $l \geq n+1.$  There are conflations $(X^0 \inflation X^1 \deflation K_1), (K_1 \inflation X^1 \deflation K_2), \ldots, (K_{n-2} \inflation X^{n-1} \deflation X^n)$ in $\overline{\EE}$ and each of these corresponds to a triangle in $\Db(\EE)$.

From the triangle $X^0 \to X^1 \to K_1 \to \Sigma X^0$, we find that $K_1$ is isomorphic to the complex $\cdots \to 0 \to X^0 \to X^1 \to 0 \to \cdots$ (where $X^1$ is in degree 0).  By iteration, we find that $K_i$ is isomorphic to the complex $\cdots \to X^0 \to X^1 \to \cdots \to X^i \to 0 \to \cdots$ (where $X^i$ is in degree 0) in $\Db(\EE)$.  Hence, up to suspension, $X^\bullet \cong K_{n+1}.$  It follows that the complex $X^\bullet$ itself is zero in $\Db(\EE)$.  This implies that $X^\bullet \in \langle \Acb(\EE) \rangle_{\textrm{thick}}$, as required.
\end{proof}

The above theorem allows us to obtain theorem \ref{theorem:IntroductionQuotientsForExact} about quotient for exact categories from theorem \ref{Theorem:MainTheoremIntroduction}.  We work in the 2-category of (essentially small) exact categories.  Let $\AA\subseteq \EE$ be an inflation- or deflation-percolating subcategory of an exact category $\AA$.  The localization $\EE[S_\AA^{-1}]$ is inflation- or deflation-exact respectively, but not necessarily two-sided exact.  As in \cite{HenrardvanRoosmalen19}, we can embed $\EE[S_\AA^{-1}]$ into its exact hull; the composition $\EE \to \EE[S_\AA^{-1}] \to \overline{\EE[S_\AA^{-1}]}$ satisfies the 2-universal property of the quotient.  We obtain the following corollary.

\begin{corollary}\label{Corollary:MainTheoremExact}
Let $\EE$ be an exact category and let $\AA$ be an inflation- or deflation-percolating subcategory of $\EE$.  The sequence 
		\[\DAb(\EE)\to \Db(\EE)\to \Db(\overline{\EE/\AA})\]
is a Verdier localization sequence. Here, $Q\colon \EE\to \overline{\EE / \AA} \eqqcolon \EE \dq \AA$ is the quotient in the 2-category of (essentially small) exact categories. 
	\end{corollary}
	
	The following proposition states that projective objects in $\EE$ stay projective in the exact hull $\overline{\EE}$.
	
\begin{proposition}
The embedding $j\colon \EE \to \overline{\EE}$ restricts to a fully faithful functor $\Proj(\EE) \to \Proj(\overline{\EE})$.
\end{proposition}

\begin{proof}
We first show that $j\colon \EE \to \overline{\EE}$ restricts to a functor $\Proj(\EE) \to \Proj(\overline{\EE}).$  Let $P \in \EE$ be any projective object, this means that $\Hom_{\Db(\EE)}(P,\Sigma (E)) = 0$, for all $E \in \EE.$  As $j\colon \Db(\EE) \to \Db(\overline{\EE})$ is an equivalence, we find $\Hom_{\Db(\overline{\EE})}(jP,\Sigma (jE)) = 0.$  As $\overline{\EE}$ is the extension-closure of $j\EE$ in $\Db(\overline{\EE})$, we find that $\Hom_{\Db(\overline{\EE})}(jP,\Sigma \overline{\EE}) = 0.$  Hence, $jP \in \Proj(\overline{\EE}).$
\end{proof}

\begin{example}\label{Example:IsbellCategory}
	Let $\II$ be the Isbell category, that is, the full additive subcategory of $\Ab$ generated by the abelian groups containing no element of order $p^2$ for some fixed prime $p$.  The Isbell category is a deflation quasi-abelian category which does not satisfy axioms \ref{L1}, \ref{L2} and \ref{L3} (see \cite[section~2]{Kelly69} and \cite[example~4.7]{BazzoniCrivei13}).  We also consider the full subcategory $\II^{\text{fg}}$ of $\II$ consisting of the finitely generated groups of $\II$.
	
  As in the proof of proposition \ref{proposition:TwoTermReplacement}, let $\II^{\text{fg}}_0 = \II^{\text{fg}}$ and let $\II^{\text{fg}}_i$ (for $i \geq 1$) be given by objects $G\in \Db(\II^{\text{fg}})$ occurring in a triangle $F\to G\to H\to \Sigma (F)$ with $F\in \II^{\text{fg}}$ and $H\in \II^{\text{fg}}_{i-1}$.	 The category $\II^{\text{fg}}$ has enough projectives, with $\Proj(\II^{\text{fg}}) = \add(\bZ),$ and hence, the natural map $\Db(\II^{\text{fg}}) \to \Db(\modd \bZ)$ is an equivalence.  Identifying these derived categories, gives the following description of $\II^{\text{fg}}_n:$ it is the full subcategory of $\modd \bZ$ consisting of those finitely generated abelian groups containing no element of order $p^{n+2}.$  We find $\overline{\II^{\text{fg}}} = \bigcup_{n \in \bN} \II^{\text{fg}}_n = \modd \bZ.$
	
	For the exact hull of $\II$, we find that $\Db(\II) \cong \Db(\Mod \bZ)$ and we find the following description of $\II_n:$ it is the full subcategory of $\Mod \bZ$ consisting of those abelian groups containing no element of order $p^{n+2}.$  In this case, we find that the natural embedding $\overline{\II} \to \Mod \bZ$ is not an equivalence.
	
	In fact, the exact hull $\overline{\II}$ is not pre-abelian.  Indeed, consider the maps $f_i\colon \bZ \to \bZ\colon z \mapsto p^i z_i$, where $i \in \bN$ (so $\coker(f_i) \cong \bZ / p^i\bZ$); the direct sum $\oplus_i f_i\colon {\oplus_{i} \bZ} \to {\oplus_{i} \bZ}$ has no cokernel in $\overline{\II}.$
\end{example}
\section{Enhancements of the derived category and \texorpdfstring{$K$}{K}-theory}\label{section:enhancements}\label{section:KTheory}

As mentioned in the introduction, one of the possible applications of theorems \ref{theorem:MainTheorem} and \ref{theorem:DerivedEquivalenceExactHull} is to non-connective $K$-theory.  However, it is known that the derived category alone of an exact category is not sufficient to recover the $K$-theory spectra (see \cite{Schlichting02}).  In this section, we formulate the aforementioned theorems in three models of the bounded derived category of a one-sided exact category.  We first look at a Frobenius model, following \cite{Schlichting06} closely.  Subsequently, we provide a dg enhancement (in the sense of \cite{Keller06}) and an $\infty$-categorical enhancement (in the sense of \cite{LurieHA, LurieHTT}).  These latter approaches follow \cite{ChenChen19} and \cite[appendix A]{AntieauGepnerHeller19} closely.

\subsection{Frobenius models}

In \cite{Schlichting06}, an enhancement of the derived category of an exact category is given in terms of Frobenius pairs.  As the construction of the derived category of a one-sided exact category is similar to that of an exact category, the same construction can be used to define the $K$-theory spectrum for one-sided exact categories.

Recall the following definitions from \cite{Schlichting06}.

\begin{definition}\label{definition:FrobeniusPair}
	\begin{enumerate}
		\item A \emph{map of Frobenius categories} $F\colon \CC\to \DD$ is an exact functor preserving projective-injective objects.
		\item A \emph{Frobenius pair} $(\CC,\CC_0)$ is a fully faithful embedding $\CC_0\hookrightarrow \CC$ of Frobenius categories (in particular, projective-injective objects of $\CC_0$ are mapped to projective-injective objects of $\CC$).
		\item A map $(\CC,\CC_0)\rightarrow (\DD,\DD_0)$ of Frobenius pairs is a map $\CC\rightarrow \DD$ of Frobenius categories such that $\CC_0$ is mapped into $\DD_0$.
	\end{enumerate}
\end{definition}

\begin{definition}
	Given a Frobenius pair $(\CC,\CC_0)$, the induced map $\underline{\CC_0}\rightarrow \underline{\CC}$ of small triangulated categories is fully faithful. The \emph{derived category} $\D(\CC,\CC_0)$ of the Frobenius pair is defined as the Verdier localization $\underline{\CC}/\underline{\CC_0}$.
\end{definition}

\begin{definition}
	An \emph{exact sequence of Frobenius pairs} is a composable pair of maps of Frobenius pairs
	\[(\CC',\CC_0')\xrightarrow{f} (\CC,\CC_0)\xrightarrow{g}(\CC'',\CC_0'')\]
	together with a natural transformation $\eta$ from $g\circ f$ to the zero functor $0\colon \CC'\rightarrow \CC''$ such that for each $C'\in \CC'$ we have that $\eta_{C'}\colon g(f(C'))\rightarrow 0$ is a weak equivalence in the induced Waldhausen structure of $(\CC'',\CC_0'')$, and the functor $\D(\CC',\CC_0')\rightarrow \D(\CC,\CC_0)$ is fully faithful and the induced functor from the Verdier quotient $\D(\CC,\CC_0)/\D(\CC',\CC_0')$ to $\D(\CC'',\CC_0'')$ is cofinal, meaning that the functor is fully faithful and any object in the target is a direct summand of an image of an object in the domain.
	
	A map $(\CC,\CC_0)\rightarrow (\DD,\DD_0)$ of Frobenius pairs is a \emph{derived equivalence} if the induced map $\D(\CC,\CC_0)\rightarrow \D(\DD,\DD_0)$ is a triangle equivalence.
\end{definition}

For a deflation-exact category $\EE$ satisfying axiom \ref{R0*}, consider the Frobenius pair $(\Cb(\EE),\Acb(\EE))$ where the conflations of $\Cb(\EE)$ are the degreewise split kernel-cokernel pairs. One readily verifies that $(\Cb(\EE),\Acb(\EE))$ is indeed a Frobenius pair and $\D(\Cb(\EE),\Acb(\EE))\simeq \Db(\EE)$ as triangulated categories.

The following proposition follows directly from theorem \ref{theorem:DerivedEquivalenceExactHull}.
	
	\begin{proposition}
	The map $(\Cb(\EE),\Acb(\EE)) \to (\Cb(\overline{\EE}),\Acb(\overline{\EE}))$ is a derived equivalence.
	\end{proposition}
	
	The following is a corollary of theorem \ref{theorem:MainTheoremPartI}.
	
\begin{proposition}
Let $\AA$ be a deflation-percolating subcategory of $\EE$.  If the natural map $\Db(\AA) \to \DAb(\EE)$ is an equivalence, then the sequence  
	\[(\Cb(\AA),\Acb(\AA)) \rightarrow (\Cb(\EE),\Acb(\EE)) \rightarrow (\Cb(\EE/\AA),\Acb(\EE/\AA))\]
	is an exact sequence of Frobenius pairs.
\end{proposition}
	
	\begin{proof}
			Let $\CAb(\EE)$ be the full subcategory of $\Cb(\EE)$ consisting of all $X^{\bullet}\in \Cb(\EE)$ such that $Q(X^{\bullet})$ is quasi-isomorphic to the zero complex in $\D(\Cb(\EE/\AA),\Acb(\EE/\AA))$. Note that $(\CAb(\EE),\Acb(\EE))$ is a Frobenius pair and $\D(\CAb(\EE),\Acb(\EE))$ is triangle equivalent to the kernel of $\Db(\EE)\rightarrow \Db(\EE/\AA)$. By theorem \ref{theorem:MainTheorem}, there is a triangle equivalence $\D(\CAb(\EE),\Acb(\EE))\simeq \DAb(\EE)$.
	By assumption, the natural map $(\CAb(\EE),\Acb(\EE)) \to (\Cb(\AA),\Acb(\AA))$ is an equivalence.  The statement now follows directly from theorem \ref{theorem:MainTheoremPartI}. 
	\end{proof}

\subsection{A dg enhancement}\label{subsection:dgEnhancement} Let $\CC$ be a dg category, i.e.~a category enriched over the monoidal category of chain complexes.  We write $Z^0(\CC)$ for the \emph{underlying category} (thus, $\Ob Z^0 (\CC) = \Ob \CC$ and the morphisms are the closed degree 0 morphisms), and we write $H^0(\CC)$ for the \emph{homotopy category} (thus, $\Ob H^0 (\CC) = \Ob \CC$ and the morphisms are given by the $0^{\rm{th}}$ homologies of the Hom-complexes).

Let $\Cdg(\EE)$ be the dg category of chain complexes, that is the objects are given by the cochain complexes in $\EE$, and for each $X^\bullet,Y^\bullet \in \Ob \Cdg(\EE)$, we give a complex $\hom(X^\bullet, Y^\bullet)$ by
\[\hom^n(X^\bullet, Y^\bullet) = \prod_{i \in \bZ}\Hom_\EE(X^i, Y^{n+i}).\]
The differential is given by the graded commutator, thus for any $f^\bullet \in \hom^n(X^\bullet, Y^\bullet)$, we set $d(f^\bullet) = d_Y \circ f^\bullet - (-1)^n f^\bullet \circ d_X \in \hom^{n+1}(X^\bullet, Y^\bullet).$

We write $\Cbdg(\EE)$ for the full subcategory of $\Cdg(\EE)$ consisting of those complexes $X^\bullet$ which are quasi-isomorphic to a bounded complex (meaning that $X^\bullet$ is isomorphic to a bounded complex in the derived category $\D(\EE)$), and we write $\Acbdg(\EE)$ for the full dg subcategory of $\Cbdg(\EE)$ consisting of those objects which are homotopic to an acyclic complex, thus $X^\bullet \in \Acbdg(\EE)$ if and only if $X^\bullet \in \AcK(\EE).$   It follows from proposition \ref{proposition:ConeOfAcyclic} that $\Acbdg(\EE)$ is closed under suspensions and cones in $\Cbdg(\EE)$.

\begin{definition}\label{definition:BoundedDG}
The \emph{bounded derived dg category} of $\EE$ is the quotient
\[\Dbdg(\EE) \coloneqq \Cbdg(\EE) / \Acbdg(\EE).\]
\end{definition}

The quotient in the above definition is a quotient of dg categories, as in \cite{Drinfeld04} (based on \cite{Keller99}).  The following proposition is a straightforward application of \cite[main theorem 1.6.2.(i)]{Drinfeld04}.
\begin{proposition}
Let $\EE$ be a deflation-exact category satisfying axiom \ref{R0*}.
\begin{enumerate}
	\item The natural functor $\Db(\EE)\to H^0(\Dbdg(\EE))$ is a triangle equivalence.
	\item The natural quasi-functor $\Dbdg(\EE) \to \Dbdg(\overline{\EE})$ induces the equivalence $\Db(\EE) \to \Db(\overline{\EE})$.
\end{enumerate}
\end{proposition}

\begin{corollary}
Let $\AA \subseteq \EE$ be a deflation-exact subcategory of a deflation-exact category $\EE$.  If the natural map $\Db(\AA) \to \DAb(\EE)$ is a triangle equivalence, then the natural quasi-functor $\Dbdg(\EE)/\Dbdg(\AA) \to \Dbdg(\EE / \AA)$ is a quasi-equivalence.
\end{corollary}

\begin{proof}
This follows from \cite[main theorem 1.6.2.(i)]{Drinfeld04} and theorem \ref{theorem:MainTheorem}.
\end{proof}

\subsection{\texorpdfstring{An $\infty$-categorical enhancement}{An infinity-categorical enhancement}}  We now lift the results of the previous section to the setting of $\infty$-categories.  Following \cite{LurieHA, LurieHTT}, we write $\Ho(\DD)$ for the homotopy category of an $\infty$-category $\DD$.  If $\DD$ is a stable category (in the sense of \cite[definition 1.1.1.9]{LurieHA}), then the homotopy category $\Ho(\DD)$ is naturally endowed with the structure of a triangulated category (see \cite[theorem 1.1.2.14]{LurieHA}).  We write $\StabInfCat$ for the $\infty$-category consisting of (small) stable categories and exact maps.

Let $\CC$ be a dg category.  We write $\Ndg(\CC)$ for the dg nerve in the sense of \cite[construction 1.3.1.6]{LurieHA}.  It is shown in \cite[proposition 1.3.2.10]{LurieHA} that $\Ndg(\Cdg(\EE))$ is a stable $\infty$-category, and it is observed in \cite[remark 1.3.1.11]{LurieHA} that $\Ho(\Ndg(\Cdg(\EE))) \simeq H^0(\Cdg(\EE))$.

We write $\Cbinf(\EE) \coloneqq \Ndg(\Cbdg(\EE))$ and $\Acbinf(\EE) \coloneqq \Ndg(\Acbdg(\EE)).$  It follows from \cite[lemma 1.1.3.3]{LurieHA} that $\Cbinf(\EE)$ and $\Acbinf(\EE)$ are stable $\infty$-categories, and it follows from from \cite[proposition 5.10]{BlumbergGepnerTabuada13} that the map $\Acbinf(\EE) \to \Cbinf(\EE)$ is fully faithful.

\begin{definition}\label{definition:BoundedDerivedInfty}
The \emph{bounded derived $\infty$-category} of $\EE$ is the quotient
\[\Dbinf(\EE) \coloneqq \Cbinf(\EE) / \Acbinf(\EE).\]
\end{definition}

The quotient in definition \ref{definition:BoundedDerivedInfty} is the one from \cite[theorem I.3.3]{NikolausScholze18}, therein called the Dwyer-Kan localization. It is shown in \cite[proposition I.3.5]{NikolausScholze18} that this agrees with the quotient given in \cite{BlumbergGepnerTabuada13}.

\begin{proposition}\label{proposition:InftyHull}
Let $\EE$ be a deflation-exact category satisfying axiom \ref{R0*}.
\begin{enumerate}
	\item The natural functor $\Db(\EE)\to \Ho(\Dbinf(\EE))$ is a triangle equivalence.
	\item The natural functor $\Dbinf(\EE) \to \Dbinf(\overline{\EE})$ is an equivalence.
\end{enumerate}
\end{proposition}

\begin{proof}
The first statement follows from \cite[theorem I.3.3]{NikolausScholze18} or \cite[proposition 5.14]{BlumbergGepnerTabuada13} (with $\kappa = \omega$).  For the second statement, it follows from \cite[corollary 5.11]{BlumbergGepnerTabuada13} that it suffices to show that $\Ho(\Dbinf(\EE)) \to \Ho(\Dbinf(\overline{\EE}))$ is a triangle equivalence.  This has been done in theorem \ref{theorem:DerivedEquivalenceExactHull}.
\end{proof}

\begin{corollary}\label{corollary:MainTheoremForInftyCategories}
Let $\AA \subseteq \EE$ be a deflation-exact subcategory of a deflation-exact category $\EE$ satisfying axiom \ref{R0*}.  If the natural functor $\Db(\AA) \to \DAb(\EE)$ is a triangle equivalence, then the natural map $\Dbinf(\EE)/\Dbinf(\AA) \to \Dbinf(\EE / \AA)$ is an equivalence.
\end{corollary}

We now recall the following notions from \cite{BlumbergGepnerTabuada13}.  Here, we write $\Idem(\DD)$ for the idempotent completion of an $\infty$-category $\DD$ (see \cite[\S5.1.4]{LurieHTT}).

\begin{definition}
\begin{enumerate}
	\item A sequence $\AA \to \BB \to \CC$ in $\StabInfCat$ is called \emph{exact} if the map $\AA \to \BB$ is fully faithfull, the composition $\AA \to \CC$ is trivial, and the induced exact functor $\Idem(\BB / \AA) \to \Idem(\CC)$ is an equivalence.
	\item A functor $K\colon \StabInfCat \to \DD$ with values in a stable presentable $\infty$-category $\DD$ is called a \emph{localizing invariant} if it preserves  filtered colimits, and sends exact sequences to fiber sequences.
\end{enumerate}
\end{definition}

\begin{remark}
The definitions presented here are taken from \cite[definitions 2.2 and 2.3]{KasprowskiWinges19} where it is shown to be equivalent to the corresponding definitions in \cite{BlumbergGepnerTabuada13}.
\end{remark}

For a one-sided exact category $\EE$ and a localizing invariant $K\colon \StabInfCat \to \DD$, we write $K(\EE)$ for $K(\Dbinf(\EE)).$

\begin{corollary}\label{corollary:LocalizingInvariants}
Let $K\colon \StabInfCat \to \DD$ be a localizing invariant with values in a stable $\infty$-category $\DD.$  For any deflation-exact category $\EE$, we have the following.
\begin{enumerate}
	\item There is an equivalence $K(\EE) \to K(\overline{\EE}).$
	\item Under the conditions of corollary \ref{corollary:MainTheoremForInftyCategories}, there is a canonical fiber sequence 
\[K(\AA) \to K(\EE) \to K(\EE / \AA).\]
\end{enumerate}
\end{corollary}

\begin{proof}
The first statement follows from proposition \ref{proposition:InftyHull}.  For the second statement, note that corollary \ref{corollary:MainTheoremForInftyCategories} implies that the sequence $\Dbinf(\AA) \to \Dbinf(\EE) \to \Dbinf(\EE / \AA)$ is exact.
\end{proof}

\begin{remark}
As non-connective $K$-theory is an example of a localizing invariant (see \cite{BlumbergGepnerTabuada13}), we obtain theorem \ref{theorem:KTheoryIntroduction} from corollary \ref{corollary:LocalizingInvariants}.
\end{remark}

\begin{remark}
A similar result to corollary \ref{corollary:LocalizingInvariants} can be found in the setting of dg categories using the results in \S\ref{subsection:dgEnhancement} (following the definitions in \cite{CisinskiTabuada11}).
\end{remark}
\appendix
\section{The weak idempotent completion}\label{appendix:WeakIdempotentCompletion}

In this appendix we recall the constructions of the idempotent completion $\widecheck{\AA}$ and the weak idempotent completion $\widehat{\AA}$ of an additive category $\AA$.  For a deflation-exact category $\EE$, both the idempotent completion $\widecheck{\EE}$ and weak idempotent completion $\widehat{\EE}$ can be endowed with a natural deflation-exact structure extending the original structure.  Finally, we show that $\Db(\EE)\simeq \Db(\widehat{\EE})$ as triangulated categories.

\subsection{Idempotent completions}

An additive category $\AA$ is called \emph{idempotent complete} if every idempotent in $\AA$ has a kernel. We recall the construction of the idempotent completion $\widecheck{\AA}$ of an additive category $\AA$.

\begin{construction}\label{construction:IdempotentCompletion}
	Let $\AA$ be an additive category. We denote by $\widecheck{\AA}$ the category whose objects are pairs $(A,p)$ where $A\in \Ob(\AA)$ and $p\colon A\to A$ is an idempotent in $\AA$. The morphisms are given by $$\Hom_{\widecheck{\AA}}((A,p),(B,q))=q\circ \Hom_{\AA}(A,B)\circ p.$$
The functor $i_{\AA}\colon\AA\to\widecheck{\AA}$ given by mapping $A\in \AA$ to $(A,1)$ is an embedding of additive categories.  The category $\widecheck{\AA}$ is called the \emph{idempotent completion} of $\AA$.
\end{construction}

\begin{remark}\label{remark:KeyPropertyIdempotentCompletion}
For each idempotent $p\colon A\to A$, one has $(A,1)\cong (A,p)\oplus (A,1-p)$ in $\widecheck{\AA}$.
\end{remark}

The following proposition characterizes the idempotent completion in terms of a $2$-universal property.

\begin{proposition}
	The embedding $\AA\hookrightarrow \widecheck{\AA}$ is $2$-universal among additive functors to idempotent complete additive categories:
	\begin{enumerate}
		\item	For any additive functor $F\colon \AA\to \II$ where $\II$ is an idempotent complete additive category, there exists a functor $\widecheck{F}\colon \widecheck{\AA}\to \II$ and a natural isomorphism $F\rightarrow \widecheck{F}i_{\AA}$.
		\item For any additive functor $G\colon \widecheck{\AA}\to \II$ and any natural transformation $\gamma\colon F\Rightarrow Gi_{\AA}$ there is a unique natural transformation $\beta\colon \widecheck{F}\Rightarrow G$ such that $\gamma=\beta*\alpha$.
	\end{enumerate}
	Equivalently, the functor $i_{\AA}$ induces an equivalence of functor categories
	\[\Hom(\widecheck{\AA},\II)\to\Hom(\AA,\II)\]
	for each idempotent complete category $\II$.
\end{proposition}

\begin{proposition}\label{proposition:IdempotentCompletionHasR3}
	Let $\EE$ be a deflation-exact category satisfying axiom \ref{R0*} and let $\widecheck{\EE}$ be its idempotent completion. Declare a sequence in $\widecheck{\EE}$ to be a conflation if it is a direct summand in $(\widecheck{\EE})^{\to\to}$ of a conflation in $(\EE)^{\to\to}$. The conflations in $\widecheck{\EE}$ induce a deflation-exact structure on $\widecheck{\EE}$ satisfying axiom \ref{R3}. Moreover, the embedding $\EE\hookrightarrow \widecheck{\EE}$ is an exact embedding which is $2$-universal among exact functors to idempotent complete deflation-exact categories satisfying axiom \ref{R3}.
\end{proposition}

\begin{proof}
	The reader may verify that the structure above endows $\widecheck{\EE}$ with the structure of a deflation-exact category satisfying axiom \ref{R0*}.
	
	We now show that $\widecheck{\EE}$ satisfies \ref{R3}. Let $i\colon A\rightarrow B$ and $p\colon B\rightarrow C$ be maps in $\widecheck{\EE}$ such that $pi$ is a deflation in $\widecheck{\EE}$. Consider the composition
	\[\xymatrix@C+5pt{
	B\oplus A \ar[r]^{\begin{psmallmatrix}1_B & -i\\0&1_A\end{psmallmatrix}} & B\oplus A \ar[r]^{\begin{psmallmatrix}1_B & 0\\0&pi\end{psmallmatrix}} & B\oplus C \ar[r]^{\begin{psmallmatrix}1_B&0\\p&1_C\end{psmallmatrix}} & B\oplus C\ar[r]^-{\begin{psmallmatrix}0&1_C\end{psmallmatrix}} & C
	}\] The first and third morphism are isomorphisms, the second a is a direct sum of deflations in $\widecheck{\EE}$ and the last is a deflation since $\EE$ satisfies axiom \ref{R0*}. By \ref{R1}, the composition is a deflation in $\widecheck{\EE}$. Note that the composition is simply $\begin{psmallmatrix}p&0\end{psmallmatrix}$. It follows that $p$ is a direct summand of a deflation in $\EE$ as we needed to show.\\
	The $2$-universal property follows from the additive $2$-universal property.
\end{proof}

\begin{corollary}\label{corollary:CharacterizationR3}
Let $\EE$ be an idempotent complete deflation-exact category satisfying \ref{R0*}.  Then $\EE$ satisfies \ref{R3} if and only if the class of conflations is closed under direct summands.
\end{corollary}

\subsection{Weak idempotent completion of additive categories}

We now show the existence of a weak idempotent completion of a (small) additive category satisfying a $2$-universal property.

\begin{definition}
	Let $\AA$ be an additive category. A \emph{retraction} is a map $r\colon B\to C$ such that there is a \emph{section} $s\colon C\to B$ in the sense that $rs=1_C$. A \emph{coretraction} is defined dually.
\end{definition}

\begin{lemma}\label{lemma:EquivalentCharacterizationWeaklyIdempotentComplete}
	Let $\AA$ be an additive category. The following are equivalent:
	\begin{enumerate}
		\item Every retraction has a kernel.
		\item Every coretraction has a cokernel.
	\end{enumerate}
\end{lemma}

\begin{definition}
	An additive category $\AA$ is called \emph{weakly idempotent complete} if the conditions of the previous lemma hold.
\end{definition}

\begin{definition}
	Let $\AA$ be an additive category. Let $\widehat{\AA}$ be the intersection of all weakly idempotent complete additive subcategories of $\widecheck{\AA}$ containing $\AA$. We write $j_{\AA}\colon \AA\hookrightarrow \widehat{\AA}$ for the inclusion.
\end{definition}

\begin{proposition}
	The embedding $j_{\AA}\colon \AA\hookrightarrow \widehat{\AA}$ is $2$-universal among additive functors to weakly idempotent complete additive categories:
	\begin{enumerate}
		\item	For any additive functor $F\colon \AA\to \WW$ where $\WW$ is a weakly idempotent complete additive category, there exists a functor $\widehat{F}\colon \widehat{\AA}\to \WW$ and a natural isomorphism $F\rightarrow \widehat{F}j_{\AA}$.
		\item For any additive functor $G\colon \widehat{\AA}\to \WW$ and any natural transformation $\gamma\colon F\Rightarrow Gj_{\AA}$ there is a unique natural transformation $\beta\colon \widehat{F}\Rightarrow G$ such that $\gamma=\beta*\alpha$.
	\end{enumerate}
	Equivalently, the functor $j_{\AA}$ induces an equivalence $- \circ j_\AA\Hom(\widehat{\AA},\WW)\to\Hom(\AA,\WW)$ for each weakly idempotent complete category $\WW$.
\end{proposition}

\begin{proof}
	It is straightforward to see that the intersection of weakly idempotent complete categories is weakly idempotent complete, it follows that $\widehat{\AA}$ is weakly idempotent complete.
	
	Let $F\colon \AA\to \WW$ be an additive functor to a weakly idempotent complete category $\WW$. Consider the following diagram:
	\[\xymatrix{
		\AA\ar[rr]^F\ar@{^{(}->}[d]_{j_{\AA}}& & \WW\ar[dd]_{i_{\WW}}\\
		\widehat{\AA}\ar@{.>}[r]\ar@{^{(}->}[d] &{\widecheck{F}}^{-1}(\WW)\ar@{^{(}->}[ld]\ar[ru]&\\
		\widecheck{\AA}\ar[rr]^{\widecheck{F}} && \widecheck{\WW}
	}\] Here ${\widecheck{F}}^{-1}(\WW)$ is the full and replete subcategory of $\widecheck{\AA}$ generated by all objects which are mapped to $i_{\WW}(\WW)\subseteq \widecheck{\WW}$ under $\widecheck{F}$.	As additive functors preserve retractions, ${\widecheck{F}}^{-1}(\WW)$ is weakly idempotent complete. By construction, we have that $\widehat{\AA}\subseteq {\widecheck{F}}^{-1}(\WW)$. Denote the composition $\widehat{\AA}\to {\widecheck{F}}^{-1}(\WW)\to \WW$ by $\widehat{F}$. The statement now easily follows from the corresponding $2$-universal property of the idempotent completion.
\end{proof}

In practice it is convenient to have a explicit construction of a weak idempotent completion.

\begin{proposition}
Let $\AA$ be an additive category.  The weak idempotent completion is given by
\[\widehat{\AA} = \left\{(A,1-p)\in \widecheck{\AA}\mid p=sr \mbox{ where $r$ is a retraction in $\AA$ and $s$ the corresponding section}\right\}.\]
\end{proposition}

\begin{proof}
We only need to check that, with this definition, $\widehat{\AA}$ is closed under retracts.  For this, consider a retraction $f\colon (A,1-p) \to (A',1-p')$ in $\widehat{\AA}$ with section $g\colon (A',1-p') \to (A,1-p).$  By assumption, we have that $p = sr$ and $p' = s'r'$ for composable maps $A \xrightarrow{r} B \xrightarrow{s} A$ and $A' \xrightarrow{r'} B' \xrightarrow{s'} A'$ with $rs = 1_B$ and $r's' = 1_{B'}.$  One can verify that
\[ \mbox{$\begin{psmallmatrix} f & r' \\ r & 0 \end{psmallmatrix} \colon A \oplus B' \to A' \oplus B$ and $\begin{psmallmatrix} g & s \\ s' & 0 \end{psmallmatrix} \colon A' \oplus B \to A \oplus B'$} \]
is a retraction together with its corresponding section.  Moreover, the kernel of $f\colon (A,1-p) \to (A',1-p')$ is given by the kernel of the retraction $\begin{psmallmatrix} f & r' \\ r & 0 \end{psmallmatrix} \colon A \oplus B' \to A' \oplus B.$
\end{proof}

\subsection{Weak idempotent completion of deflation-exact categories}

Given a deflation-exact category $\EE$ satisfying axiom \ref{R0*}, we endow the weak idempotent completion $\widehat{\EE}$ with a deflation-exact structure satisfying a $2$-universal property. Moreover, we show that $\Db(\EE)\simeq \Db(\widehat{\EE})$ as triangulated categories.

We start with an obvious characterization of weakly idempotent complete deflation-exact categories satisfying axiom \ref{R0*}.

\begin{proposition}
	Let $\EE$ be a deflation-exact category satisfying axiom \ref{R0*}. The following are equivalent:
	\begin{enumerate}
		\item\label{item:Weak} The category $\EE$ is weakly idempotent complete.
		\item\label{item:Retract} Every retraction is a deflation.
		\item\label{item:Coretract} Every coretraction is an inflation.
	\end{enumerate}
\end{proposition}

\begin{proof}
Assume (\ref{item:Weak}) and let $r\colon B\to C$ be a retraction. By assumption $r$ has a kernel $A\to B$. It is straightforward to show that the sequences $A\to B\to C$ and $A\to A\oplus C\to C$ are isomorphic. By axiom \ref{R0*} (see remark \ref{remark:R0*SplitKernelCokernel}), the latter sequence is a conflation. As conflations are closed under isomorphisms, $r\colon B\to C$ is a deflation.\\
The implication (\ref{item:Retract})$\Rightarrow$(\ref{item:Weak}) is trivial. The equivalence of (\ref{item:Weak}) and (\ref{item:Coretract}) is similar.
\end{proof}

The following lemma is left to the reader.

\begin{lemma}
	Let $\EE$ be a deflation-exact category and let $\FF$ be an extension-closed subcategory (i.e.~for any conflation $A\rightarrowtail B\twoheadrightarrow C$ in $\EE$, we have that $A,C\in \FF$ implies $B\in \FF$). The category $\FF$ inherits a deflation-exact structure from $\EE$ such that the inclusion $\FF\hookrightarrow \EE$ is a fully exact embedding. 
\end{lemma}

In the following lemma, we consider the idempotent completion $\widecheck{\EE}$ of a deflation-exact category $\EE$.  We will use the deflation-exact structure induced on $\widecheck{\EE}$ from proposition \ref{proposition:IdempotentCompletionHasR3}.

\begin{lemma}
	Let $\EE$ be a deflation-exact category satisfying axiom \ref{R0*} and let $\widecheck{\EE}$ be the idempotent completion.  Let $\FF$ be a weakly idempotent complete fully exact subcategory of $\widecheck{\EE}$ containing $\EE$.  Then $\FF$ is an extension-closed subcategory of $\widecheck{\EE}$.
\end{lemma}

\begin{proof}
	As $\FF$ is weakly idempotent complete, remark \ref{remark:KeyPropertyIdempotentCompletion} implies that for each $L\in \widecheck{\EE}$, one has that $L\in \FF$ if and only if there exists an object $L'\in \FF$ such that $L\oplus L'\in \EE$.
	
	Let $A\rightarrowtail B\twoheadrightarrow C$ be a conflation in $\widecheck{\EE}$ and assume that $A,C\in \FF$. By definition, there is a commutative diagram
	\[\xymatrix{
		A \ar@{>->}[r]\ar@{^{(}->}[d]& B\ar@{->>}[r]\ar@{^{(}->}[d] & C\ar@{^{(}->}[d]\\
		X \ar@{>->}[r]& Y\ar@{->>}[r] & Z
	}\] in $\widecheck{\EE}$ whose rows are conflations, the second row belongs to $\EE$ and the vertical maps are corectractions. Since $\widecheck{\EE}$ is a deflation-exact category, \cite[proposition 5.2]{BazzoniCrivei13} or \cite[proposition~3.9]{HenrardvanRoosmalen19} yields a commutative diagram
	\[\xymatrix{
		A\ar@{>->}[r]\ar@{^{(}->}[d] & B\ar@{->>}[r]\ar@{^{(}->}[d] & C\ar@{=}[d]\\
		X\ar@{>->}[r]\ar@{=}[d] & P\ar@{->>}[r]\ar@{^{(}->}[d] & C\ar@{^{(}->}[d]\\
		X\ar@{>->}[r] & Y\ar@{->>}[r] & Z	
	}\] in $\widehat{\EE}$ such that upper left and lower right squares are bicartesian squares. Since $C\in \FF$, there exists an $C'\in \FF$ such that $C\oplus C'\in \EE$, moreover, the square 
	\[\xymatrix{
		P\oplus C'\ar@{->>}[r]\ar[d] & C\oplus C'\ar[d]\\
		Y\ar@{->>}[r] & Z
	}\] is a pullback square in $\widehat{\EE}$ with $Y,Z,(C\oplus C')\in \EE$. Since $\EE$ is a deflation-exact category, the pullback is contained in $\EE$. Hence $P\oplus C'\in \EE$.
	
	As $A \in \FF$, there is an $A' \in \FF$ such that $A \oplus A' \cong X$.  Note that there is a map from the conflation $A' = A' \to 0$ to the conflation $X \inflation P \oplus C' \deflation C \oplus C'$.  Here, we use \ref{R0*} to see that the former is indeed a conflation.  Applying the short five lemma \cite[lemma~5.3]{BazzoniCrivei13} to the commutative diagram (the middle map is determined by the commutativity of the diagram)
	\[\xymatrix{
		A\oplus A'\ar@{>->}[r]\ar[d]_{\rotatebox{90}{$\cong$}} & B\oplus A'\oplus C'\ar@{->>}[r] \ar[d]& C\oplus C'\ar@{=}[d]\\
		X\ar@{>->}[r] & P\oplus C'\ar@{->>}[r] & C\oplus C'
	}\] yields that $B\oplus A'\oplus C'\cong P\oplus C'\in \EE$. Since $(A'\oplus C')\in \FF$, we conclude that $B\in \FF$. Note that we used that $\EE$ satisfies axiom \ref{R0*} to see that the upper row of the previous diagram is a conflation.
\end{proof}

\begin{proposition}\label{proposition:WeakIdempotentCompletionRightExact}
	Let $\EE$ be a deflation-exact category satisfying axiom \ref{R0*}. The weak idempotent completion $\widehat{\EE}$ has a canonical deflation-exact structure satisfying axiom \ref{R3} such that the embedding $j_{\EE}$ is $2$-universal among exact functors to weakly idempotent complete deflation-exact categories satisfying axiom \ref{R3}.
	
	In particular, each conflation in $\widehat{\EE}$ can be realized as a direct summand of a conflation in $\EE \subseteq \widehat{\EE}$.
\end{proposition}

\begin{proof}
	Combining the previous two lemmas, the category $\widehat{\EE}$ inherits a deflation-exact structure. By proposition \ref{proposition:IdempotentCompletionHasR3}, the idempotent completion satisfies axiom \ref{R3}. As the weak idempotent completion $\widehat{\EE}$ is a fully exact subcategory of $\widecheck{\EE}$, axiom \ref{R3} is satisfied. The result follows.
\end{proof}

We are now in a position to prove the following theorem.

\begin{theorem}\label{theorem:WeakidempotentCompleteTriangleEquivalence}
	Let $\EE$ be a deflation-exact category satisfying axiom \ref{R0*}. The embedding $j_{\EE}\colon \EE\hookrightarrow \widehat{\EE}$ lifts to a triangle equivalence $\Db(\EE)\simeq \Db(\widehat{\EE})$.
\end{theorem}

\begin{proof}
	The embedding $j_{\EE}\colon \EE\hookrightarrow \widehat{\EE}$ lifts to a fully faithful triangle functor $F\colon\Kb(\EE)\hookrightarrow \Kb(\widehat{\EE})$. We claim that $F$ is essentially surjective.  To see that this holds, let $X = \ker r \in \widehat{\EE}$ where $r\colon Y \to Z$ is a retraction.  This yields the commutative diagram
	\[\xymatrix{
		\cdots\ar[r] & 0\ar[r]\ar[d] & X\ar[r]\ar[d] & 0\ar[r]\ar[d] & 0\ar[r]\ar[d] &\cdots\\
		\cdots\ar[r] & 0\ar[r] & Y\ar[r] & Z\ar[r] & 0\ar[r] &\cdots
	}\] in $\Kb(\widehat{\EE})$. The cone of this map is a split conflation and hence the rows are homotopy equivalent in $\Kb(\widehat{\EE})$. It follows that $F$ is essentially surjective and hence $\Kb(\EE)\simeq \Kb(\widecheck{\EE})$ as triangulated categories.
	
	Consider now the following diagram:
	\[\xymatrix{
		\Ac^b(\EE)\ar[r]\ar[d] & \Kb(\EE)\ar[r]\ar[d]_F & \Db(\EE)\\
		\Ac^b(\widehat{\EE})\ar[r] & \Kb(\widehat{\EE})\ar[r] & \Db(\widehat{\EE})
	}\]
	We claim that $\Ac^b(\EE)$ and $\Ac^b(\widehat{\EE})$ have the same thick closure when viewed as subcategories of $\Kb(\widehat{\EE})$. Obviously $\Ac^b(\EE)\subseteq \Ac^b(\widehat{\EE})$. On the other hand, any bounded acyclic complex in $\Ac^b(\widehat{\EE})$ can be written as a finite extension of conflations in $\widehat{\EE}$.  By proposition \ref{proposition:WeakIdempotentCompletionRightExact} each such conflation arises as a direct summand of a conflation in $\EE$. The claim follows.
	
Combining both claims we find that $\Db(\EE)$ and $\Db(\widehat{\EE})$ are triangle equivalent categories.	
\end{proof}

\providecommand{\bysame}{\leavevmode\hbox to3em{\hrulefill}\thinspace}
\providecommand{\MR}{\relax\ifhmode\unskip\space\fi MR }
\providecommand{\MRhref}[2]{%
  \href{http://www.ams.org/mathscinet-getitem?mr=#1}{#2}
}
\providecommand{\href}[2]{#2}

\end{document}